\documentclass[11pt]{amsart}

\usepackage{amsmath, amsfonts, amssymb, graphicx, amsthm, mathtools, enumitem, hyperref, setspace, tikz-cd, tabu, dsfont, caption, subcaption, bm}
\usepackage[left=1in, right=1in, top=1in, bottom=1in]{geometry}
\usepackage[ruled,vlined]{algorithm2e}
\RestyleAlgo{ruled}
\usepackage{lipsum}% http://ctan.org/pkg/lipsum
\makeatletter
\g@addto@macro{\endabstract}{\@setabstract}
\newcommand{\authorfootnotes}{\renewcommand\thefootnote{\@fnsymbol\c@footnote}}%
\makeatother
\usepackage{booktabs}
\usepackage{caption}
\usepackage{tablefootnote}
\usepackage{subcaption}
\usepackage[capitalise]{cleveref}
\usepackage{bbm}

\theoremstyle{definition}
\newtheorem{theorem}{Theorem}

\newtheorem{proposition}{Proposition}
\newtheorem{assumption}{Assumption}
\newtheorem{lemma}{Lemma}
\newtheorem{observation}{Observation}

\theoremstyle{definition}

\newtheorem{main problem}[theorem]{Main problem}

\theoremstyle{remark}

\numberwithin{equation}{section}

\begin{document}
\begin{center}
	\LARGE 
	The Nonstationary Newsvendor with (and without) Predictions \par \bigskip
	
	\normalsize
	\authorfootnotes
	Lin An, Andrew A. Li, Benjamin Moseley, and R. Ravi \par 
	Tepper School of Business, Carnegie Mellon University, Pittsburgh, Pennsylvania 15213 \par 
	linan, aali1, moseleyb, ravi@andrew.cmu.edu \par

\end{center}

	\begin{abstract}
	~\\
	\noindent\textbf{\emph{Problem definition:}}
	The classic newsvendor model yields an optimal decision for a ``newsvendor'' selecting a quantity of inventory, under the assumption that the demand is drawn from a known distribution. Motivated by applications such as cloud provisioning and staffing, we consider a setting in which newsvendor-type decisions must be made sequentially, in the face of demand drawn from a stochastic process that is both unknown and nonstationary. All prior work on this problem either (a) assumes that the level of nonstationarity is known, or (b) imposes additional statistical assumptions that enable accurate {\em predictions} of the unknown demand. Our research tackles the Nonstationary Newsvendor without these assumptions, both with and without predictions.
	
	\noindent\textbf{\emph{Methodology/results:}}
	We first, in the setting without predictions, design a policy which we prove (via matching upper and lower bounds) achieves order-optimal regret -- ours is the first policy to accomplish this without being given the level of nonstationarity of the underlying demand. We then, for the first time, introduce a model for generic (i.e. with no statistical assumptions) predictions with arbitrary accuracy, and propose a policy that incorporates these predictions without being given their accuracy. We upper bound the regret of this policy, and show that it matches the best achievable regret had the accuracy of the predictions been known.
	
	\noindent\textbf{\emph{Managerial implications:}}
	Our findings provide valuable insights on inventory management. Managers can make more informed and effective decisions in dynamic environments, reducing costs and enhancing service levels despite uncertain demand patterns. This study advances understanding of sequential decision-making under uncertainty, offering robust methodologies for practical applications with nonstationary demand. We empirically validate our new policy with experiments based on {three} real-world datasets containing thousands of time-series, showing that it succeeds in closing approximately 74\% of the gap between the best approaches based on nonstationarity and predictions alone.
	
	~\\
	\noindent This paper is published at 	 https://pubsonline.informs.org/doi/10.1287/msom.2024.1168.
	\end{abstract}

	\vspace{1em}
	\section{Introduction}
	
	\label{intro}
	
The newsvendor problem is a century-old model (\cite{edgeworth1888mathematical}) that remains fundamental to the practice of operations management. In its original instantiation, a ``newsvendor'' is tasked with selecting a quantity of inventory before observing the demand for that inventory, with the demand itself randomly drawn from a {\em known} distribution. The newsvendor incurs a per-unit underage cost for unmet demand, and a per-unit overage cost for unsold inventory. The objective is to minimize the total expected cost, and %If the demand distribution is known, then 
the classic result is that the optimal inventory level is a certain problem-specific quantile (depending only on the underage and overage costs) of the demand distribution.

This paper is concerned with a modern instantiation of the same model, consisting of a {\em sequence} of newsvendor problems over time, each with {\em unknown} demand distributions that {\em vary} over time. While this version of the problem is arguably ubiquitous in practice today, it may be worth highlighting a few motivating examples:
\begin{itemize}
	\item {\bf Cloud Provisioning:} Consider a website which  provisions computational resources from a commercial cloud provider to serve its web requests. Such provisioning is typically done {\em dynamically}, say on an hourly basis, with the aim of satisfying incoming requests at a sufficiently high service level. Thus, the website faces a single newsvendor problem every hour, with an hourly ``demand'' that can (and does) vary drastically over time.
	\item {\bf Staffing:} A more traditional example is staffing, say for a brick-and-mortar retailer, a call center, or an emergency room. Each day (or even each shift) requires a separate newsvendor problem to be solved, with demand that is highly nonstationary. 
\end{itemize}

Despite its ubiquity, this problem is far from resolved, precisely because the demand (or sequence of demand distributions) is both {\em nonstationary} and {\em unknown} -- indeed, the repeated newsvendor with stationary, but unknown demand was solved by \cite{levi2015data}, and the same setting with known, but nonstationary demand can be treated simply as a sequence of completely separate newsvendor problems. At present, there are by and large two existing approaches to this problem: 
\begin{enumerate}
	\item {\bf Limited Nonstationarity:} One approach is to
	design policies which ``succeed'' under limited nonstationarity, i.e.~the cost incurred by the policy should be parameterized by some carefully-chosen measure of nonstationarity (e.g.~quadratic variation), and nothing else. This approach has proved fruitful across a diverse set of problems ranging from dynamic pricing (\cite{keskin2017chasing}) to multi-armed bandit problems (\cite{besbes2014stochastic}) to stochastic optimization (\cite{besbes2015non}). Most relevant here, the recent work of \cite{keskin2021nonstationary} applies this lens to the newsvendor setting (we will discuss this work in detail momentarily). This approach yields policies with theoretical guarantees that are quite robust -- no assumption on the demand (beyond the limited nonstationarity) is required. 
	However, this is far removed from practice, where the next approach is more common. 
	
	\item {\bf Predictions:} The second approach is to utilize some sort of {\em predictions} of the unknown demand. These predictions can be generated from simple forecasting algorithms for univariate time-series, all the way to state-of-the-art machine learning models that leverage multiple time-series and additional feature information. Therefore these predictions may contain much more information than past demand data points, such as various features/contexts, or even black-box type information that is non-identifiable. In addition to being the de facto approach in practice, the use of predictions in newsvendor-type problems is well-studied, and in fact provable guarantees exist for many specific prediction-based approaches (\cite{ban2019big,huber2019data,oroojlooyjadid2020applying,zhang2023optimal}). All such guarantees rely on (at the very least) the demand and potential features being generated from a known family of stochastic models, so that the framework and tools of statistical learning theory can be applied. Absent these statistical assumptions, it is unclear a priori whether the resulting predictions will be sufficiently accurate to outperform robust policies such as those generated in the previous approach. As a concrete example of this, see \cref{fig:first-example}, which demonstrates on a real set of retail data that prediction accuracy may vary drastically and unexpectedly, even when those predictions are generated according to the same procedure and applied during the same time period.
	
\end{enumerate}

\begin{figure}[h]
	\centering
	\includegraphics[width=4.6in]{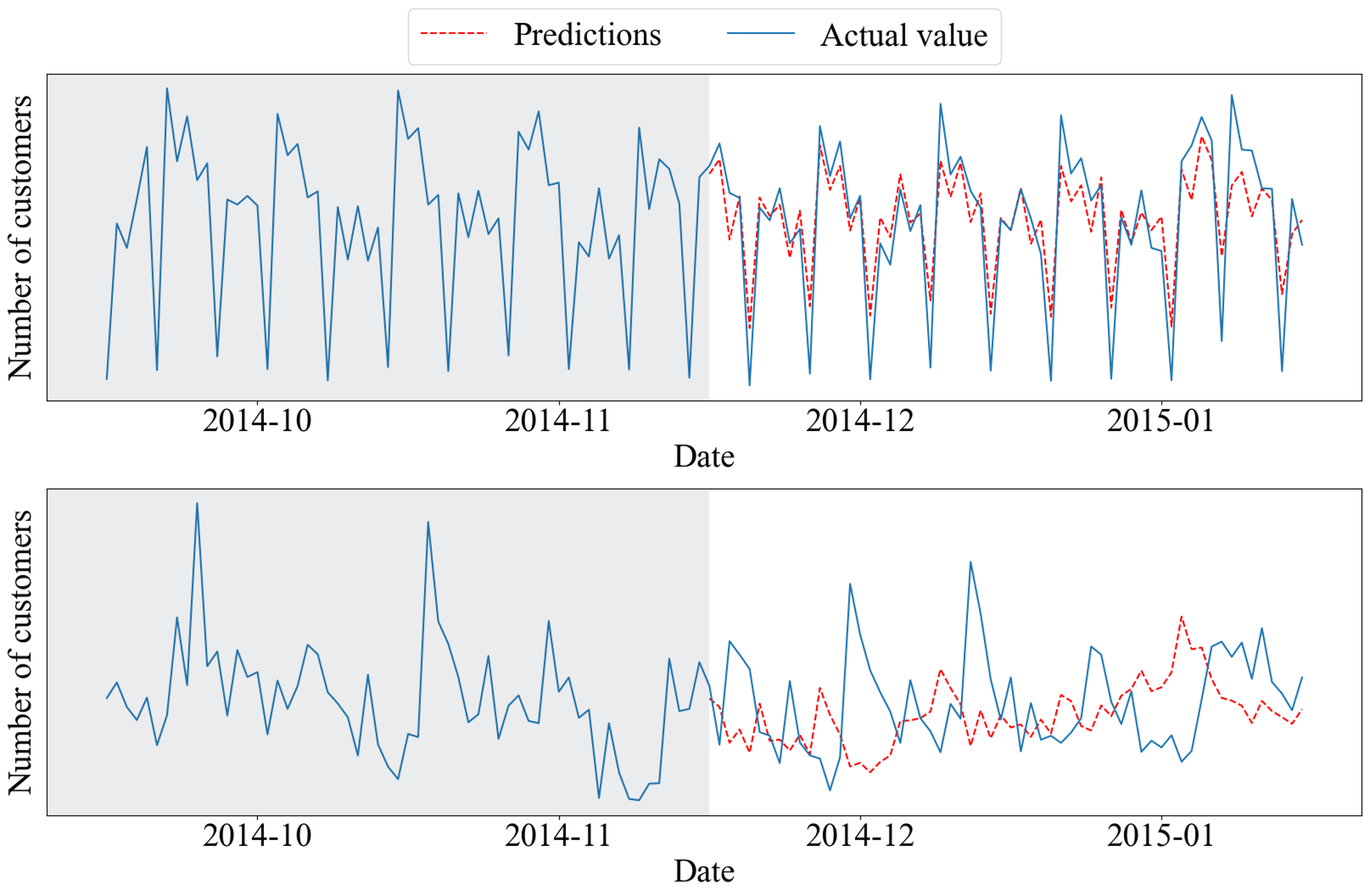}
	\caption{Daily number of customers (in blue), from September 2014 to January 2015, at two different stores in the Rossmann drug store chain. Predictions (in red), starting November  2014, are generated using Exponential Smoothing with the same fitting process. The store in the upper sub-figure has substantially more accurate predictions ($R^2=0.88$) than that of the lower sub-figure ($R^2=0.11$).} 
	\label{fig:first-example}
\end{figure}

To summarize, the repeated newsvendor with unknown, nonstationary demand (which from here on we refer to as the {\em Nonstationary Newsvendor}) admits policies with nontrivial guarantees, which can be made significantly better or worse by following predictions. This suggests the opportunity to design a policy that uses predictions {\em optimally}, in the sense that the predictions are utilized when accurate, and ignored when inaccurate. Ideally, such a policy would run without knowledge of (a) the accuracy of the predictions and (b) the method with which they are generated. {\em This is precisely what we accomplish in this paper.}

%While learning algorithms often give near-accurate predictions, they are inherently imperfect and prone to large errors. Ideally a policy can take advantage of the predictions when they are highly accurate and, otherwise, revert to using the best possible policy when no prediction is given.   

%The goal of this work is to provide the first model that captures predictions in  newsvendor applications.  The model this paper introduces is inspired by recent work on the worst-case online setting which incorporates predictions into the competitive analysis framework \citep{MitzenmacherV22,LattanziLMV20}. 

%\textbf{Ben: I am editing below. DO NOT SUBMIT below to ArXiv}
%\vspace{1em}
\subsection{The Nonstationary Newsvendor, with and without Predictions}
The primary purpose of this paper is to develop a policy that optimally incorporates {\em predictions} (defined in the most generic sense possible) into the {\em Nonstationary Newsvendor} problem. Naturally, a prerequisite to this is a fully-solved model of the Nonstationary Newsvendor without predictions.  At present this prerequisite is only partially satisfied (via the work of \cite{keskin2021nonstationary}), so a nontrivial portion of our contributions will be to fully solve this problem.

Without predictions, the Nonstationary Newsvendor consists of a sequence of newsvendor problems indexed by periods $t \in 1,\ldots,T$, each with \emph{unknown} demand distribution $D_t$. %Luckily, in practice, while demand distributions can charge, the shift is often not arbitrary. 
The level of nonstationarity is characterized via a {\em variation parameter} $v\in[0,1]$, where $v=0$ essentially amounts to stationary demand, and $v=1$ is effectively arbitrary (in a little more detail: a deterministic analogue of quadratic variation is applied to the sequence of means $\{\mathbb{E}[D_1],\ldots,\mathbb{E}[D_T]\}$, and $v\in[0,1]$ is the exponent such that this quantity equals $T^v$).  
Finally, we measure the performance of any policy using {\em regret}, which is the expected difference in the total cost incurred by the policy versus that of an optimal policy that ``knows'' the demand distributions. At minimum we aim to design a policy that achieves {\em sub-linear} (i.e.~$o(T)$) regret, as such a policy would incur a per-period cost that is on average no worse than the optimal, as $T$ grows. We will in fact design policies which achieve {\em order-optimal} regret with respect to the variation parameter $v$.

To this base problem, we introduce the notion of predictions. In each period we receive a prediction $a_t$ of the mean demand $\mu_t = \mathbb{E}[D_t]$ before selecting the order quantity. Our predictions are generic: no assumption is made on how they are generated. We measure the accuracy of the predictions through an {\em accuracy parameter} $a\in[0,1]$, defined such that $\sum_{t=1}^{T} |a_t-\mu_t| = T^a$. 
%before making the ordering decision $q_t$. The  {\em prediction error} in each period is $|a_t-\mu_t|$, and similar to the demand variation, let $a\in[0,1]$ be an accuracy parameter such that $\sum_{t=1}^{T} |a_t-\mu_t|$ is $O(T^a)$.  The value of $a$ measures the quality of the predictions. 
Notice that when $a=0$ the predictions are almost perfect, and when $a=1$ the predictions are effectively useless. We will characterize a precise threshold on $a$ (which depends on $v$) that determines when the predictions should be utilized. Our primary challenge will be to design a policy that makes use of the predictions only when they are sufficiently accurate, and {\em without} having access to $a$. As to the variation parameter $v$, we will separately consider policies which do and do not have access to $v$ -- this distinction will turn out to be {\em the} critical factor in classifying what is and is not achievable.

%The quality of the accuracy is typically unknown. The challenge is for a policy to determine when to leverage the predictions when not to. Ideally, a policy has strong performance when the predictions are high quality. Alternatively, the policy should determine when the predictions are poor quality and revert to a strong policy that does not use the predictions. 

%Ben: found this boring:

%To obtain intuition, we could always `follow' the predictions, meaning set $q_t = a_t$. We refer to this policy simply as the \textit{Follow-the-predictions Policy}, which achieves $O(T^a)$ regret (see Observation \ref{upper bound on regret: predictions-only}). Depending on the value of $a$ (which we always treat as unknown), this policy can be essentially optimal, or grossly sub-optimal.

%To summarize, the problem is to optimally incorporate predictions into our decision-making. The policy may not know the variation parameter $v$ or the accuracy parameter $a$. The question looms, what is the best policy when $a$, $v$ or both are unknown? In particular, 
%is there a policy that obtains nearly the best performance when the predictions are high quality and is robust to poor quality predictions. When the predictions are poor, the policy should realize this and  nearly match the best performance possible when no predictions are given.

%\vspace{1em}
\subsection{Our Contributions}
Our primary contributions can be summarized as follows.

%\vspace{1pt}
%\noindent
\paragraph{\bfseries\sffamily 1. Nonstationary Newsvendor (without predictions):} We completely solve the Nonstationary Newsvendor problem. This consists of first constructing a policy and proving an upper bound on its regret:
%This paper first considers the repeated newsvendor problem where the sequence of demand distributions is nonstationary and no predictions are given. The paper presents a policy called \textit{Shrinking-Time-Window Policy} and proves its performance is the best possible up to lower order terms. The policy achieves a regret on the order of $T^{(3+v)/4}\log^{5/2} T$ even without knowing the variation parameter $v$:
\begin{theorem}[Informal]
	There exists a policy which achieves $\tilde{O}(T^{(3+v)/4})$ regret\footnote{The $\tilde{O}(\cdot)$ notation hides logarithmic factors.} without knowing $v$.
\end{theorem}
We then show that this regret is minimax optimal up to logarithmic factors:
\begin{proposition}[Informal]
	No policy can achieve regret better than $O(T^{(3+v)/4})$, even if $v$ is known.	
\end{proposition}
As alluded to earlier, \cite{keskin2021nonstationary} previously initiated the study of the Nonstationary Newsvendor. Our results are distinct in terms of both modeling and theoretical contributions. We will expound these distinctions more carefully later on.
\begin{itemize}
	\item {\bf Modeling: } The most crucial difference in our model is that we allow both the demand and the set of possible ordering quantities to be {\em discrete}. This is certainly of practical concern (e.g.~physical inventory, employees, and virtual machines are all indivisible units of demand), but moreover we will show that the results of \cite{keskin2021nonstationary} {\em require} both the demand and set of feasible ordering quantities to be continuous. Thus, there is no overlap in our theoretical results.
	\item {\bf Results:} \cite{keskin2021nonstationary} succeed in designing a policy that achieves order-optimal regret, but crucially, their policy requires that the variation parameter $v$ be known. In addition to being concerning from a practical standpoint, this leaves open the theoretical question of what exactly is achievable in settings for which $v$ is unknown. Our results show that the {\em same} regret can be achieved without knowing $v$.
\end{itemize}

%\vspace{1em}
\noindent
\paragraph{\bfseries\sffamily 2. Nonstationary Newsvendor with Predictions:}
%The main results of the paper are policies for the Nonstationary Newsvendor with Predictions model. 
We construct a policy that {\em optimally} leverages predictions, i.e.~it is robust to unknown prediction accuracy. To be precise, the previous contribution offers a policy that achieves $\tilde{O}(T^{(3+v)/4})$ regret, and predictions yield a simple policy that achieves $O(T^a)$ regret, so we would expect that the best possible regret is the minimum of these two quantities. We show this formally:
\begin{proposition}[Informal]
	No policy can achieve regret better than $O(T^{\min\{(3+v)/4, a\}})$, even if $v$ and $a$ are known.
	%	Fix any (known) variation parameter $v\in[0,1]$ and any (known) accuracy parameter $a\in[0,1]$, there exists some universal constant $c\in(0,\infty)$ such that any policy that uses both the past demand observations and the predictions have regret at least $cT^{\min\{(3+v)/4, a\}}$.
\end{proposition}
Our main algorithmic contribution is a policy which achieves this lower bound (up to log factors) {\em without} knowing the prediction accuracy:
% The policy is able to take high-quality predictions in terms of its relationship to the variation.  If the predictions are poor, the policy maintains the best guarantees for the Nonstationary Newsvendor when no predictions are given.  
%Specifically, the accuracy-robust policy achieves regret on the order of $\sqrt{\log T}\cdot T^{\min\{(3+v)/4, a\}}$ even without knowing the predictions accuracy beforehand:
\begin{theorem}[Informal]
	There exists a policy which achieves regret $\tilde{O}(T^{\min\{(3+v)/4, a\}})$, knowing $v$, and without knowing $a$.
\end{theorem}
\addtocounter{theorem}{-2}
%Similar to the above, we then show that this regret is minimax optimal up to a $\sqrt{\log T}$ factor. %\textbf{Lin: Do we want to delete the following statement?} 
%Taken together, these two results establish that our policy perfectly incorporates predictions with unknown quality into the decision-making process.
Finally, since our policy relies on knowledge of the variation parameter $v$, the remaining question is whether the same regret is achievable if both $v$ and $a$ are unknown. We show that in fact predictions {\em cannot} be incorporated in any meaningful way in this case:
%PERP uses knowledge of the variation parameter $v$, but not the prediction parameter $a$. The following proposition demonstrates this is essentially the best any algorithm can do, as there is a strong lower bound when $v$ and $a$ are both unknown. 

\begin{proposition}[Informal]
	If $v$ and $a$ are unknown, then no policy can achieve regret better than $O(T^{\max\{(3+v)/4,a\}})$ for all $v,a \in [0,1]$. 
	%For any policy that does not depend on $v$ or $a$, there exists a problem instance such that $a \ne (3+v)/4$, and the policy incurs regret at least $cT^{\max\{(3+v)/4,a\}}$ on the instance, where $c>0$ is a universal constant.
\end{proposition}

\addtocounter{proposition}{-3}

\noindent Our theoretical results are summarized in the \cref{results summary}. Each entry has a corresponding policy that achieves the stated regret, along with a matching lower bound.

\begin{table}[h]
	\centering
	\begin{tabular}{@{}lcc@{}}
		\toprule
		& \qquad Without predictions \qquad & \qquad With predictions of unknown accuracy \qquad\\
		\hline
		Known variation & $\qquad \tilde{O}(T^{(3+v)/4})$ & $\tilde{O}(T^{\min\{(3+v)/4,a\}})$ \\
		Unknown variation &\qquad $\tilde{O}(T^{(3+v)/4})$ & $\tilde{O}(T^{\max\{(3+v)/4,a\}})$ %\tablefootnote{As stated in Proposition \ref{lower bound on regret unknown parameters}, for some $v$ and $a$ we have this as a lower bound. The Shrinking-Time-Window Policy and the Prediction Policy both trivially achieve this regret.} 
		\\
		\bottomrule
	\end{tabular}
	\caption{Summary of main theoretical results. Each entry has a corresponding policy that achieves the stated regret, along with a matching lower bound.}
	\label{results summary}
\end{table}

%\vspace{1pt}
\noindent
\paragraph{\bfseries\sffamily 3. Empirical Results:}
Finally, we demonstrate the practical value of our model (namely the Nonstationary Newsvendor with Predictions) and our policy via empirical results on three real-world datasets that span our motivating applications above: daily web traffic for Wikipedia.com (of various languages), daily foot traffic across the Rossmann store chain, and daily visitors at a certain Japanese restaurant. These datasets together contain over one thousand individual time-series on which we generate predictions of varying quality, using four different popular forecasting and machine learning algorithms. 
We apply our policy, and compare its performance against the two most-natural baseline policies: our optimal policy without predictions, and the simple policy which always utilizes the predictions (these correspond to the two ``existing approaches'' described previously). A snapshot of our results, for the Rossmann stores depicted in \cref{fig:first-example}, is given in \cref{tab:first-example}.

	\begin{table}
		\centering
		\begin{tabular}{@{}lrrr@{}} \toprule
			& {\bf \small No Prediction} & {\bf \small Prediction} & {\bf \small Our Policy} \\ \midrule
			{\bf \small Upper store} & \$28,303 & \$14,454 & \$14,454 \\
			{\bf \small Lower store} & \$23,460 & \$35,600 & \$23,899 \\ \bottomrule
		\end{tabular}
		\caption{Continuation of \cref{fig:first-example}: costs incurred by an optimal policy which makes no use of predictions, a policy which relies entirely on predictions, and our policy.}
		\label{tab:first-example}
	\end{table}

More generally, on any given experimental instance (i.e.~a time-series and a set of predictions), the minimum (maximum) of the costs incurred by these two baselines can be viewed as the best (worst) we can hope for. Thus we measure performance in terms of the proportion of the gap between these two costs incurred by our policy, so if this ``optimality gap'' is close to 0, our policy performs almost as good as the better one of the two baselines. Note that {\em randomly} selecting between the two baseline policies yields an (expected) optimality gap of 0.5. We find that in the Rossmann dataset, the average optimality gap is 0.26 when the predictions are accurate, and 0.28 when the predictions are inaccurate. In the Wikipedia dataset, the average optimality gap is 0.40 when the predictions are accurate, and 0.07 when the predictions are inaccurate. In the Restaurant dataset, the average optimality gap is 0.10 when the predictions are accurate, and 0.39 when the predictions are inaccurate. This demonstrates that our policy performs well, irrespective of the quality of the predictions. 

%\vspace{1em}
\subsection{Literature Review}

The earliest works on the newsvendor model assume that the demand distribution is fully known \cite{arrow1958studies,scarf1960optimality}. This assumption has then been relaxed, and we can divide the approaches in which the demand distribution is unknown into parametric and nonparametric ones. One of the most popular parametric approaches is the Bayesian approach, where there is a prior belief in parameters of the demand distribution, and such belief is updated based on observations that are collected over time.  \cite{scarf1959bayes} first applied the Bayesian approach to inventory models, and later this is studied in many works \cite{karlin1960dynamic,iglehart1964dynamic,azoury1985bayes,lovejoy1990myopic}.  \cite{liyanage2005practical} introduced another parametric approach called operational statistics which, unlike the Bayesian approach, does not assume any prior knowledge on the parameter values. Instead it uses past demand observations to directly estimate the optimal ordering quantity.

Nonparametric approaches have been developed in recent years. The first example of a nonparametric approach is the SAA method, first proposed by \cite{kleywegt2002sample} and \cite{shapiro2003monte}. \cite{levi2007provably} applied SAA to the newsvendor problem by using samples to approximate the optimal ordering quantity, and  \cite{levi2015data} improve significantly upon the bounds of  \cite{levi2007provably} for the same problem. Other non-parametric approaches include stochastic gradient descent algorithms \cite{burnetas2000adaptive,kunnumkal2008using,huh2009nonparametric} and the concave adaptive value estimation (CAVE) method \cite{godfrey2001adaptive,powell2004learning}. With the development of machine learning, \cite{ban2019big} and \cite{oroojlooyjadid2020applying} propose machine learning/deep learning algorithms using demand features and historical data to solve the newsvendor problem.

All the previous studies consider the newsvendor in a static environment where the demand distribution is the same over time. However, in reality the demand distribution is often nonstationary. There are two common practices to resolve this issue. The first is to model the nonstationarity and utilize past demand observations according to the model. One common way is to model the nonstationarity as a Markov chain. For example, 
\cite{treharne2002adaptive} applied this idea to inventory management and \cite{aviv2005partially} and \cite{chen2019coordinating} applied this idea to revenue management. Another approach is to bound the nonstationarity via a variation budget, which has been applied to stochastic optimization \cite{besbes2015non}, dynamic pricing \cite{keskin2017chasing}, multi-armed bandit  \cite{besbes2014stochastic}, newsvendor problem \cite{keskin2021nonstationary}, among others. Some of these works are applicable in the sense that our problem can be mapped to their settings (e.g.~multi-armed bandit such as \cite{besbes2015non,karnin2016multi,luo2018efficient,cheung2022hedging}), but these connections do not appear to be fruitful.  In particular, the multi-armed bandit papers cited above typically consider a {\em limited-feedback} setting rather than the {\em full-feedback} setting explored in this work. Related to feedback, while our study provides a complete characterization of the regret behavior for the nonstationary newsvendor problem with {\em uncensored} demand, practical applications often involve {\em censored} demand. The nonstationary newsvendor problem under censored demand is an interesting direction for future research. 

Beyond the bandit literature, it is worth mentioning recent work on online convex optimization (OCO) with limited nonstationarity. When the level of nonstationarity is {\em known}, the standard first-order OCO algorithms can be modified with carefully chosen restarts and updating rules (\cite{besbes2015non},\cite{yang2016tracking}, \cite{chen2019nonstationary}). There are also recent works that concern {\em unknown} nonstationarity, such as \cite{zhang2018adaptive}, \cite{baby2019online}, \cite{bai2022adapting}, and \cite{huang2023stability}. Finally, as mentioned before, \cite{keskin2021nonstationary} is particularly relevant, so we delay a careful comparison to \cref{formulation,Section newsvendor without predictions}. 

The second common practice is to use predictions on the demand distribution of each time period. Predictions can often be obtained e.g.~via machine learning, and a recent line of work looks to help decision-making by incorporating predictions. This framework has been applied to many online optimization problems such as revenue optimization \cite{munoz2017revenue, balseiro2022single}, caching \cite{lykouris2021competitive,rohatgi2020near}, online scheduling \cite{lattanzi2020online}, and the secretary problem \cite{dutting2021secretaries}. In this paper we will combine the nonstationarity framework and the prediction framework on the newsvendor problem.

Finally, most previous works involving algorithms with predictions analyzed algorithms' performances using competitive analysis (e.g. \cite{mahdian2012online,antoniadis2020secretary,balseiro2022single,jin2022online}) and obtained optimal \textit{consistency-robustness} trade-offs, where \textit{consistency} is an algorithm's competitive ratio when the prediction is accurate, and \textit{robustness} is the competitive ratio regardless of the prediction's accuracy. However, competitive ratio transfers to a regret bound that is linear in $T$. In contrast, we do regret analysis under this framework and design an algorithm that has near-optimal worst-case regret without knowing the prediction quality. Other papers with regret analyses under the prediction model include \cite{munoz2017revenue} (revenue optimization in auctions), \cite{hao2023leveraging} (Thompson sampling), \cite{hu2024constrained} (constrained online two-stage stochastic optimization), and \cite{an2024best} (online resource allocation).

\section{Model: The Nonstationary Newsvendor (without Predictions)}\label{formulation}

We begin this section with a formal description of the {\bf Nonstationary Newsvendor}, along with a comparison to the problem of the same name from \cite{keskin2021nonstationary}. Consider a sequence of newsvendor problems over $T$ time periods labeled $t=1,\dots, T$. At the beginning of each time period $t$, the decision-maker selects a quantity $q_t\in Q$, where $Q$ is a fixed subset of $\mathbb{R}^+$ bounded above by a quantity we denote as $Q_{\max}$.\footnote{All of our results carry through if $Q$ is allowed to depend on $t$.} Then the period's demand $d_t$ is drawn from an (unknown) demand distribution $D_t$, which  depends on the time period $t$. These demand distributions are  independent over time. Finally a cost is incurred -- specifically, there is a (known) per-unit underage cost $b_t \in [0,b_{\max}]$ %for each unit of unmet demand, 
and a (known) per-unit overage cost $h_t\in [0,h_{\max}]$, so that the total cost is equal to $$b_t(d_t-q_t)^{+}+h_t(q_t-d_t)^{+},$$ where $x^{+} = \max \{0, x\}$. The decision-maker observes the realized demand $d_t$,\footnote{The demand is not censored here, as is the case in all of the motivating examples in the introduction. The censored version of our problem is an interesting, but separate subject.} and thus the cost.
%We define the {\em expected cost function} (which the decision-maker does not observe) as $$C_t(q_t) = \mathbb{E}\left[b_t(D_t-q_t)^{+}+h_t(q_t-D_t)^{+}\right],$$ where the expectation is taken with respect to the stochastic demand $D_t$. 
%We will assume that $b_t\in[\underline{b},\overline{b}]$ and  $h_t\in[\underline{h},\overline{h}]$ for some $0<\underline{b}<\overline{b}<\infty$, $0<\underline{h}<\overline{h}<\infty$ over all time periods $t$. {\color{red} Is this notation necessary?}
%assume we know $b_1,\dots, b_T$ and $h_1,\dots, h_T$ in advance.
% for each unsold product unit, 
%where  That is,
%where $x^{+} = \max \{0, x\}$.  
Note that requiring $q_t \in Q$ does not impose any restriction on {\em modeling}, since $Q$ could simply be selected to be $\mathbb{R}$ (as in much of the literature). In fact, introducing $Q$ allows for modeling important practical concerns such as batched inventory or even simply the integrality of physical items. As we will discuss momentarily, this is a non-trivial concern insofar as theoretical guarantees are concerned. %{\color{red} Do we still $Q$, $b_t$, and $h_t$ to be bounded?}

To complete our description of the Nonstationary Newsvendor, we will need to (a) impose a few assumptions on the demand distributions, and then (b) describe how ``nonstationarity'' is quantified. These are, respectively, the subjects of the following two subsections.

%\vspace{1em}
\subsection{Demand Distributions} \label{sec:demand-distributions}
We will assume that the demand distributions come from a known, parameterized family of distributions $\mathcal{D}$:
\begin{assumption} \label{assu:demand}
	Every demand distribution $D_t$ comes from a family of distributions $\mathcal{D}$ satisfying the following:
	\begin{itemize}
		\item[(a)] $\mathcal{D} = \{\mathcal{D}_\mu: \mu \in [\mu_{\min},\mu_{\max}]\}$, that is $\mathcal{D}$ is parameterized by a scalar $\mu$ taking values in some bounded interval. %This parameterization  
		\item[(b)] Each distribution $\mathcal{D}_\mu \in \mathcal{D}$ is sub-Gaussian.\footnote{A random variable $X$ is {\em sub-Gaussian} with {\em sub-Gaussian norm} $\|X\|_{\psi_2}$ if $\mathbb{P}\left(|X| > x \right) \le 2 \exp\left(  -x^2/\|X\|_{\psi_2}^2\right)$ for all $x \ge 0.$ For sub-Gaussian variables, we have $\mathbb{E}[|X|] \le 3\|X\|_{\psi_2}$.}
		%    Precisely, this means that there exists some $C>0$ such that each distribution $\mathcal{D}_\mu\in \mathcal{D}$ satisfies $\Pr(|\mathcal{D}_\mu| \ge x) \le 2\exp(-x^2/C)$ for all $x \ge 0$.} 
\end{itemize}

% The set of all $\mu$'s of distributions in $\mathcal{D}$ is $[\underline{\mu},\overline{\mu}]$ where $0\leq \underline{\mu}<\overline{\mu}<\infty$, and each $\mu\in [\underline{\mu},\overline{\mu}]$ corresponds to a unique distribution in $\mathcal{D}$, i.e., $\mu$ completely characterizes $D$.
\end{assumption}
Assumption \ref{assu:demand} is fairly minimal. Parsing it in reverse: the sub-Gaussianity in part (b) allows for many commonly-used variables, such as the Gaussian distribution and any bounded random variable, while letting us eventually apply Hoeffding-type concentration bounds. Part (a) is particularly minimal at the moment, as $\mu$ represents an arbitrary parameterization of $\mathcal{D}$, but will become meaningful when combined with Assumption \ref{assu:lipschitz}. The choice of the symbol ``$\mu$'' might suggest that $\mu$ represents the mean of $\mathcal{D}_\mu$, and indeed this is what we will assume from here on. But it should be emphasized that our taking $\mu = \mathbb{E}[\mathcal{D}_\mu]$ is strictly for notational convenience (because we will frequently need to refer to the means of these distributions): if $\mu$ were any other parameterization of $\mathcal{D}$, we could simply define a mapping from $\mu$ to the mean values.

Now define $C(\mu,b,h,q)$ to be the expected newsvendor cost when selecting quantity $q \in Q$, given underage/overage costs $b$ and $h$, and demand distribution $\mathcal{D}_\mu$:
$$C(\mu,b,h,q) = \mathbb{E}_{d\sim \mathcal{D}_\mu}[b(d-q)^+ + h(q - d)^+].$$
The critical assumption, with respect to the parameterization in Assumption \ref{assu:demand}(a), is that the expected cost is well-behaved as a function of $\mu$:
\begin{assumption}\label{assu:lipschitz}
For every $b \in [0,b_{\max}]$, $h \in [0,h_{\max}]$, and $q \in Q$, the function $C(\cdot,b,h,q)$ is Lipschitz on its domain $[\mu_{\min},\mu_{\max}]$, i.e.~there exists $\ell \in \mathbb{R}^+$ such that for every $\mu_1,\mu_2 \in [\mu_{\min},\mu_{\max}]$, we have 
$$|C(\mu_1,b,h,q)-C(\mu_2,b,h,q)|\leq \ell|\mu_1-\mu_2|.$$
\end{assumption}
Note that in the above description, the Lipschitz constant $\ell$ may depend on $b,h,$ and $q$, but by continuity, there exists a single $\ell$ so that the above holds for all $b,h,q$ simultaneously.

Some useful examples of families $\mathcal{D}$ satisfying Assumptions \ref{assu:demand} and \ref{assu:lipschitz} are the following:\footnote{Unfortunately, \cref{assu:lipschitz} is not guaranteed to hold. For example, for the family of distributions $$\mathcal{D}_\mu \sim \begin{cases}
	\mu, \;\; \mu \in [0,1]\\
	\mu + \mathrm{Bernoulli}(0.5) - 0.5, \;\; \mu \in (1,2]
\end{cases},$$ the function $C(\mu,1,1,1)$ is discontinuous (and thus not Lipschitz) at $\mu = 1$.
}
\begin{enumerate}
\item $\mathcal{D}_\mu \sim \mathcal{N}(\mu,\sigma^2)$, the family of normal distributions with fixed variance $\sigma^2$. In this case, $\ell = O(\sigma(b_{\max} + h_{\max}))$. A relaxation is that the variances may vary (continuously) with $\mu$. 
\item $\mathcal{D}_\mu = \mu + \epsilon$, where $\epsilon$ is any mean-zero, sub-Gaussian variable. 
\item The Poisson distribution is frequently used to model demand (since arrivals are often modeled as a Poisson process). While the Poisson distribution is {\em not} sub-Gaussian, any reasonable truncation satisfies our assumptions. For example, $\mathcal{D}_\mu \sim \min\{\mathrm{Poisson}(\mu),K\mu\}$, for some constant $K$. Here, $K$ can be taken to be large enough so that the truncation happens with small probability (in fact, this probability is $O(e^{-K\mu})$).
\end{enumerate}

%For example, if $\mathcal{D}\sim\mathcal{N}(\mu,\sigma^2)$ is the class of normal distributions with fixed variance $\sigma^2$, then $\mathcal{D}$ is $l$-Lipchitz in $\mu$ where $l=\max\{b_{\max},h_{\max}\}$; if $\mathcal{D}\sim \text{Bernoulli}(p)$ is the class of Bernoulli distributions, then $\mathcal{D}$ is $l$-Lipchitz in $\mu$ where $l=\max\{b_{\max},h_{\max}\}$.

%In fact the results in our paper hold for any notion of $\mu$ that measures the centrality of a distribution, such as mean, median, etc., as long as it completely characterizes $D$. Throughout the paper we naturally set $\mu$ to be the mean.
To understand the reasoning behind Assumption \ref{assu:lipschitz}, consider the problem faced at some time $t$. The optimal choice for the decision-maker here is
\begin{equation} \label{eqn:qtstar}
q_t^* \in \text{argmin}_{q\in Q}C_t(\mu_t,q),
\end{equation}
where $\mu_t$ is the mean of $D_t$ (i.e.~$D_t \sim \mathcal{D}_{\mu_t}$), and $C_t(\mu,q) = C(\mu,b_t,h_t,q)$ to simplify the notation.\footnote{As a sanity check, the classical result for the newsvendor problem (\cite{arrow1958studies,scarf1960optimality}) states that if $Q=\mathbb{R}$, then $q^*_t$ is the $b_t/(b_t+h_t)$-th quantile of $D_t$.} Since $D_t$ is unknown, it is likely that some $q_t \ne q_t^*$ will ultimately be selected, and we could measure the sub-optimality of this decision (i.e.~regret, to be defined soon): $C_t(\mu_t,q_t) - C_t(\mu_t,q_t^*)$. It would be natural then to try to characterize this suboptimality as a function of $|q_t - q_t^*|$, but in fact all of the algorithms we will consider ``work'' by making an estimate $\hat{\mu}_t$ of $\mu_t$, and then selecting $\hat{q}_t \in \text{argmin}_{q\in Q}C_t(\hat{\mu}_t,q)$. So motivated, the purpose of Assumption \ref{assu:lipschitz} is to allow us to ``translate'' error in our estimate of $\mu_t$ to (excess) costs. The following structural lemma makes this precise, and will be used throughout the paper.

%Since $D$ is specified by $\mu$, we use the notation $C(\mu_t,q_t)=C_t(q_t)$.  Let $q_t^*$ denote the optimal ordering quantity for the demand distribution $D_t$, which we assume can always be achieved. For simplicity, we assume this to be unique (otherwise, we choose an arbitrary optimal ordering quantity), i.e., 
%$$q_t^*=\text{argmin}_{q_t\in Q}C(\mu_t,q_t).$$

%To better understand the usefulness of $\mathcal{D}$ being $l$-Lipchitz in $\mu$, we provide 

\begin{lemma} \label{l-Lipchitz}
Fix any $b$ and $h$ (we will suppress them from the notation).
For any $D_{\mu_1}, D_{\mu_2}\in \mathcal{D}$, let $q_1^* \in \text{argmin}_{q\in Q} C(\mu_1,q)$ and $q_2^* \in \text{argmin}_{q\in Q} C(\mu_2,q)$. Then we have
$$C(\mu_1,q_2^*)-C(\mu_1,q_1^*)\leq 2\ell|\mu_1-\mu_2|.$$
\end{lemma}

% \proof{Proof.} By Assumption \ref{assu:lipschitz}, we have
% \begin{eqnarray*}
% 	C(\mu_1,q_2^*)-C(\mu_1,q_1^*) &=& C(\mu_1,q_2^*)-C(\mu_2,q_1^*)+C(\mu_2,q_1^*)-C(\mu_1,q_1^*)\\
% 	&\overset{(a)}{\leq}& C(\mu_1,q_2^*)-C(\mu_2,q_1^*)+\ell|\mu_1-\mu_2|\\
% 	&\overset{(b)}{\leq}& C(\mu_1,q_2^*)-C(\mu_2,q_2^*)+\ell|\mu_1-\mu_2|\\
% 	&\overset{(c)}{\leq}& 2\ell|\mu_1-\mu_2|.
% \end{eqnarray*}
% Here $(a)$ and $(c)$ follow from $C(\mu,q)$ being $\ell$-Lipchitz in $\mu$, and $(b)$ uses the definition of $q_2^*$. \Halmos
% \endproof
\noindent Lemma \ref{l-Lipchitz} states that estimation error of the mean $\mu_t$ translates {\em linearly} to excess cost. The proof of Lemma \ref{l-Lipchitz} appears in Appendix \ref{app:A}.

%, if the true demand distribution is $D_1$ but we estimate it to be $D_2$ and order $q_2^*$ instead, then the difference between the cost (and hence the regret, which we will define again later) we incurred can be linearly bounded in our estimation error of the mean $|\mu_1-\mu_2|$. The above linear bound holds for both discrete and continuous demand distributions. 

%Let $\mu_t$ be the mean of $D_t$ and we write $D_t=\mu_t+\epsilon_t$. The noise $\epsilon_t$ is parameterized by $t$ and does not have to be identically distributed over all time periods. This model ensures each time period's noise can be different depending on the time period's demand distribution.

\noindent
\paragraph{\bfseries\sffamily Aside: Comparison to \cite{keskin2021nonstationary}:} The final component in describing the Nonstationary Newsvendor is defining a proper quantification of nonstationarity. Before doing so, we delineate the {\em modeling} differences between our Nonstationary Newsvendor and that of \cite{keskin2021nonstationary}. There are two primary differences:\footnote{Other minor differences: \cite{keskin2021nonstationary} require the demand distribution to be bounded, and this assumption is easily relaxed.}
\begin{enumerate}    
\item The demand distributions $D_t$ in \cite{keskin2021nonstationary} are assumed to be  of the form $D_t=\mu_t+\epsilon_t$, where $\mu_t$ is the mean of $D_t$ that drifts across time and $\epsilon_t$ is the noise distribution that is i.i.d., continuous, and bounded. Effectively, the demand distribution fall into a {\em non-parametric} family of distributions with the same ``shape''. In contrast, our demand distributions fall into a {\em parametric} family of distributions, though not necessarily of the same ``shape''.
\item Our set of allowed order quantities $Q$ is bounded, but otherwise arbitrary. In particular, it need not contain the optimal unconstrained order quantity $\text{argmin}_{q\in \mathbb{R}}C(\mu,q)$ for each $\mu$ (or any $\mu$, for that matter). \cite{keskin2021nonstationary} assume $Q = \mathbb{R}^+$.\footnote{While not stated explicitly, the results in \cite{keskin2021nonstationary} only require $Q$ to contain points arbitrarily close to every optimal unconstrained order quantity.}
\end{enumerate}
Besides the practical reasons why discrete quantities arise in practice (non-divisible items, batched inventory, etc.), the primary consequence of either of the two differences above is that they preclude a critical lemma used in \cite{keskin2021nonstationary} (and in fact by \cite{levi2007provably}) which states that $C(\mu_1,q_2^*)-C(\mu_1,q_1^*)$ (as defined in our Lemma \ref{l-Lipchitz}) scales as $(q_1^* - q_2^*)^2$. This scaling does {\em not} necessarily hold when either the demand distribution or $Q$ is discrete. These relaxations in assumptions yield different lower bounds in the worst-case regret from \cite{keskin2021nonstationary}, which we will discuss in detail later.

%\vspace{1em}
\subsection{Demand Variation}
Just as in \cite{keskin2021nonstationary} (and \cite{keskin2017chasing} before that), we measure the level of nonstationarity via a deterministic analogue of quadratic variation for the sequence of means $\mu_1,\ldots,\mu_T$. Specifically, define a {\em partition} of the time horizon $\{1,\dots, T\}$ to be any subset of time periods $\{t_0,\dots, t_K\}$ where $1\leq t_0<\cdots < t_K\leq T$. Here the subset can have any size between 1 and $T$, i.e. $0\leq K\leq T-1$. Then for any sequence of means $\bm{\mu}=\{\mu_1,\dots,\mu_T\}$, its \textbf{demand variation} is 
\begin{equation}\label{eq:demand-variation}
	V_{\bm\mu}=\max_{0\le K \le T-1}\max_{\{t_0,\dots, t_K\}\in\mathcal{P}}\left\{\sum_{k=1}^{K}\left|\mu_{t_k}-\mu_{t_{k-1}}\right|^2\right\},
\end{equation}
where $\mathcal{P}$ is the set of all partitions.

To motivate the use of partitions in the definition of $V_{\bm{\mu}}$, it is worth contrasting with a measure that may feel more natural, namely the sum of squared differences ({\em SSD}) between consecutive terms, $\sum_{t=2}^{T}\left(\mu_{t}-\mu_{t-1}\right)^2$, which corresponds to taking the densest possible partition $\{1,2,\ldots,T\}$. The maximum in the definition of $V_{\bm{\mu}}$ is {\em not} necessarily achieved by selecting the densest possible partition, but rather by setting $t_0,\dots, t_K$ to be the periods when the sequence $\mu_1,\dots,\mu_T$ changes direction. Thus, the demand variation penalizes {\em trends}, or consecutive increases/decreases, more so than the SSD. For example, the mean sequences $\bm{\mu_1}=\{1,2,3,4,5\}$ and $\bm{\mu_2}=\{1,0,1,0,1\}$, respective variations $V_{\bm{\mu_1}}=(5-1)^2=16$ and $V_{\bm{\mu_2}}=1^2+1^2+1^2+1^2+1^2=5$, despite having identical SSDs.

All of our theoretical guarantees (upper and lower bounds) will be parameterized by $V_{\bm{\mu}}$. This quantity of course depends on $T$, and so it is natural to allow $V_{\bm{\mu}}$ to grow $T$. It will turn out that the most natural parameterization of this growth is via what we will simply call the {\bf variation parameter} $v\in [0,1]$, such that $V_{\bm{\mu}} = BT^v$, where $B$ is some constant (which we take to be equal to one from here on).
%With demand variation, we can measure the speed of demand change by setting $V_{\bm{\mu}}\leq BT^v$ for $T=1,2,\dots$ where $B$ is some constant and $v\in [0,1]$ is the variation parameter. Note $V_{\bm{\mu}}\leq \lambda T$ where $\lambda=\max\{(\mu_1-\mu_2)^2:\mu_1,\mu_2\in \bm{\mu}\}$ is a natural upper bound. For simplicity, without loss of generality we may assume $B=1$ as all of our results hold with scaling. 
We denote the set of demand distribution sequences $\{D_1,\dots D_T\}$ whose means $\bm\mu=\{\mu_1,\dots,\mu_T\}$ satisfy $V_{\bm{\mu}}\leq T^v$ as
$$\mathcal{D}(v)=\{\{D_1,\dots D_T\}: D_t\in\mathcal{D}\text{ for all }t\text{ and } V_{\bm{\mu}}\leq T^v\}.$$ In the next section, we will show via a minimax lower bound that non-trivial guarantees are only achievable when $v < 1$, and provide an algorithm which achieves the same bound. 

%\vspace{1em}
\noindent
\paragraph{\bfseries\sffamily Aside: Time-Series Modeling:} At this point, we have fully described our model for the Nonstationary Newsvendor. All that remains is to define our performance metric, which we will do in the next subsection. We conclude this subsection with an important practical consideration with respect to time-series models and our variation parameter. 

Consider, as an example, the following class of time-series models:
\begin{equation} \label{eqn:time-series-example}
	d_t = R(t) + S(t) + \epsilon_t.
\end{equation}
Here, $R(t)$ represents a deterministic (and usually simple, e.g.~linear) function representing some notion of ``trend,'' and $S(t)$ represents a deterministic, periodic function representing some notion of ``seasonality.'' Finally, all stochastic behavior is captured by the random variables $\epsilon_t$, which are assumed to be independent and mean-zero. This time-series model is classic, and yet drives forecasting algorithms (e.g.~exponential smoothing) which are still competitive in modern forecasting competitions (\cite{makridakis2000m3}).

The above model raises an important practical issue: if there exists any (non-trivial) trend $R(\cdot)$ or seasonality $S(\cdot)$, then the demand variation of the sequence of means $\mu_t = R(t) + S(t)$ would scale at least as $T$, meaning $v = 1$ and no meaningful guarantee will be achievable. Our main observation is that time-series effects like trend and seasonality are easily detected and estimated, so that in any practical setting, estimates $\hat{R}(\cdot)$ and $\hat{S}(\cdot)$ should be available, and used to ``de-trend'' and ``de-seasonalize'' the data. Concretely, the Nonstationary Newsvendor would take place on the sequence $$ \tilde{d_t} = d(t) - \hat{R}(t) - \hat{S}(t) = (R(t)-\hat{R}(t)) + (S(t)-\hat{S}(t)) + \epsilon_t. $$ The resulting sequence of means $\mu_t = (R(t)-\hat{R}(t)) + (S(t)-\hat{S}(t))$ does not stem from the trend and seasonality, but rather the error in estimating the trend and seasonality. It is this error that is assumed to be nonstationary, but with reasonable variation parameter.

%\vspace{1em}
\subsection{Performance Metric: Regret}
We conclude this section by formally defining our performance metric for any policy. 
%At time period $t$, the information that the decision maker has are the past demand observations $d_1,\dots, d_{t-1}$, where each $d_i$ is drawn from $D_i$.
%, and the predictions up to the current period, $a_1,\dots, a_t$. 
A \textbf{policy} is simply a sequence of mappings $\pi=\{\pi_1,\dots,\pi_T\}$, where each $\pi_t$ is a mapping from $d_1,\dots, d_{t-1}$ to an order quantity $q_t\in Q$ at time $t$ (by convention, $\pi_1$ is a constant function).\footnote{Note that we are not considering randomized policies here, but all of our theoretical results (the lower bounds, in particular) hold even when randomization is allowed.}
We measure the performance of a policy by its \textbf{regret}. Fix a sequence of demand distributions $\bm{D} = \{D_1,\dots, D_T$\}.  Following the earlier notation from Eq. \eqref{eqn:qtstar}, the regret incurred by a policy which selects order quantities $q_1,\ldots,q_T$ is  
$$\mathbb{E}^\pi_{\bm{D}}\left[\sum_{t=1}^T\left(C_t(\mu_t,q_t)-C_t(\mu_t, q^*_t)\right)\right],$$
where the expectation is with respect to the randomness of the realized demands. Recall that the demand distributions are independent, so $q_t^*$ as defined in Eq. \eqref{eqn:qtstar} depends only on $D_t$.  
%condition on the policy $\pi$ and the prediction sequence $\bm{a}$. 
In words, the regret  measures the difference between the (expected) total cost incurred by the policy and that of a clairvoyant that knows the underlying demand distributions $\bm{D} = \{D_1,\dots, D_T$\}.\footnote{Note that this is  different from a clairvoyant that knows the realized demands $d_1,\ldots,d_T$. Such a clairvoyant would incur zero cost.} 

We will be concerned with the {\bf worst-case regret} of a policy across {families} of instances (i.e.~sequences of demand distributions) controlled by the variation parameter $v$:
$$\mathcal{R}^\pi(T)=\sup_{\bm{D}\in \mathcal{D}(v)}  \mathbb{E}^\pi_{\bm{D}}\left[\sum_{t=1}^T\left(C_t(\mu_t,q_t)-C_t(\mu_t, q^*_t)\right)\right].$$
Note that if the worst-case regret $\mathcal{R}^\pi(T)$ of some policy is sublinear in $T$, then that policy is essentially cost-optimal on average as $T$ goes to infinity. In the next section, we will prove a lower bound on the achievable across all policies, and describe an algorithm which achieves this lower bound.

\section{Solution to the Nonstationary Newsvendor (without Predictions)}\label{Section newsvendor without predictions}
This section contains a complete solution (i.e.~matching lower and upper bounds on regret) to the Nonstationary Newsvendor. We begin with the lower bound: 
\begin{proposition}[Lower Bound: Nonstationary Newsvendor] \label{lower bound on regret: past-demand-only}
	For any variation parameter $v\in[0,1]$, and any policy $\pi$ (which may depend on the knowledge of $v$), we have $$\mathcal{R}^{\pi}(T)\geq cT^{(3+v)/4},$$ where $c>0$ is a universal constant.
\end{proposition}
\cref{lower bound on regret: past-demand-only} is a corollary of a more general lower bound (\cref{lower bound on regret} in the next section) -- it will turn out the Nonstationary Newsvendor is a special case of the Nonstationary Newsvendor with Predictions -- so the proof is omitted. 
%we focus on the basic case where the decision maker does not have predictions (or has useless predictions, i.e., $a=1$) and can only use past demand observations.
\cref{lower bound on regret: past-demand-only} states that the regret of any policy is at least $\Omega(T^{(3+v)/4})$. It is useful to contrast this with two existing results:
 \begin{enumerate}
		\item {\bf Stationary Newsvendor:} In the special case of i.i.d.~demand, it is known that the optimal achievable regret is $\Theta(T^{1/2})$ -- Example 1 of \cite{besbes2013implications} demonstrates the lower bound, and the SAA method of \cite{levi2007provably,levi2015data} achieves the upper bound. 
		%    \cite{besbes2013implications} showed that for a completely stationary newsvendor (i.i.d.~demand at every time period), there is a lower bound of $\Omega(T^{1/2})$ (Example 1 of \cite{besbes2013implications}), which can be achieved by the SAA method of \cite{levi2007provably,levi2015data}. 
		This point might appear to be incompatible with our result, which states a lower bound of $\Omega(T^{3/4})$ when $v=0$, but in fact the case of $v=0$ is more general than i.i.d.~demand since it allows $O(T^0)=O(1)$ demand variation, while i.i.d.~demand amounts to \textit{zero} demand variation. Indeed, our proof of \cref{lower bound on regret: past-demand-only}, for the special case of $v=0$, utilizes instances for which the demand distribution is allowed to change $T^{1/2}$ times (by an amount of $T^{-1/4}$, resulting in $O(1)$ variation). 
		
As an aside, this discussion raises a natural question: is the disconnect here between stationary (i.i.d.) demand and variation parameter $v=0$ a consequence of our use {\em quadratic} variation, and would the same disconnect arise for other measures of demand variation? In Appendix \ref{sec:general-variation}, we answer both questions in the affirmative by showing that if the exponent $2$ in the demand variation (\cref{eq:demand-variation}) is instead some ${\theta} \ge 0$, then \cref{lower bound on regret: past-demand-only} generalizes to a lower bound of $\Omega(T^{(1+{\theta}+v)/(2+{\theta})})$. Thus, for any ${\theta} > 0$, the case of variation parameter $v=0$ is meaningfully more general than stationary demand.
		
		\item {\bf Continuous Newsvendor:} A similar ``story'' plays out in the setting of \cite{keskin2021nonstationary}, which recall (among other key differences with our model, as described in Section \ref{sec:demand-distributions}) requires the additional assumption that both the demand distributions and the possible order quantities be continuous. \cite{keskin2021nonstationary} show an optimal achievable regret of $\Theta(T^{(1+v)/2})$, which can be contrasted with the stationary (i.i.d.) setting for which an
		%lower bound of $\Theta(T^{(1+v)/2})$ for their version of the nonstationary newsvendor. 
		%which is strictly lower than our lower bound, and as we described in Section \ref{sec:demand-distributions}, comes from a stronger (quadratic) relationship between error and regret that requires stronger modeling assumptions. Note that when $v=0$, this lower bound is strictly higher than the lower bound of the i.i.d. case under the same assumptions, which is 
		$\Omega(\log T)$ lower bound exists (\cite{besbes2013implications}). The following table  summarizes these lower bounds:
		\vspace{1em}
		\begin{center}
			\begin{tabular}[h]{@{}llcc@{}}
				\toprule
				&& {\bf \small Continuous} & {\bf \small General} \\ \midrule
				{\bf \small Stationary (i.i.d.)}  && $\log T$ & $T^{1/2}$ \\ 
				{\bf \small Nonstationary}        && $T^{(1+v)/2}$ & $T^{(3+v)/4}$ \\ \bottomrule
			\end{tabular}
		\end{center}
		%This again illustrates the difference between zero variation (i.i.d.) and $O(1)$ variation $v=0$.
\end{enumerate}
\vspace{1em}

In the next two subsections, we will first analyze a simple algorithm which achieves the lower bound of \cref{lower bound on regret: past-demand-only} when the variation parameter $v$ is known, and then use this as a building block for an algorithm which achieves the same bound when $v$ is unknown.

%\vspace{1em}
\subsection{Upper Bound with Known Variation Parameter $v$}
If we assume that $v$ is known, then designing a policy which achieves regret matching  \cref{lower bound on regret: past-demand-only} is fairly straightforward. In fact, a simple policy based on averaging a fixed number of past demand observations does the job (\cite{keskin2021nonstationary} use the same policy). That policy, which we call the {\bf Fixed-Time-Window Policy} is defined in Algorithm \ref{alg:fixed}.
\begin{algorithm} 
	\caption{Fixed-Time-Window Policy}\setstretch{1.5}
	{\bf Inputs:} variation parameter $v \in [0,1]$ and scaling constant $\kappa > 0$
	
	{\bf Initialization:} $n \gets \lceil \kappa T^{(1-v) / 2} \rceil$%\footnote{$\lceil x \rceil$ rounds $x$ up to the nearest integer.}
	
	\For{$t = 1,\ldots,n$}{
		select $q_t\in Q$ arbitrarily;
	}
	\For{$t = n+1,\ldots,T$}{
		$\hat{\mu}_{t}\gets\frac{1}{n} \sum_{s=t-n}^{t-1} d_{s}$ (if $\hat{\mu}_{t}\notin[\mu_{\min},\mu_{\max}]$, round $\hat{\mu}_{t}$ to the nearest value in $[\mu_{\min},\mu_{\max}]$)\;
		%		Let $\hat{D}_t$ be the unique distribution in $\mathcal{D}$ with mean $\hat{\mu}_t$\;
		$q_t\gets \text{argmin}_{q \in Q} C_t(\hat{\mu}_t,q)$.
	}
	\label{alg:fixed}
\end{algorithm}

The Fixed-Time-Window Policy uses a carefully-selected ``window'' size $n$ that is on the order of $T^{(1-v)/2}$. At each time period $t$, it constructs an estimate $\hat{\mu}_t$ of the mean by averaging the observed demands from the previous $n$ periods, and then selects the optimal order quantity corresponding to $\hat{\mu}_t$. Note that Algorithm \ref{alg:fixed} also includes a ``scaling constant'' $\kappa$ -- this should be thought of as a practical tuning parameter, but for the coming theoretical result, it can be chosen arbitrarily (e.g.~$\kappa=1$ suffices).

The following result bounds the worst-case regret of the Fixed-Time-Window Policy:

%We now propose a policy that achieves a regret on the order of $T^{(3+v)/4}$, which matches the lower bound in Corollary \ref{lower bound on regret: past-demand-only} and therefore is optimal. In a nonstationary environment, the demand distribution changes over time. Under the demand variation constraint, over a small period of time the demand distribution does not change much, but over a long period of time the demand distribution could already have drifted a lot from the initial demand distribution. In other words, past demand observations devalue over time. Therefore, the intuition behind our policy is to only look at the past demand observations within a certain time window and ignore the time before that. We call this the \textit{Fixed-Time-Window Policy} $\pi^{\mathrm{fixed}}$, which is presented below.

\begin{lemma}[Upper Bound:  Nonstationary Newsvendor with Known $v$]\label{upper bound on regret: past-demand-only}
	Fix any variation parameter $v\in[0,1]$. The Fixed-Time-Window Policy $\pi^{\mathrm{fixed}}$ achieves worst-case regret $$\mathcal{R}^{\pi^{\mathrm{fixed}}}(T)\leq CT^{(3+v)/4},$$ where $C\le 3\max\{b_{\max},h_{\max}\}(\delta + Q_{\max})+\ell(2\sqrt{\kappa}+\sqrt{\frac{\pi}{36e}}\frac{\delta}{\sqrt{\kappa}})$, and $\delta = \sup_{\mathcal{D}_\mu \in \mathcal{D}}\|\mathcal{D}_\mu\|_{\psi_2}$.
\end{lemma}
As promised, Lemma \ref{upper bound on regret: past-demand-only} shows that the Fixed-Time-Window Policy achieves regret that matches the lower bound  in  \cref{lower bound on regret: past-demand-only}. Its proof can be found in Appendix \ref{appendix:lemma2}, and amounts to bounding the estimation error incurred by demand noise (which is worse for smaller time windows) and demand mean variation (which is worse for larger time windows). The exact time window used in the policy comes from balancing these two sources of error.

%\vspace{1em}
\subsection{Upper Bound with Unknown Variation Parameter $v$}
The lower bound in \cref{lower bound on regret: past-demand-only} holds for policies that ``know'' $v$. Naturally, it also holds for policies that do not know $v$, but an unanswered question at the moment is whether (a) the lower bound should be even larger when $v$ is unknown, or (b) there exists a policy that matches \cref{lower bound on regret: past-demand-only} without knowing $v$. We show here that case (b) holds by constructing such a policy.\footnote{As a final comparison to \cite{keskin2021nonstationary}, they do not consider the unknown $v$ setting.} 
	Our policy, which we call the \textbf{Shrinking-Time-Window Policy} (Algorithm \ref{alg:shrinking}), at a high level uses the Fixed-Time-Window Policy with the smallest variation parameter that is consistent with the demand observed so far.
	In more detail: 
	\begin{enumerate}
		\item It begins with a discrete set of candidate variation parameters $\mathcal{V} = \{v_1,\ldots,v_k\}$:
		\begin{align}
			v_j &= \left(1 + \frac{1}{\log T} \right)^{j-1} \frac{1}{\log T}, \quad j=1,\ldots,k   \label{eqn:vi}
		\end{align}
		where $k$ is chosen so that $v_{k-1} < 1 \le v_k$. $\mathcal{V}$ is defined specifically so that the variation parameters are increasing ($v_{j-1} < v_{j}$), and so that it discretizes the interval $[0,1]$ at a sufficiently fine granularity. 
		\item At any time period $t$, there is a ``current'' candidate parameter $v_i$ (initialized to be $v_1$ at $t=1$) that is assumed to be the true variation $v$, and so the corresponding Fixed-Time-Window Policy is applied: a time window of 
		\begin{align}
			n_i &= \lceil \kappa T^{(1-v_i) / 2}\rceil \label{eqn:ni}
		\end{align} 
		is used, and an estimate of $\mu_t$ is made: 
		\begin{equation} \label{eqn:mu-hat-fixed}
			\hat{\mu}_{t}^i = \frac{1}{n_i} \sum_{s=t-n_i}^{t-1} d_{s}, \quad \text{rounded to the nearest value in } [\mu_{\min},\mu_{\max}]. 
		\end{equation} 
		\item The index $i$ of the ``current'' candidate parameter $v_i$ is incremented at any period in which the policy gathers sufficient evidence that $v_i < v$. This is possible due to the following observation: if $v_i\approx v$, then by \cref{upper bound on regret: past-demand-only} we have that for any $v_j>v_i$, the regret incurred by the Fixed-Time-Window Policy corresponding to $v_j$ is $O(T^{(3+v_j)/4})$, and thus the {\em cumulative difference} between the estimated mean demands ($|\hat{\mu}_t^i-\hat{\mu}_t^j|$) cannot exceed $O(T^{(3+v_j)/4})$. Thus if this is observed for some $v_j > v_i$, we can conclude that $v_i < v$, and $i$ is incremented.
		
		%using the smaller window size corresponding to $v_j$ is $O(T^{(3+v_j)/4})$. Therefore we can track the cumulative difference between $\hat{\mu}_{t}^i$ and $\hat{\mu}_{t}^j$ for every $j>i$, and if it is too large for some $j$ we have the evidence that $v_i<v$, so we shrink the window size.
		
		%a carefully-designed ``certificate'' that the current window size is too large. 
		%is initially taken to be $1$ (corresponding to the smallest variation parameter, and thus the largest window size), and increments at any period in which the policy observes a carefully-designed ``certificate'' that the current window size is too large. 
	\end{enumerate}

\begin{algorithm}
	\setstretch{1.5}
	\caption{Shrinking-Time-Window Policy}
	{\bf Inputs:} scaling constants $\kappa > 0$, and $\gamma$ sufficiently large (Eq. {\ref{eqn:gamma} in  Appendix \ref{Appendix C}})\;
	{\bf Initialization:} Set $\mathcal{V} = \{v_1,\ldots,v_k\}$ and $\{n_1,\ldots,n_k\}$ according to Equations \ref{eqn:vi} and \ref{eqn:ni}\;
	\For{$t=1,\ldots,T^{3/4}$}{
		select $q_t\in Q$ arbitrarily;
	}
	Initialize $i \gets 1$ and $t_{\textbf{if}}\gets T^{3/4}+1$\; 
	\For{$t = T^{3/4}+1,\ldots,T$}{
		\uIf{$\sum_{s=t_{\emph{\textbf{if}}}}^{t}|\hat{\mu}_s^{i}-\hat{\mu}_s^{j}|\geq 2\left(\gamma\sqrt{\log T}+\sqrt{\kappa}\right) T^{(3+v_j)/4}$ \emph{for some} $j>i$}{
			$i \gets i+1$\;
			$t_{\textbf{if}}\gets t$\;
		}
		$q_t\gets \text{argmin}_{q \in Q} C_t(\hat{\mu}_t^{i},q)$\;
	} \label{alg:shrinking}
\end{algorithm}

This policy's regret matches (up to log factors) the lower bound in \cref{lower bound on regret: past-demand-only}:
\begin{theorem}[Upper Bound: Nonstationary Newsvendor with Unknown $v$]\label{upper bound on regret: past-demand-only with unknown variation}
	For any variation parameter $v\in[0,1]$, the Shrinking-Time-Window Policy $\pi^{\mathrm{shrinking}}$ achieves worst-case regret $$\mathcal{R}^{\pi^{\mathrm{shrinking}}}(T)\leq C T^{(3+v)/4} \log^{5/2} T,$$ where $C\le 3\max\{b_{\max},h_{\max}\}(\delta + Q_{\max})+12e^{1/4}\ell\left(\gamma+\sqrt{\kappa}\right)+e^{1/4}C_{\mathrm{\cref{upper bound on regret: past-demand-only}}}$, and $C_{\mathrm{\cref{upper bound on regret: past-demand-only}}}$ is the constant in \cref{upper bound on regret: past-demand-only}.
\end{theorem}
The proof of \cref{upper bound on regret: past-demand-only with unknown variation} can be found in Appendix \ref{Appendix C}, but at a high level works as follows: 
\begin{proof}[Proof Sketch of Theorem \ref{upper bound on regret: past-demand-only with unknown variation}]
First note that the total regret incurred during the first $T^{3/4}$ time periods is at most $O(T^{3/4})$. After that, the total regret incurred during the time periods in which the {\bf if} condition in Algorithm \ref{alg:shrinking} is triggered is at most $O(k)$, where $k$ is the number of candidate variation parameters defined in \cref{eqn:vi}, and $k \approx \log^2 T$. Thus, it suffices to bound the total regret incurred {\em between} successive triggerings of the {\bf if} condition. As a final reduction before proceeding, consider the smallest candidate variation parameter that is at least $v$, i.e.~define $\ell$ to be the smallest index such that $v\leq v_{\ell}$. Because 
$v_\ell=\left(1+\frac{1}{\log T}\right)v_{\ell-1}$, we have that $T^{v_{\ell}}$ is a constant multiple away from $T^v$. Thus, it will suffice to bound the total regret by $O(T^{(3+v_\ell)/4} \log^{5/2} T)$.

To do this, we first show that (with high probability) throughout the algorithm, the running index $i$ never exceeds $\ell$. To see this, consider the following steps:
\begin{enumerate}
	\item For every $j\geq \ell$, because $v\leq v_j$, by Lemma \ref{upper bound on regret: past-demand-only} the Fixed-Time-Window Policy corresponding to the window size $n_j$ has worst-case regret $O(T^{(3+v_j)/4})$. In addition, we can show with high probability (via Hoeffding’s inequality) that  $\sum_{s=n_j+1}^{T}|\hat{\mu}_s^{j}-\mu_s| \le \left(\gamma\sqrt{\log T}+\sqrt{\kappa}\right)T^{(3+v_j)/4}$ for every $j\geq \ell$.
	We may assume this event occurs from now on.
	\item 
	For the sake of contradiction, suppose $i$ exceeds $\ell$ at some period $t$, or equivalently, there is a period $t$ in which $i \ge \ell$, and the {\bf if} condition is triggered with some $j > i$.
	Then we have 
	\begin{align*}
		\sum_{s=t_{\textbf{if}}}^{t}|\hat{\mu}_s^{i}-\hat{\mu}_s^{j}|&\overset{(a)}{\leq}\sum_{s=t_{\textbf{if}}}^{t}|\hat{\mu}_s^{i}-\mu_s|+\sum_{s=t_{\textbf{if}}}^{t}|\hat{\mu}_s^{j}-\mu_s|\\
		&\overset{(b)}{\leq}\sum_{s=n_i+1}^{T}|\hat{\mu}_s^{i}-\mu_s|+\sum_{s=n_j+1}^{T}|\hat{\mu}_s^{j}-\mu_s|\\
		&\overset{(c)}{\leq}\left(\gamma\sqrt{\log T}+\sqrt{\kappa}\right)\cdot T^{(3+v_i)/4}+\left(\gamma\sqrt{\log T}+\sqrt{\kappa}\right)\cdot T^{(3+v_j)/4}\\
		&\overset{(d)}{<}2\left(\gamma\sqrt{\log T}+\sqrt{\kappa}\right)\cdot T^{(3+v_j)/4},
	\end{align*}
	where $(a)$ is the triangle inequality and $(b)$ holds because $t_{\textbf{if}}>T^{3/4}$ and $n_i,n_j\leq \kappa T^{1/2}$; since $j>i\geq \ell$, by our assumption in the previous step we get $(c)$; $(d)$ follows since $v_j>v_i$.
	But this directly contradicts our assumption that the {\bf if} condition is triggered at period $t$. Therefore $i$ never exceeds $\ell$.
\end{enumerate}    

Now suppose two consecutive \textbf{if} conditions occur at times $t'$ and $t''$. Between these periods, we may apply the negation of the {\bf if} condition for any $j > i$, and since $i$  never exceeds $\ell$, we specifically can take  $j = \ell$.  This yields $\sum_{s=t'+1}^{t''-1}|\hat{\mu}_s^{i}-\hat{\mu}_s^{\ell}|< 2\left(\gamma\sqrt{\log T}+\sqrt{\kappa}\right) T^{(3+v_\ell)/4}$. Then we have 
\begin{align*}
	\sum_{s=t'+1}^{t''-1}|\hat{\mu}_s^{i}-\mu_s|&\overset{(a)}{\leq}\sum_{s=t'+1}^{t''-1}|\hat{\mu}_s^{\ell}-\mu_s|+\sum_{s=t'+1}^{t''-1}|\hat{\mu}_s^{i}-\hat{\mu}_s^{\ell}|\\
	&\overset{(b)}{\leq}\sum_{s=n_i+1}^{T}|\hat{\mu}_s^{\ell}-\mu_s|+\sum_{s=t'+1}^{t''-1}|\hat{\mu}_s^{i}-\hat{\mu}_s^{\ell}|\\
	&\overset{(c)}{<}\left(\gamma\sqrt{\log T}+\sqrt{\kappa}\right)\cdot T^{(3+v_{\ell})/4}+2\left(\gamma\sqrt{\log T}+\sqrt{\kappa}\right)\cdot T^{(3+v_{\ell})/4}\\
	&=3\left(\gamma\sqrt{\log T}+\sqrt{\kappa}\right)\cdot T^{(3+v_\ell)/4},
\end{align*}
where $(a)$ is the triangle inequality and $(b)$ holds because $t'>T^{3/4}$ and $n_i\leq \kappa T^{1/2}$; the first part of $(c)$ follows by our high-probability assumption, and the second part of $(c)$ follows since the \textbf{if} condition is not triggered between time $t'$ and time $t''-1$; Therefore between two consecutive \textbf{if} conditions, the worst-case regret incurred is $O(T^{v_{\ell}}\sqrt{\log T})$. 
Because the \textbf{if} condition can happen at most $k\approx \log^2 T$ times, the total worst-case regret of the Shrinking-Time-Window Policy is $O(T^{v_{\ell}}\log^{5/2} T)=O(T^{v}\log^{5/2} T)$.
\end{proof}

This concludes our discussion of the Nonstationary Newsvendor. In the next section, we turn to the second subject of this paper, which is the same problem with predictions.

\section{The Nonstationary Newsvendor with Predictions}\label{Section newsvendor with predictions}

As described in the introduction, it is likely that when the Nonstationary Newsvendor is faced practice, some notion of a ``prediction'' of future demand will be made. 
Such predictions can come from a diverse set of sources ranging from simple human judgement, to forecasting algorithms built on previous demand data, to more-sophisticated machine learning algorithms trained on feature information. The process of sourcing or constructing such predictions is orthogonal to our work. Instead, we treat these predictions as  given to us endogenously (and in particular, we make no assumption on the accuracy of these predictions), and attempt to use these predictions optimally.

%\vspace{1em}
\subsection{Model}
The {\bf Nonstationary Newsvendor with Predictions} problem assumes all of the setup, assumptions, and notation of the previous Nonstationary Newsvendor problem. In addition, at each time period $t$, we assume that the decision-maker receives a {\bf prediction} $a_t$ before selecting an order quantity $q_t \in Q$.\footnote{We are taking the predictions to be entirely deterministic, so for example, $a_t$ is not allowed to depend on the previously-observed demands $d_1,\ldots,d_{t-1}$. Our results hold if we extend to the setting in which the predictions are stochastic (and adapted to the demand filtration).} This prediction is meant to be an estimate of $\mu_t$, and so we measure the \textbf{prediction error}  of a  sequence $\bm{a}=\{a_1,\dots, a_T\}$ with respect to a sequence of means  
simply as $$\sum_{t=1}^{T}|a_t-\mu_t|.$$
Note that unlike demand variation, we have not used partitions here (and in fact, introducing partitions would not have any effect since we are measuring {\em absolute} rather than squared differences). Intuitively, we do not want to require the sequence of errors to be meaningful time-series: the predictions are generic, and their accuracy is allowed to change rapidly. Just as for the demand variation, the prediction error is expected to grow with the time horizon $T$, and the proper parameterization of this growth is via an exponent: we call the {\bf accuracy parameter} the smallest $a \in [0,1]$ such that the prediction error is at most $T^a$. We will always assume that $a$ is unknown to the decision-maker. 

\begin{algorithm} 
	\setstretch{1.5}
	\caption{Prediction Policy}
	\For{$t = 1,\ldots,T$ }{
		$\hat{\mu}_{t}\gets a_t$ (if $\hat{\mu}_{t}\notin[\mu_{\min},\mu_{\max}]$, round $\hat{\mu}_{t}$ to the nearest value in $[\mu_{\min},\mu_{\max}]$)\;
		$q_t\gets \text{argmin}_{q \in Q} C_t(\hat{\mu}_t,q)$\;}
	\label{alg:prediction}
\end{algorithm}

Naturally, the notion of a \textbf{policy} $\pi$ expands to include the predictions: $\pi=\{\pi_1,\dots,\pi_T\}$, where each $\pi_t$ is a mapping from $d_1,\dots, d_{t-1}$ and $a_1,\ldots,a_t$ to an order quantity $q_t\in Q$. The simplest policy, which ``should'' be used if the prediction error is known to be sufficiently small, is to simply behave as if the predictions were perfect. We call this the {\bf Prediction Policy} (Algorithm \ref{alg:prediction}). The following observation collects a few (likely unsurprising) facts about the performance of this policy, with respect to worst-case regret (generalized in the ``obvious'' manner to incorporate prediction accuracy via the accuracy parameter $a$):

\begin{observation}[\textbf{Upper and Lower Bounds: Prediction Policy}]\label{upper bound on regret: predictions-only}
	Fix any variation parameter $v\in[0,1]$ and any accuracy parameter $a\in[0,1]$. 
	
	\begin{itemize}
		\item[a)] The Prediction Policy $\pi^{\mathrm{prediction}}$ achieves worst-case regret 
		$$\mathcal{R}^{\pi^{\mathrm{prediction}}}(T)\leq CT^{a},$$ 
		where $C\le 2\ell$.
		\item[b)] For any policy $\pi$ (which may depend on the knowledge of  $a$) that is solely a function of the predictions (i.e.~does not depend on the observed demands), we have $$\mathcal{R}^{\pi}(T)\geq cT^{a},$$ where $c>0$ is a universal constant.
	\end{itemize}
\end{observation}
\noindent
\cref{upper bound on regret: predictions-only}a) states that the Prediction Policy translates prediction error directly to regret (incidentally, it does this without ``knowing'' $a$). There are of course other ways in which the predictions could be used, but \cref{upper bound on regret: predictions-only}b) essentially states that there is nothing to be gained by doing so (even if $a$ is known). %We conclude this subsection with 
The proof of \cref{upper bound on regret: predictions-only}a) appears in Appendix \ref{app:A}. 
\cref{upper bound on regret: predictions-only}b) is a direct corollary of \cref{lower bound on regret}, which is given in the next subsection.

%\vspace{1em}
\subsection{Extreme Cases}
What exactly is achievable for the Nonstationary Newsvendor with Predictions  depends heavily on whether or not $v$ and $a$ are known to the policy. To see this, it is worth first considering the two extremes.

%\vspace{1em}
\noindent
\paragraph{\bfseries\sffamily Case 1: Known $v$ and $a$:} A simple policy is available when $v$ and $a$ are both known. Compare the quantities $(3+v)/4$ and $a$. If the former quantity is smaller, apply the Fixed-Time-Window Policy. If the latter is smaller, apply the Prediction Policy. \cref{upper bound on regret: past-demand-only,upper bound on regret: predictions-only} together imply that this achieves a worst-case regret of $O(T^{\min\{(3+v)/4,a\}})$. This is optimal, as demonstrated by the following result:
\begin{proposition}[Lower Bound: Known $v$ and $a$] \label{lower bound on regret}
	Fix any variation parameter $v\in[0,1]$ and any  accuracy parameter $a\in[0,1]$. For any policy $\pi$ (which may depend on  the knowledge of $v$ and $a$), we have $$\mathcal{R}^{\pi}(T)\geq cT^{\min\{(3+v)/4, a\}},$$ where $c>0$ is a universal constant.
\end{proposition}
The proof of this result can be found in Appendix \ref{appD}, and relies on an explicit construction of a family of problem instances.
%Proposition \ref{lower bound on regret} states that when the decision maker knows both the variation parameter $v$ and the accuracy parameter $a$, the the $T$-period regret of any General Policy is at least on the order of $T^{\min\{(3+v)/4,a\}}$. 
Our construction  breaks the total time horizon into cycles wherein the demand distribution is i.i.d.. We tune the length of each cycle to be small enough so that it is (provably) hard to detect the change in demand distributions and the predictions are essentially useless for most time periods in the cycle, and large enough so that the demand variation is within $T^v$ and the prediction error is within $T^a$. 

%\vspace{1em}
\noindent
\paragraph{\bfseries\sffamily Case 2: Unknown $v$ and $a$:}
At the opposite extreme, if $v$ and $a$ are both unknown, is it still possible to achieve $O(T^{\min\{(3+v)/4,a\}})$ worst-case regret? The answer is no:
\begin{proposition}[Lower Bound: Unknown $v$ and $a$] \label{lower bound on regret unknown parameters}
	For any policy that does not depend on the knowledge of $v$ or $a$, there exists a problem instance such that $a \ne (3+v)/4$, and the policy incurs regret at least $cT^{\max\{(3+v)/4,a\}}$ on the instance, where $c>0$ is a universal constant.
\end{proposition}
\cref{lower bound on regret unknown parameters} states that the best we can hope for, when $v$ and $a$ are unknown, is a worst-case regret of at least $\Omega(T^{\max\{(3+v)/4,a\}})$. Indeed, it implies that no algorithm can achieve regret $O(T^{f(v,a)})$ for a function $f:[0,1] \times [0,1] \to [0,1]$ satisfying $f(v,a) \le \max\{(3+v)/4,a\}$ for all $v,a\in [0,1]$ and $f(v,a) < \max\{(3+v)/4,a\}$ for some $v,a\in [0,1]$. Note that \cref{lower bound on regret unknown parameters} shows {\em there exists} a pair of $v$ and $a$, and a corresponding problem instance, such that this lower bound holds. This is in contrast to a result showing that {\em for any} pair of $v$ and $a$, there exists a problem instance such that the lower bound holds, as is common in the literature (e.g.~Theorem 1 of \cite{keskin2017chasing}). This lower bound is easily achieved, for example by applying the Shrinking-Time-Window Policy or the Prediction Policy (or any blind randomization of the two). The proof of \cref{lower bound on regret unknown parameters} is in Appendix \ref{appE}. In contrast to Case 1, the lower bound construction here relies heavily on the fact that we do not know which one of $(3+v)/4$ and $a$ is smaller. 

\subsection{Final Solution}
We have finally reached the problem which motivates this entire paper: designing an optimal policy for the Nonstationary Newsvendor with Predictions when the prediction error $a$ is unknown. We will assume that $v$ is known, since when $v$ is unknown, \cref{lower bound on regret unknown parameters} rules out the possibility of using the predictions to improve on what is already achievable without predictions. On the other hand, by \cref{lower bound on regret}, the absolute best we could hope for is a policy which achieves a worst-case regret of $O(T^{\min\{(3+v)/4, a\}})$. In words, we would like a policy which, without knowing $a$, achieves the same regret had $a$ been known. Our main result is the design of such a policy. 

Our policy is called the \textbf{Prediction-Error-Robust Policy (PERP)}, and is given in Algorithm \ref{alg:PERP}. PERP utilizes the Fixed-Time-Window policy $\pi^{\mathrm{fixed}}$ in Section \ref{Section newsvendor without predictions} as an estimate of the true mean to track the quality of the predictions over time. 

\begin{algorithm}
	\setstretch{1.5}
	\caption{Prediction-Error-Robust Policy (PERP)}
	{\bf Inputs:} variation parameter $v \in [0,1]$ and scaling constants $\kappa > 0$, $\gamma$ sufficiently large ({\cref{eqn:gamma2} in  Appendix \ref{Appendix F}})
	
	{\bf Initialization:} $n \gets \lceil \kappa T^{(1-v) / 2} \rceil$\;	
	\For{$t=1,\ldots,n$}{
		$\pi_t\gets\pi^{a}_t$\;}
	\For{$t=n+1,\ldots,T$}{
		$\hat{\mu}_t^{\mathrm{fixed}}\gets\frac{1}{n} \sum_{s=t-n}^{t-1} d_s$\; 
		$\hat{\mu}_t^{a}\gets a_t$\;
		\If{$\sum_{s=n+1}^{t}|\hat{\mu}_s^{a}-\hat{\mu}_s^{\mathrm{fixed}}|\geq \left(\gamma\sqrt{\log T}+\sqrt{\kappa}+1\right)\cdot T^{(3+v)/4}$}{
			$\pi_t\gets\pi^{\mathrm{fixed}}_t$\;
			{\bf break}
		}
		\Else{
			$\pi_t\gets \pi^{a}_t$\;}
	}
	\label{alg:PERP}
\end{algorithm}

\begin{theorem}[Upper Bound: Known $v$ and Unknown $a$]\label{upper bound on regret}
	For any variation parameter $v\in[0,1]$ and any accuracy parameter $a\in[0,1]$, the Prediction-Error-Robust Policy ${\pi^{\mathrm{PERP}}}$ achieves worst-case regret $$\mathcal{R}^{\pi^{\mathrm{PERP}}}(T)\leq \min\{3C_{\mathrm{\cref{upper bound on regret: past-demand-only}}} \sqrt{\log T}\cdot T^{(3+v)/4}, 2\ell \cdot T^a \},$$ where $C_{\mathrm{\cref{upper bound on regret: past-demand-only}}}$ is the constant in \cref{upper bound on regret: past-demand-only} (and $2\ell$ matches the constant in \cref{upper bound on regret: predictions-only}a).
	%$$\mathcal{R}^{\pi^{\mathrm{PERP}}}(T)\leq C T^{\min\{(3+v)/4,a\}} \sqrt{\log T},$$ where $C$ is a universal constant.
\end{theorem}

The intuition behind PERP is to follow the predictions until a time that is late enough to have evidence that the prediction quality is bad (compared to the Fixed-Time-Window Policy), while early enough to not incur much regret caused by the poor quality of the predictions. Because we do not observe the true past mean $\mu_t$ after time period $t$, we naturally use $\hat{\mu}_t^{\mathrm{fixed}}$ from the Fixed-Time-Window policy $\pi^{\mathrm{fixed}}$ as an estimation of $\mu_t$, and in turn keep tracking the cumulative difference the prediction quality $\left|a_t-\mu_t\right|$. We carefully choose the parameters in $\pi^{\mathrm{PERP}}$ so that this estimation is not accurate only with a small probability, and we can identify the prediction quality is bad if this cumulative difference is too large. By Proposition \ref{lower bound on regret}, any policy can only achieve worst-case regret on the order of $T^{\min \{(3+v) / 4, a\}}$, so PERP is order-optimal.

	%\vspace{1em}
	\noindent
	\paragraph{\bfseries\sffamily Aside: Unknown $v$ and Known $a$:} 
	There are four possible scenarios depending on the knowledge of $v$ and $a$: known/unknown $v$ and known/unknown $a$. So far we have discussed three of them: known $v$ and $a$ (\cref{lower bound on regret}), unknown $v$ and $a$ (\cref{lower bound on regret unknown parameters}), and known $v$ and unknown $a$ (\cref{upper bound on regret}). For the sake of completeness, we discuss the remaining case of unknown $v$ and known $a$ in Appendix \ref{appFinal}, where we give a policy that achieves worst-case regret $\tilde{O}(T^{\min\{(3+v)/4, a\}})$. This is order-optimal by \cref{upper bound on regret}. %We suppress this case from the main body because in reality the decision-maker rarely knows the prediction quality, and the algorithmic contribution is fairly minimal.

\section{Experiments on Real Data}\label{Section experiments}
Finally, we describe a set of experiments we performed to evaluate our policy (\texttt{PERP}) for the Nonstationary Newsvendor with Predictions. In all of our experiments, we compared \texttt{ PERP} against the Shrinking-Time-Window Policy (\texttt{NO-PRED}) and the Prediction Policy (\texttt{PURE-PRED}). 
%All of our experiments included real-world data (from a web traffic application or a retail setting), along with predictions generated from popular forecasting methods ranging from classic forecasting models to state-of-the-art machine learning algorithms. 
The main takeaways are: \begin{enumerate}
	%    {\color{blue} \item Fix the cost of \texttt{ NO-PRED} (resp. \texttt{ PURE-PRED}), \texttt{ PERP}'s cost almost stays the same as \texttt{ PURE-PRED}'s cost (resp. \texttt{ NO-PRED}) increases when \texttt{ PURE-PRED}'s cost (resp. \texttt{ NO-PRED}) is higher than \texttt{ NO-PRED}'s cost (resp. \texttt{ PURE-PRED}). This shows \texttt{ PERP} performs well even if one of \texttt{ NO-PRED} and \texttt{ PURE-PRED} performs badly.}
	\item \texttt{PERP}'s performance is robust with respect to the quality of the predictions, without knowing the prediction quality beforehand. 
	Specifically, the (newsvendor) cost it incurs is consistently ``close'' to the lower of the costs incurred by \texttt{NO-PRED} and \texttt{PURE-PRED}. 
	\item \texttt{PERP} performs especially well when the absolute difference between the costs of \texttt{NO-PRED} and \texttt{PURE-PRED} is large, i.e.~when the ``stakes'' are highest.
\end{enumerate}

	\subsection{Experiments on Synthetic Data}
	The objective of our first batch of experiments was to fix one of the two theoretical parameters ($v$ or $a$) and test \texttt{PERP}'s performance as the other parameters changes. To generate demand sequences, we used the parametric time-series model that corresponds to triple exponential smoothing (Holt Winters), a classic model for time-series data in the family of \cref{eqn:time-series-example}.
	%Triple exponential smoothing takes in a set of parameters and a series of historical data, and outputs an estimate of future data. 
	We give the exact formulas, and our choices of parameters, in Appendix \ref{app:experiment}; for more on triple exponential smoothing, see \cite{winters1960forecasting}.
	%More specifically, triple exponential smoothing contains four parameters: data smoothing factor $\alpha\in[0,1]$, trend smoothing factor $\beta\in[0,1]$, seasonal change smoothing factor $\gamma\in[0,1]$, and season length $L\in\mathbb{N}$. For specific formulas and more on triple exponential smoothing, see \cite{winters1960forecasting}. 
	%We first randomly generated a demand sequence for 30 time periods, which was treated as historical data. 
	In our experiment, each demand sequence consisted of the demands for the next 365 time periods, %, which was generated using triple exponential smoothing with some specific set of parameters. % of the experiment , which was generated as follows: First we fixed a set of parameters $\alpha,\beta,\gamma, L$. Then we used triple exponential smoothing to generate the mean of demands for 365 time periods, where 
	%At each time period we also added a random Gaussian noise and 
	with the realized demands generated as Poisson variables with the corresponding means. %{\color{red} How did you initialize the demand for the first $L$ days?}
	We ran two sets of experiments:%did the following to fix one of the two parameters ($v$ or $a$).
	\begin{itemize}
		\item \textbf{Fixed $v$:} We fixed a single set of parameters %$(\alpha,\beta,\gamma, L) = (0.5,0.5,0.5,30)$ 
		for the demand sequence and %generated a demand sequence as described above. We then 
		generated 1,000 different predictions of this demand sequence, each from a set of ``predicted'' parameters with different accuracy. %$(\hat{\alpha},\hat{\beta},\hat{\gamma}, \hat{L})$ where each $\hat{\alpha},\hat{\beta},\hat{\gamma}$ was sampled uniformly at random from $[0.2,0.8]$ and $\hat{L}$ from $\{10, 20,  30\}$. %, and then generated a predicted demand sequence using the same procedure with the set of predicted parameters. 
		Thus the variation parameter $v$ was fixed, and the accuracy parameter $a$ varied across instances. %Because the demand sequence was the same for all predictions, the variation parameter $v$ was fixed.
		
		\item \textbf{Fixed $a$:} We generated 1,000 demand sequences by selecting the parameters randomly. %$(\alpha,\beta,\gamma, L)$ 
		% uniformly at random where each $\alpha,\beta,\gamma$ was sampled uniformly at random from $[0.2,0.8]$ and $L$ from $\{10, 20,  30\}$.
		% , along with a prediction of the demand sequence. For each demand sequence, we generated a set of parameters $\alpha,\beta,\gamma, L$ uniformly at random and generated a demand sequence as described above. 
		We then generated predictions by changing each parameter 10\% %(e.g., $\alpha$ becomes either $1.1\alpha$ or $0.9\alpha$) 
		and using the corresponding sequence. %used the new set of parameters to generate a predicted demand sequence using the same procedure. 
		%Since the set of parameters was generated uniformly at random, 
		Thus the variation parameter $v$ varied across instances, but % Because the estimation error of the parameters was (roughly) the same for all predictions, 
		the accuracy parameter $a$ was (roughly) fixed.
	\end{itemize}
	
	\begin{figure}[h]
		\centering
		\begin{subfigure}{0.4\textwidth} 
			\centering
			\includegraphics[width=\linewidth]{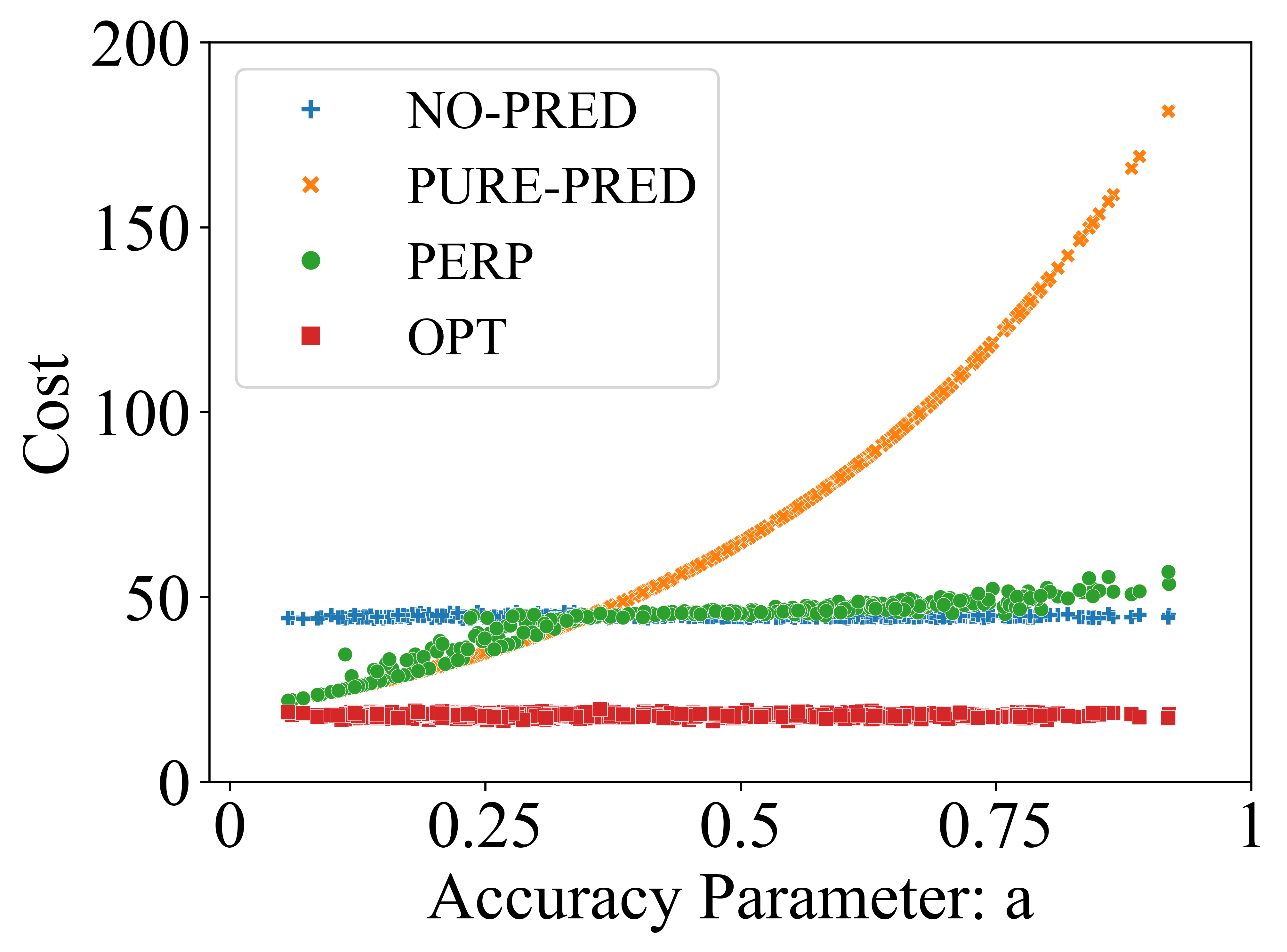}
			\caption{Fixed $v$}
		\end{subfigure}
		\hspace{2em}
		\begin{subfigure}{0.4\textwidth}
			\centering
			\includegraphics[width=\linewidth]{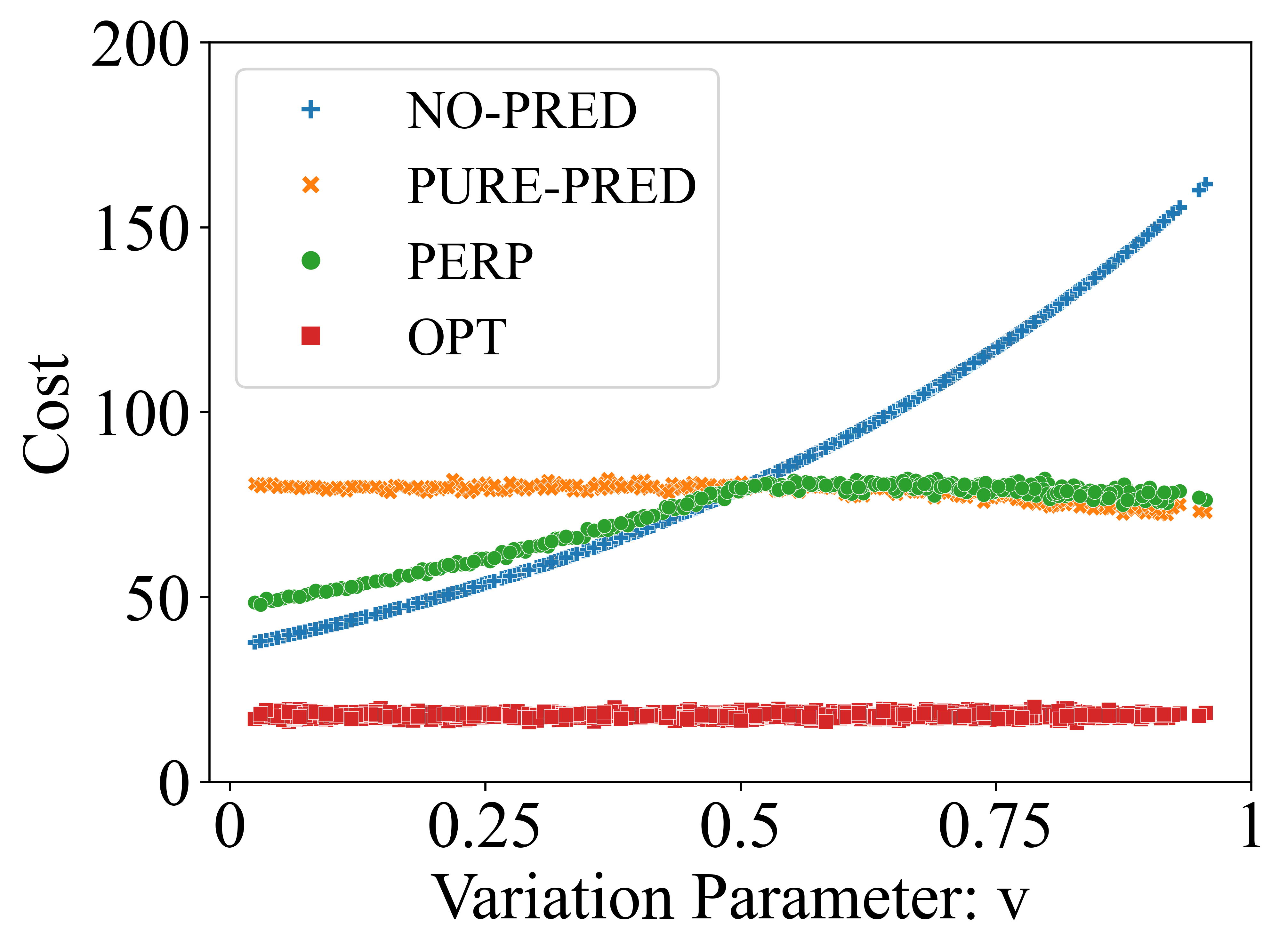}
			\caption{Fixed $a$}
		\end{subfigure}
		\caption{The costs of {\texttt NO-PRED}, {\texttt PURE-PRED}, {\texttt PERP}, and {\texttt OPT} when (a) the variation parameter $v$ is fixed, and (b) the accuracy parameter $a$ is fixed. Each dot represents the cost of the corresponding policy on a given instance.}
		\label{fig:synthetic}
	\end{figure}
	
	For each demand sequence and corresponding prediction, we ran \texttt{NO-PRED}, \texttt{PURE-PRED}, and \texttt{PERP} with equal overage and underage costs, %$h_t=b_t=0.01$, 
	and scaling constants $\kappa=\gamma=1$. Because the experiment was synthetic, the true underlying demand distribution was known at each time period. Therefore we also ran {\texttt OPT} as a benchmark, which simply ordered the optimal quantile at each time period. The variation parameter $v$ in \texttt{PERP} was calculated using the past demands of the pre-fixed 30 time periods by the definition in Section 2.2. We calculated the parameters $v$ and $a$ by their definitions (given in Section 2.2 and Section 4.1, respectively), scaled appropriately to make them lie in $[0,1]$. The resulted scatter plots are shown in \cref{fig:synthetic}. In (a), $v$ is fixed, so the cost of \texttt{NO-PRED} (blue dots) is approximately the same for all instances. The cost of \texttt{PURE-PRED} (orange dots) is approximately exponential in $a$, which follows by \cref{upper bound on regret: predictions-only}a). In (b), $a$ is fixed, so the cost of \texttt{PURE-PRED} is approximately the same for all instances. The cost of \texttt{NO-PRED} is approximately exponential in $v$, which follows by \cref{upper bound on regret: past-demand-only with unknown variation}. Note that in both (a) and (b), the cost of \texttt{PERP} (green dots) is close to the minimal cost of \texttt{NO-PRED} and \texttt{PURE-PRED} across all instances, showing that \texttt{PERP}'s performance is robust in both $v$ and $a$.

\subsection{Experiments on Real Data}
We used real-world datasets to represent the ``demand'' sequences in our experiments. \cref{fig:timeseries} depicts example time series from each of these datasets. All datasets include {\em multiple} daily time series and are publicly available: 
\begin{itemize}
	\item {\bf Rossmann:}\footnote{Available at \url{https://www.kaggle.com/competitions/rossmann-store-sales/data}} Daily number of customers that visited each of 1,115 stores in the {\em Rossmann} drug store chain during a 781-day period in 2013-2015.
	
	%    Retail stores naturally face the Nonstationary Newsvendor problem. Before each sales period, retail stores need to estimate the number of customers that will arrive during the sales period and prepare inventories accordingly. We use Rossmann drug stores' data on the number of customers arriving each day for 781 days from 2013 to 2015 as our first data set.
	
	\item {\bf Wikipedia:}\footnote{Available at \url{https://www.kaggle.com/competitions/web-traffic-time-series-forecasting/data}} %An online platform might get millions of clicks within a short time period, and the host of such online platform needs to rent servers to support the network traffic. If the host rents more servers than needed the extra servers are wasted, and if they rents less servers than needed they need to  rent extra ones during the time period, but at a higher price because of the short notice. Therefore, online platforms face the Nonstationary Newsvendor problem. Because the cost of renting extra servers in an emergency is much higher than renting servers beforehand, the critical quantile of these problems are usually very high. 
	Daily web traffic across {\em Wikipedia.com} pages of 9 different languages for an 803-day period from 2015 to 2017.
	
		\item {\bf Restaurant:}\footnote{Available at \url{https://www.kaggle.com/competitions/recruit-restaurant-visitor-forecasting/}} Daily number of visitors and online reservations across 185 restaurants in Japan, during a 478-day period in 2016-2017. We treated the number of visitors as the ``demand,'' and the reservations as a predictive feature.
	
\end{itemize}

\begin{figure}[h]
	\centering
	\begin{subfigure}{0.32\textwidth} 
		\centering
		\includegraphics[width=.91\linewidth]{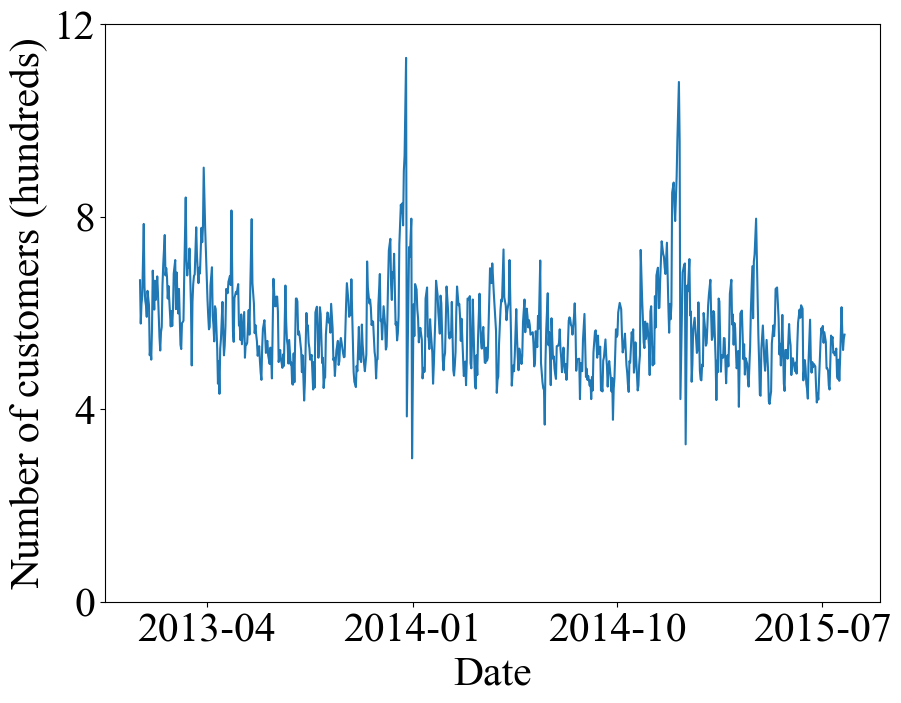}
		\caption{Rossmann}
		\label{fig:rossmann}
	\end{subfigure}
	%\hfill
	\begin{subfigure}{0.32\textwidth}
		\centering
		\includegraphics[width=.93\linewidth]{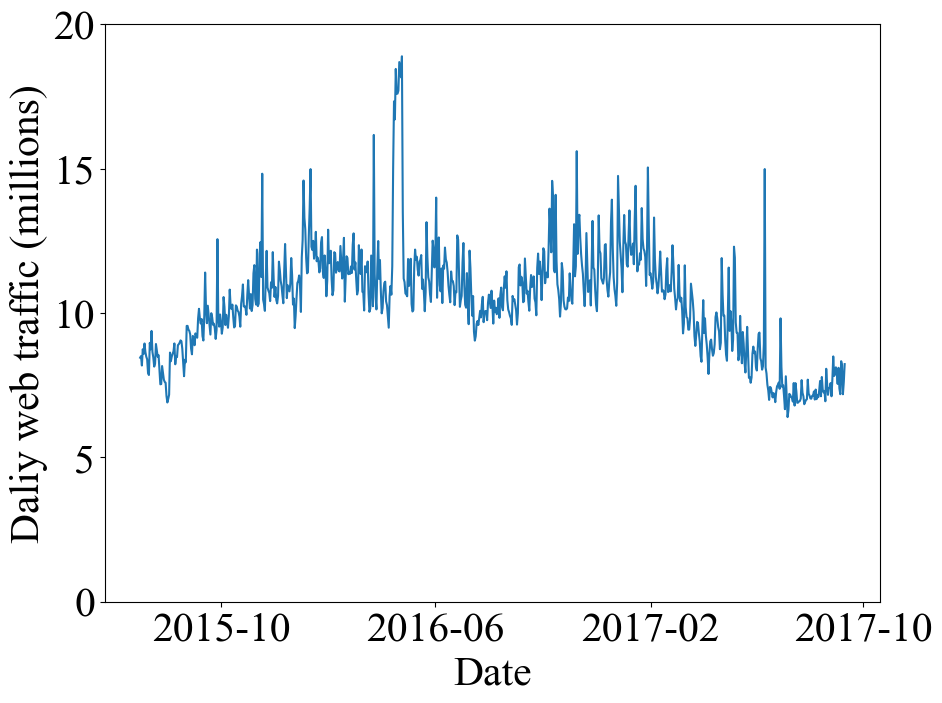}
		\caption{Wikipedia}
		\label{fig:wikipedia}
	\end{subfigure}
	\begin{subfigure}{0.32\textwidth}
		\centering
		\includegraphics[width=.93\linewidth]{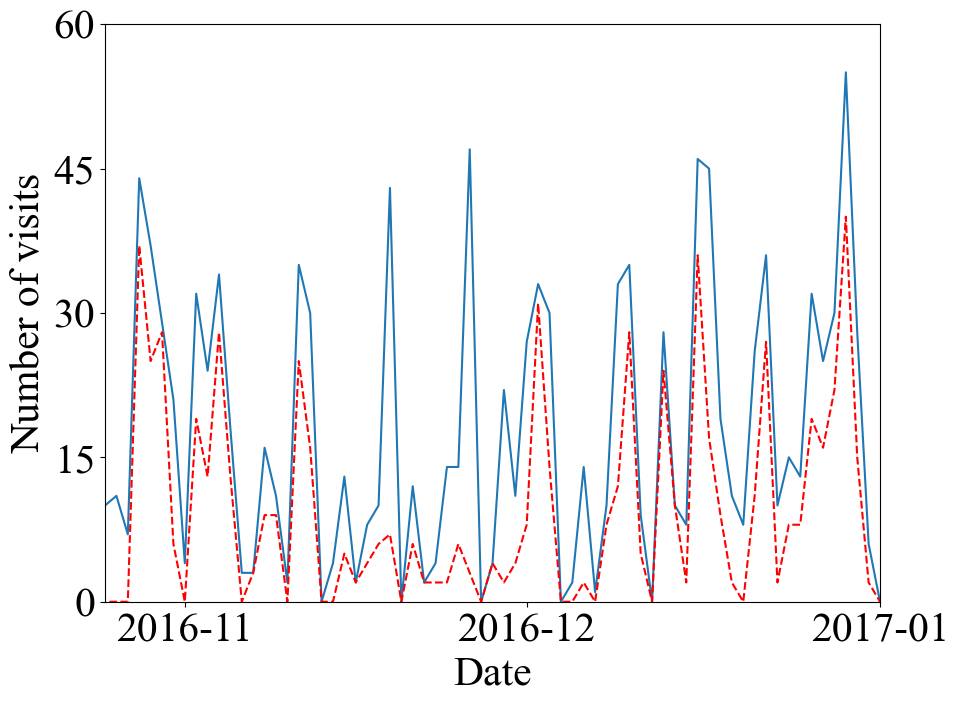}
		\caption{Restaurant}
		\label{fig:restaurant}
	\end{subfigure}
	\caption{An example of a single time series from each dataset. In (c), the red dashed line represents an additional fetaure: daily online reservations.}
	\label{fig:timeseries}
\end{figure}

Each {\em instance} of our experiment represented a single Nonstationary Newsvendor with Predictions problem, with the realized demands taken from a single time series in our data (a single Rossmann store, a single language on Wikipedia, or a single restaurant). The overage and underage costs were constant within each instance, and without loss of generality the two costs for an instance can be characterized by the corresponding critical quantile (specifically the ratio of the underage cost to the sum). The time horizon for each instance was a set number of days taken from the end of the time series, with the preceding days used to train one of four prediction methods. These predictions were also updated over the course of the instance at a set frequency. For the Wikipedia dataset, this yielded a total of 2,880 possible instances, all of which were tested. The Rossmann dataset has multiple orders of magnitude more instances, so we randomly sampled 1,000 from this set. For the Restaurant dataset, we used a single prediction method to generate two sets of predictions for each restaurant: one only utilized the number of past visitors and the other incorporated the number of reservations as a feature, which gave 740 instances. \cref{tab:instances} describes all of the instances used.
\begin{table}[h] \centering \small
	\begin{tabular}{@{}l l l l@{}} \toprule
		& {\bf Rossmann} & {\bf Wikipedia} & {\bf Restaurant}\\ \midrule
		Number of time series & 1,115 & 9 & {185} \\
		Critical quantiles (\%) & 30,40,50,60,70 & 95,98,99,99.9 & {50}\\
		Experimental period (days) & 300,400,500,600\; & 300,400,500,600,700 & {100}\\
		Prediction update frequency (days) \;\;\;\; & 2,4,10,20 & 2,4,10,20 & {5, 10}\\
		Total number of instances & 1,000 (sampled) & 2,880 (exhaustive) & {740 (exhaustive)}\\ \bottomrule
	\end{tabular} 
	\caption{Description of experimental instances.}
	\label{tab:instances}
\end{table}

For each instance, we applied \texttt{NO-PRED}, \texttt{PURE-PRED}, and \texttt{PERP} with scaling constants $\kappa=\gamma=1$, and the variation parameter $v$ in \texttt{PERP} was calculated using the past demands of the training data by the definition in Section 2.2. To generate predictions, we used four popular forecasting method ranging from classic to the state-of-the art:
\begin{itemize}
	\item {\bf Exponential Smoothing (Holt Winters):} A classic algorithm based on a (linear) trend and seasonality decomposition as in \cref{eqn:time-series-example}, known for its simplicity and robust performance. It is frequently used as a benchmark in forecasting competitions (\cite{makridakis2000m3}). Tuning parameters: seasonality of length 50.
	\item {\bf ARIMA:} Another classic algorithm that is rich enough to model a wide class of nonstationary time-series. Tuning parameters: $(p,q,r)=(3,2,5)$.
	\item {\bf Prophet:} A recent algorithm developed by Facebook (\cite{taylor2018forecasting}) based on a (piecewise-linear) trend and seasonality decomposition as in \cref{eqn:time-series-example}, known to work well in practice with minimal tuning. Tuning parameters: software default.
	\item {\bf LightGBM:} A recent algorithm developed by Microsoft (\cite{ke2017lightgbm}) based on tree algorithms. LightGBM formed the core of most of the top entries in the recent \$100,000 M5 Forecasting Challenge (\cite{makridakis2022m5}). Tuning parameters: software default.
\end{itemize}

For the Restaurant dataset, we used Prophet as the forecasting method, with and without the reservations as an additive linear feature. We treated the outputs of these methods as predictions of the {\em mean} demand. To estimate the demand distribution around this mean, we used the empirical distribution of the residuals of the same predictions on the training period.\footnote{That is, if the training data consists of $T_{\mathrm{train}}$ periods, which without loss we index as $\{t=-T_{\mathrm{train}}+1,T_{\mathrm{train}}+2,\ldots,-1,0\}$, then the demand distribution at any time $t$ was estimated to be 
	%for every prediction $\hat{\mu}_t$ in the testing period $t$, the demand distribution was estimated as 
	$$\hat{\mu}_t+\text{Uniform}\left(\{d_s - \hat{\mu}_s: s=-T_{\mathrm{train}}+1,T_{\mathrm{train}}+2,\ldots,-1,0\}\right).$$} In practice, even if the prediction quality is good, the predictions of the first few days might incur large costs due to noise/instability of the predictions, which may cause \texttt{PERP} to misidentify the prediction quality. Therefore we restricted \texttt{PERP} to following the predictions for the first 20 days, only allowing switches afterward.% -- this is to prevent premature switching caused by the noise/instability of the predictions. 

%\vspace{1em}
\paragraph{\bf Results:}
Each instance yields three total costs: one incurred by \texttt{PERP}, and two incurred by the benchmark algorithms (\texttt{NO-PRED} and \texttt{PURE-PRED}).
The primary performance metric we report is a form of optimality gap. For an instance $I$, let $\mathrm{cost}^{\mathrm{PURE-PRED}}(I)$ be the cost of \texttt{PURE-PRED}, and similarly define $\mathrm{cost}^{\mathrm{NO-PRED}}(I)$, $\mathrm{cost}^{\mathrm{PERP}}(I)$. Then the optimality gap (GAP) of \texttt{PERP} is defined as
$$\mathrm{GAP}(I)=\frac{\mathrm{cost}^{\mathrm{PERP}}(I)-\min\{\mathrm{cost}^{\mathrm{PURE-PRED}}(I),\mathrm{cost}^{\mathrm{NO-PRED}}(I)\}}{
	%\max\{\mathrm{cost}^{\mathrm{PURE-PRED}}(I),\mathrm{cost}^{\mathrm{NO-PRED}}(I)\}-\min\{\mathrm{cost}^{\mathrm{PURE-PRED}}(I),\mathrm{cost}^{\mathrm{NO-PRED}}(I)\}
	\left|\mathrm{cost}^{\mathrm{PURE-PRED}}(I) - \mathrm{cost}^{\mathrm{NO-PRED}}(I)\right|
}.$$ 
If we think of \texttt{PERP} as trying to achieve the minimum of the costs incurred by the two benchmark policies, then GAP measures the excess cost that \texttt{PERP} incurs on top of this minimum, normalized so that $\mathrm{GAP}=0$ implies that the minimum has been achieved, and $\mathrm{GAP}=1$ implies that the maximum of the two costs was incurred.\footnote{GAP may technically be outside of $[0,1]$.}

%If we think of the minimum cost incurred by the two benchmark policies as the lowest achievable cost, and likewise the maximum of the two as 

\begin{figure}[h]
	\centering
	\begin{subfigure}{0.32\textwidth}
		\centering
		\includegraphics[width=\linewidth]{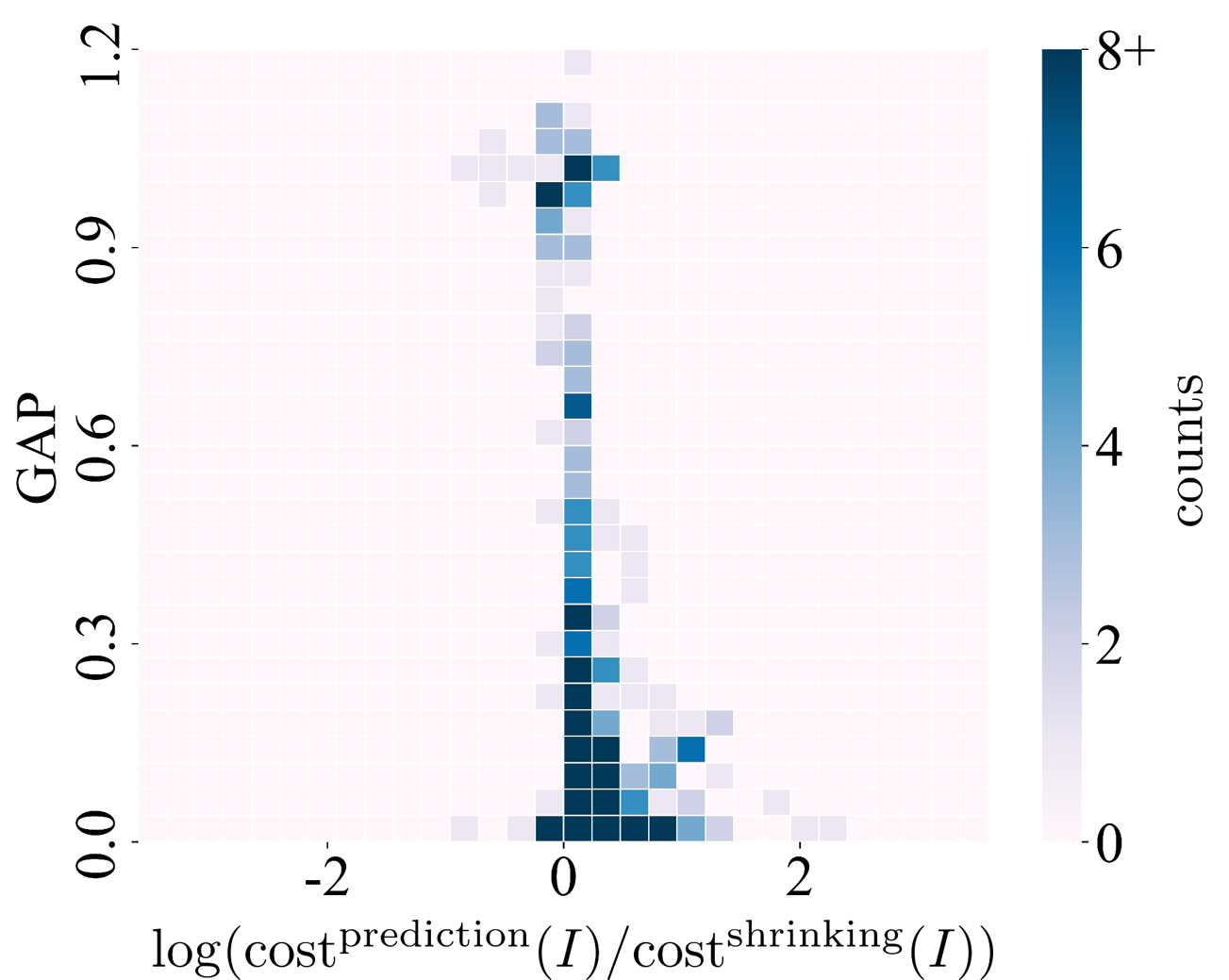}
		\caption{Rossmann}
	\end{subfigure}
	\hfill
	\begin{subfigure}{0.32\textwidth}
		\centering
		\includegraphics[width=\linewidth]{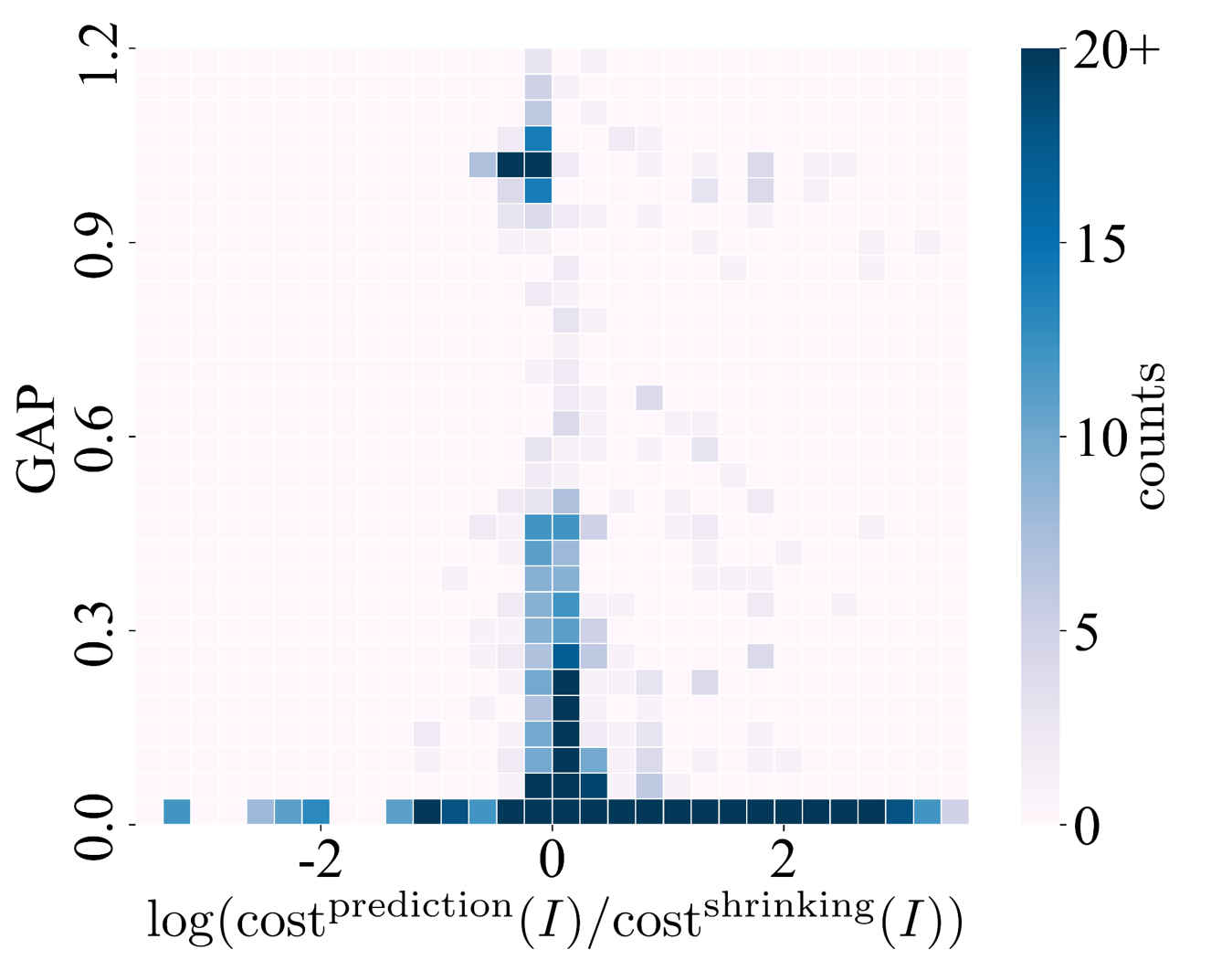}
		\caption{Wikipedia}
	\end{subfigure}
	\hfill
	\begin{subfigure}{0.32\textwidth}
		\centering
		\includegraphics[width=\linewidth]{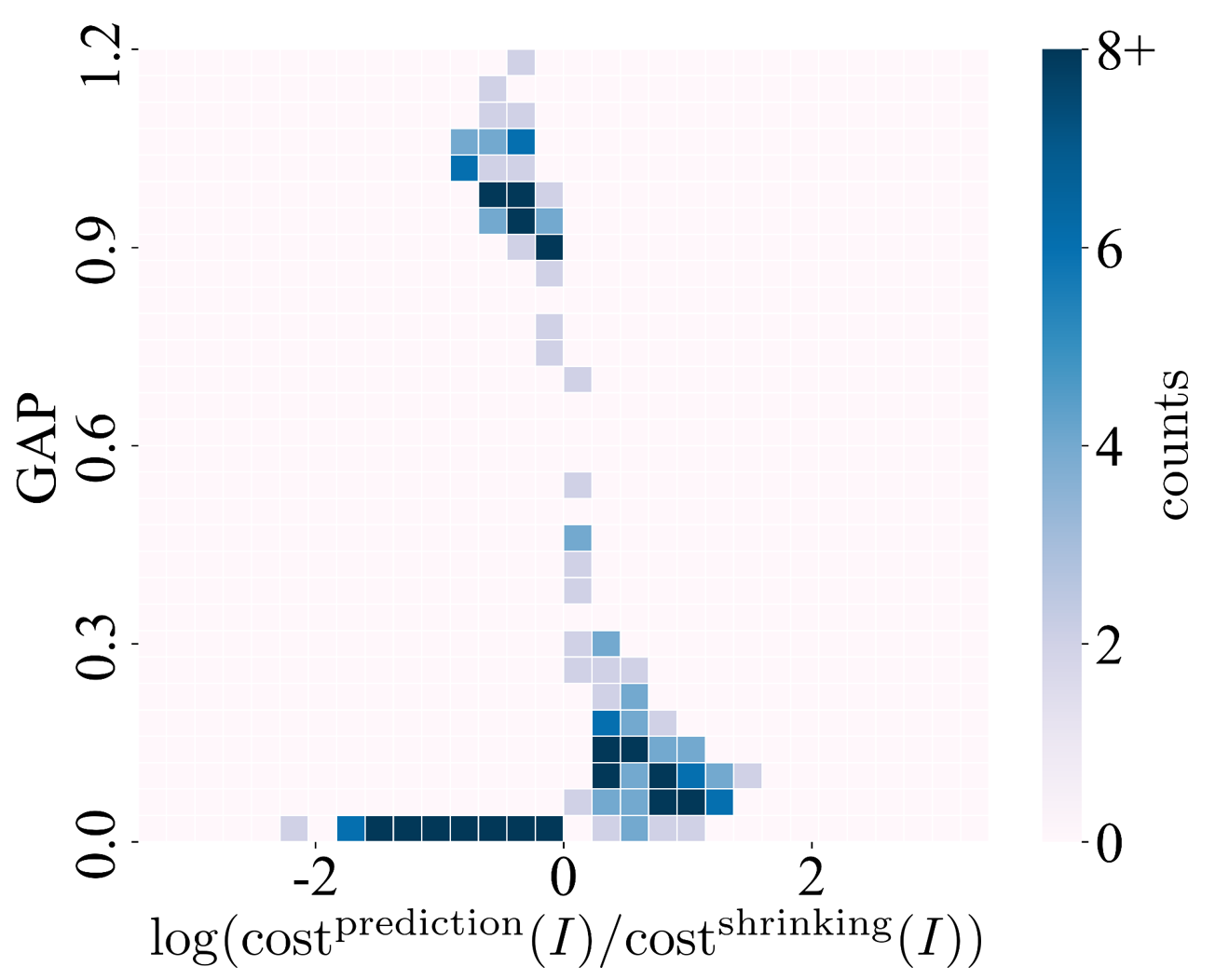}
		\caption{Restaurant}
	\end{subfigure}
	\caption{Histograms of GAPs across (a) 1,000 randomly-sampled instances on the Rossmann dataset, (b) 2,880 instances on the Wikipedia dataset, {and (c) 740 instances on the Restaurant dataset.}}
	\label{fig:gap}
\end{figure}

Experiments on the datasets yielded the  histograms in \cref{fig:gap}. For each instance $I$, the value on the horizonal axis is $\log(\mathrm{cost}^{\mathrm{PURE-PRED}}(I)/\mathrm{cost}^{\mathrm{NO-PRED}}(I))$, which is greater than 0 if \texttt{NO-PRED} has a lower cost, and less than 0 if \texttt{PURE-PRED} has a lower cost. In the 1,000 Rossman instances \texttt{NO-PRED} had a lower cost 82.7\% of the time, in the 2,880 Wikipedia instances \texttt{NO-PRED} had a lower cost 81.9\% of the time,  and in the 740 Restaurant instances \texttt{NO-PRED} had a lower cost 64.3\% of the time.
The values on the vertical axis are the GAPs. Note that most GAPs are small when the absolute values of the log difference are large. This shows \texttt{PERP} performs very well when the difference of costs between \texttt{NO-PRED} and \texttt{PURE-PRED} is large. On the other hand, there are instances where \texttt{PERP} has large GAPs, in particular there are instances with GAPs equal to 1 when the log difference of costs is close to 0. This happens because when the log difference of costs is close to 0, the cost of \texttt{NO-PRED} and the cost of \texttt{PURE-PRED} are close, so \texttt{PERP} may misidentify the prediction quality. Still, since the max cost and the min cost of the other two policies are close, even the GAPs are large in these instances, \texttt{PERP} does not perform badly.

\begin{table}[h] \centering
	\begin{tabular}{@{}l c c c@{}} \toprule
		& {\bf \small Rossmann} & {\bf \small Wikipedia} & {\bf \small Restaurant} \\ \midrule
		Average GAP with good predictions & 0.26  & 0.40 & {0.10} \\
		Average GAP with bad predictions & 0.28 & 0.07 & {0.39}\\ \bottomrule
	\end{tabular} 
	\caption{Summary of experimental results.}
	\label{tab:results}
\end{table}

We further divide the instances according to which of \texttt{NO-PRED} and \texttt{PURE-PRED} had lower cost in \cref{tab:results} and  
\cref{fig:separated-gap}. \begin{figure}[h]
	\centering
	\begin{subfigure}[t]{\textwidth}
		% \centering
		\includegraphics[width=.295\linewidth]{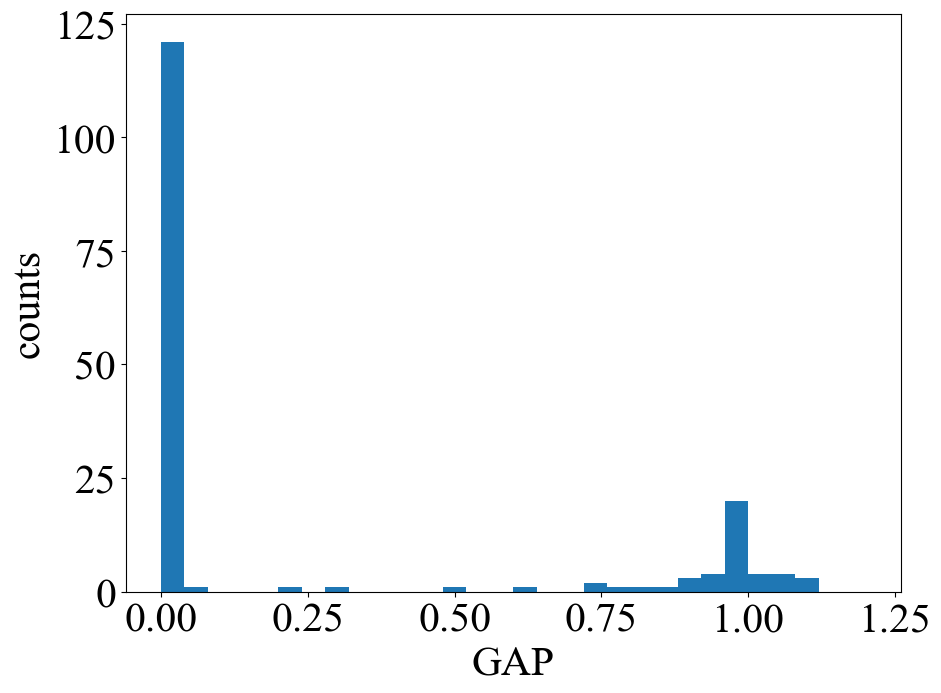}
		% \caption{Rossmann: Low Prediction Accuracy}
		% \end{subfigure}
	\hfill
	% \begin{subfigure}{0.46\textwidth}
		% \centering
		\includegraphics[width=.302\linewidth]{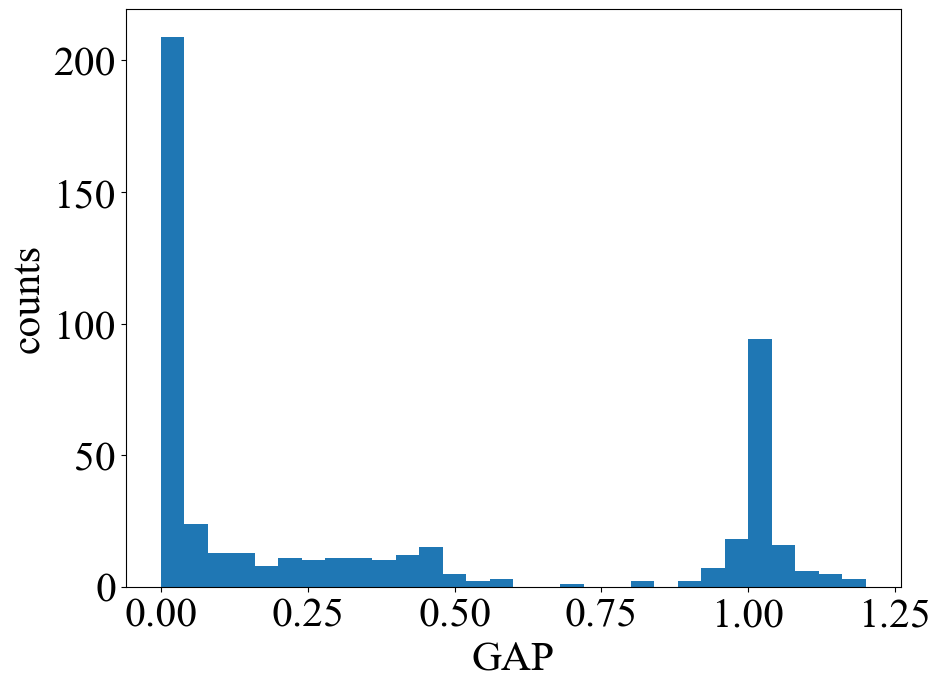}
		\hfill
		\includegraphics[width=.295\linewidth]{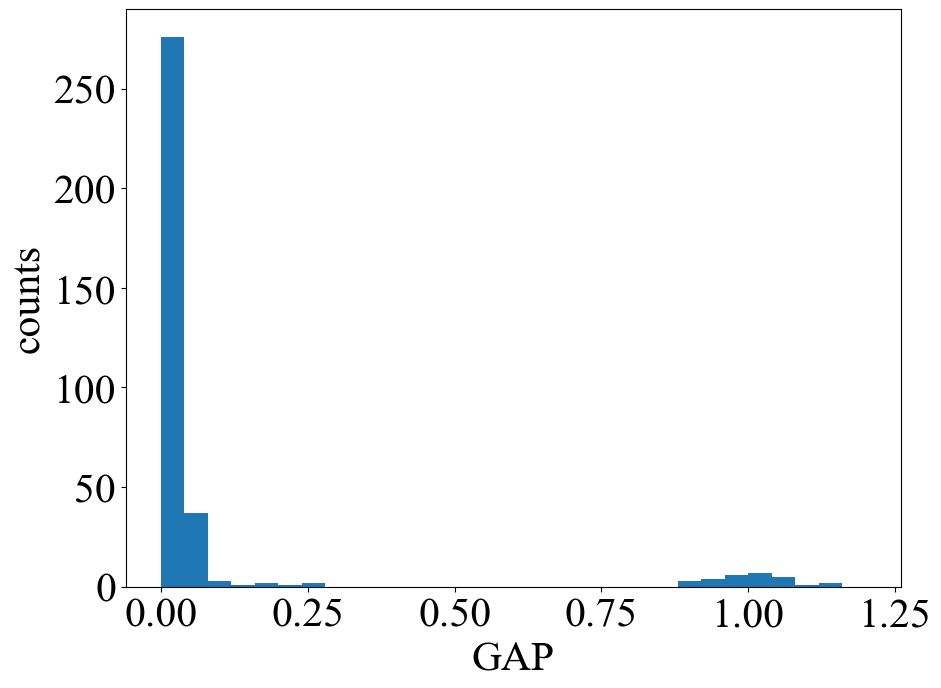}
		\caption{GAP with high prediction accuracy. Left to right: Rossmann, Wikipedia, Restaurant.\vspace{1em}\quad}
	\end{subfigure}
	\begin{subfigure}[t]{\textwidth}
		% \centering
		\includegraphics[width=.295\linewidth]{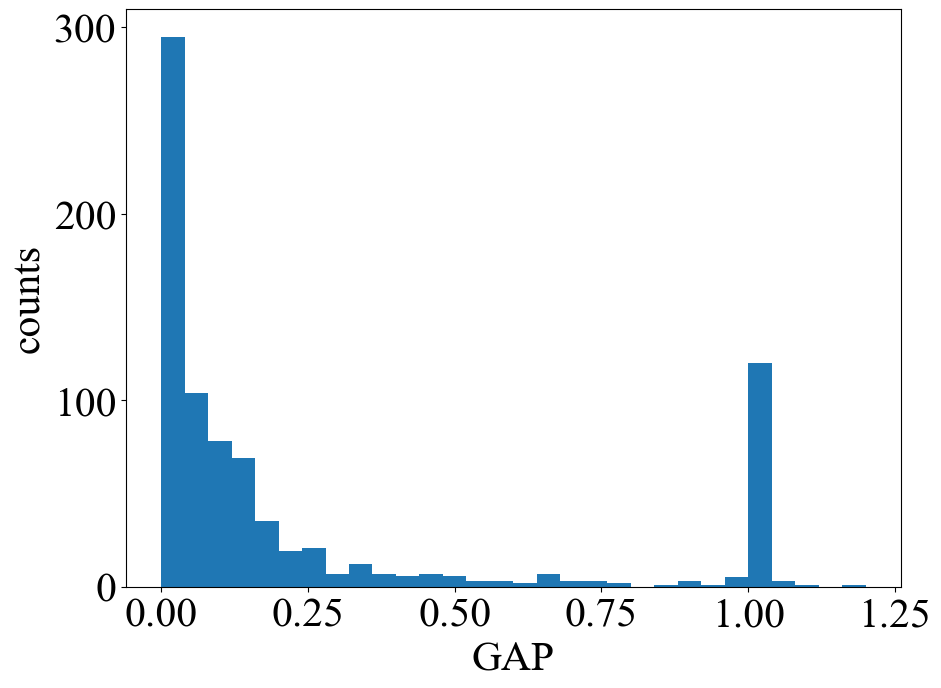}
		% \caption{Rossmann: Low Prediction Accuracy}
		% \end{subfigure}
	\hfill
	% \begin{subfigure}{0.46\textwidth}
		% \centering
		\includegraphics[width=.302\linewidth]{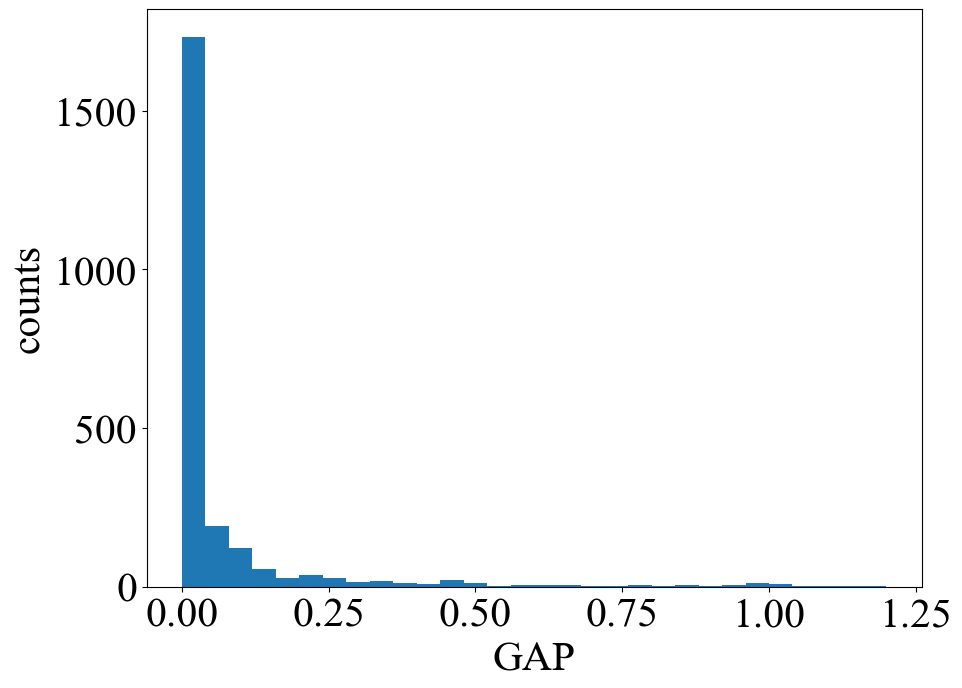}
		\hfill
		\includegraphics[width=.295\linewidth]{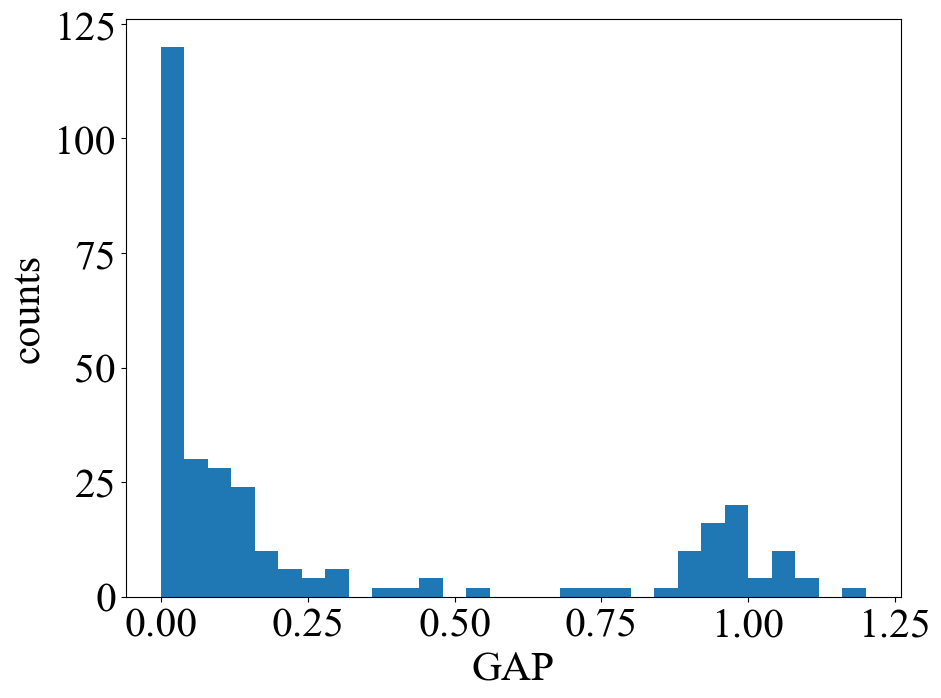}
		\caption{GAP with low prediction accuracy. Left to right: Rossmann, Wikipedia, Restaurant.\vspace{1em}\quad}
	\end{subfigure}
	
	%   \begin{subfigure}{0.46\textwidth}
		%   \centering
		%   \includegraphics[width=.85\linewidth]{restaurant_hist_bad.png}
		%   \caption{Wikipedia: Low Prediction Accuracy}
		% \end{subfigure}
	% \hfill
	% \begin{subfigure}{0.46\textwidth}
		%   \centering
		%   \includegraphics[width=.85\linewidth]{Wikipedia_advice.png}
		%   \caption{Wikipedia: High Prediction Accuracy}
		% \end{subfigure}
	%   \begin{subfigure}{0.46\textwidth}
		%   \centering
		%   \includegraphics[width=.85\linewidth]{restaurant_hist_bad.png}
		%   \caption{Restaurant: Low Prediction Accuracy}
		% \end{subfigure}
	% \hfill
	% \begin{subfigure}{0.46\textwidth}
		%   \centering
		%   \includegraphics[width=.85\linewidth]{restaurant_hist_good.png}
		%   \caption{Restaurant: High Prediction Accuracy}
		% \end{subfigure}
	
	% \caption{\color{blue} Histograms of the GAPs for (a) 827 Rossmann instances for which the Shrinking-Time-Window Policy has lower cost, (b) 173 Rossmann instances for which the Prediction Policy has lower cost, (c) 2358 Wikipedia instances for which the Shrinking-Time-Window Policy has lower cost, and (d) 522 Wikipedia instances for which the Prediction Policy has lower cost.}
	\caption{Histograms of the GAPs for (a) 173 Rossmann instances (left), 522 Wikipedia instances (middle), 476 Restaurant instances (right) for which \texttt{PURE-PRED} has lower cost, and (b) 827 Rossmann instances (left), 2358 Wikipedia instances (middle), 264 Restaurant instances (right) for which \texttt{NO-PRED} has lower cost.}
	\label{fig:separated-gap}
\end{figure} %In the Rossman dataset, the average GAP was 0.280 on the 173 instances where \texttt{ PURE-PRED} had a lower cost, and 0.258 on the remaining 827 instances. In the Wikipedia dataset, the average GAP was 0.404 on the 522 instances where \texttt{ PURE-PRED} had a lower cost, and 0.071 on the other 2358 instances. {\color{blue} In the Restaurant dataset, the average GAP was 0.095 on the 476 instances where \texttt{ PURE-PRED} had a lower cost, and 0.394 on the other 264 instances.} 
For comparison, if we did not know the prediction quality beforehand, uniformly random choosing between \texttt{NO-PRED} and \texttt{PURE-PRED} has an expected GAP of 0.5. Therefore \texttt{PERP} outperforms this natural benchmark in all cases of all datasets.

\section{Conclusion}\label{Section conclusion}

We proposed a new model  incorporating predictions into the nonstationary newsvendor problem.  
We  first gave a complete analysis of the Nonstationary Newsvendor (without predictions) by proving a lower regret bound and developing the Shrinking-Time-Window Policy, which was the first policy that achieves the lower bound up to log factors without knowing the variation parameter. %Moreover, our model was the first nonstationary newsvendor model where the distribution and order quantities are allowed to be discrete.  
We then considered the Nonstationary Newsvendor with Predictions and proposed the Prediction-Error-Robust Policy, which does not need to know the prediction quality beforehand, and achieves  nearly optimal minimax worst-cast regret.

\bibliographystyle{apalike}

\newpage

\appendix
	
	\section{Preliminary Observations and Proofs}\label{app:A}
	\begin{proof}[Proof of Lemma \ref{l-Lipchitz}.] By Assumption \ref{assu:lipschitz}, we have
	\begin{eqnarray*}
		C(\mu_1,q_2^*)-C(\mu_1,q_1^*) &=& C(\mu_1,q_2^*)-C(\mu_2,q_1^*)+C(\mu_2,q_1^*)-C(\mu_1,q_1^*)\\
		&\overset{(a)}{\leq}& C(\mu_1,q_2^*)-C(\mu_2,q_1^*)+\ell|\mu_1-\mu_2|\\
		&\overset{(b)}{\leq}& C(\mu_1,q_2^*)-C(\mu_2,q_2^*)+\ell|\mu_1-\mu_2|\\
		&\overset{(c)}{\leq}& 2\ell|\mu_1-\mu_2|.
	\end{eqnarray*}
	Here $(a)$ and $(c)$ follow from $C(\mu,q)$ being $\ell$-Lipchitz in $\mu$, and $(b)$ uses the definition of $q_2^*$. 
	\end{proof}
	
	\begin{proof}[Proof of \cref{upper bound on regret: predictions-only}a)]
	By Lemma \ref{l-Lipchitz}, $$C_t(\mu_t,q_t)-C_t(\mu_t,q_t^*)\leq 2\ell|\hat{\mu}_t-\mu_t|,$$ where $\ell$ is the Lipschitz constant of $C(\cdot, q)$. Therefore, 
	\begin{eqnarray*}
		\mathcal{R}^{\pi^{\mathrm{prediction}}}(T)&=&\sup_{\bm{D}\in \mathcal{D}(v)}\mathbb{E}^{\pi^{\mathrm{prediction}}}_{\bm{D}}\left\{\sum_{t=1}^{T}\left(C_t(\mu_t,q_t)-C_t(\mu_t, q^*_t)\right)\right\}\\
		&\leq& 2\ell\cdot\sup_{\bm{D}\in \mathcal{D}(v)}\mathbb{E}^{\pi^{\mathrm{prediction}}}_{\bm{D}}\left\{\sum_{t=1}^T|\hat{\mu}_t-\mu_t|\right\}.
	\end{eqnarray*}
	Finally, by the construction of $\hat{\mu}_{t}$, $$\sum_{t=1}^{T}|\hat{\mu}_{t}-\mu_{t}|\leq\sum_{t=1}^{T}|a_t-\mu_t|\leq T^a.$$
	Taking $C=2\ell$ gives the desired result. 
	\end{proof}
	
	For any time periods $a,b$, let $\mathcal{R}^{\pi}(T)[a,b]=\sup_{\bm{D}\in \mathcal{D}(v)}\mathbb{E}^{\pi}_{\bm{D}}\left\{\sum_{t=a}^{b}\left(C_t(\mu_t,q_t)-C_t(\mu_t, q^*_t)\right)\right\}$ be the regret incurred by the policy $\pi$ from time $a$ to time $b$. We 
	make the following two useful observations:
	
	\begin{observation}\label{regret doing nothing}
		For any policy $\pi$, $\mathcal{R}^{\pi}(T)[a,b]\leq C(b-a+1)$ for some universal constant $C\in(0,\infty)$.
	\end{observation}

\begin{proof}
Because $\mu_t$ and $q_t$ are both bounded, $C_t(\mu_t,q_t)$ is also bounded in $[C_{\min},C_{\max}]$ where $$C_{\min}=\inf_{\substack{\mu_t\in [\mu_{\min},\mu_{\max}],q_t\in Q\\b_t\in[0,b_{\max}],h_t\in[0,h_{\max}]}}C_t(\mu_t,q_t) \qquad\text{ and }\qquad C_{\max}=\sup_{\substack{\mu_t\in [\mu_{\min},\mu_{\max}],q_t\in Q\\b_t\in[0,b_{\max}],h_t\in[0,h_{\max}]}}C_t(\mu_t,q_t).$$
Thus, $$\mathcal{R}^{\pi}(T)[a,b]=\sup_{\bm{D}\in \mathcal{D}(v)}\mathbb{E}^{\pi}_{\bm{D}}\left\{\sum_{t=a}^{b}\left(C_t(\mu_t,q_t)-C_t(\mu_t, q^*_t)\right)\right\}\leq (C_{\max}-C_{\min})(b-a+1).$$
Thus, $$\mathcal{R}^{\pi}(T)[a,b]=\sup_{\bm{D}\in \mathcal{D}(v)}\mathbb{E}^{\pi}_{\bm{D}}\left\{\sum_{t=a}^{b}\left(C_t(\mu_t,q_t)-C_t(\mu_t, q^*_t)\right)\right\}\leq (C_{\max}-C_{\min})(b-a+1).$$ Take $C=(C_{\max}-C_{\min})$ gives the desired result. 
\end{proof}

\begin{observation}\label{enough to estimate mean}
	For any policy $\pi$ such that, from time $a$ to time $b$, $\pi$ first estimates the mean at time $t$ to be $\hat{\mu}_t\in [\mu_{\mathrm{min}},\mu_\mathrm{max}]$ for every $a\leq t\leq b$ and then order $q_t\in\text{argmin}_{q \in Q} C_t(\hat{\mu}_t,q)$, we have $\mathcal{R}^{\pi}(T)[a,b]\leq C\cdot \sup_{\bm{D}\in \mathcal{D}(v)}\mathbb{E}^{\pi}_{\bm{D}}\left\{\sum_{t=a}^b|\hat{\mu}_t-\mu_t|\right\}$ for some universal constant $C\in(0,\infty)$.
\end{observation}

\begin{proof}
Let $\ell$ be the Lipschitz constant of $C(\cdot, q)$, then by Lemma \ref{l-Lipchitz} $$C_t(\mu_t,q_t)-C_t(\mu_t,q_t^*)\leq 2\ell|\hat{\mu}_t-\mu_t|.$$ Therefore, 
\begin{eqnarray*}
	\mathcal{R}^{\pi}(T)[a,b]&=&\sup_{\bm{D}\in \mathcal{D}(v)}\mathbb{E}^{\pi}_{\bm{D}}\left\{\sum_{t=a}^{b}\left(C_t(\mu_t,q_t)-C_t(\mu_t, q^*_t)\right)\right\}\\
	&\leq& 2\ell\cdot\sup_{\bm{D}\in \mathcal{D}(v)}\mathbb{E}^{\pi}_{\bm{D}}\left\{\sum_{t=a}^b|\hat{\mu}_t-\mu_t|\right\}.
\end{eqnarray*}
Taking $C=2\ell$ gives the desired result. 
\end{proof}

Observation \ref{enough to estimate mean} implies that any policy $\pi$ that estimates the mean accurately at each time achieves low regret. 

Finally, we will make use of standard sub-Gaussian concentration:
\begin{lemma}[Hoeffding's Inequality, e.g. \cite{vershynin2018high}]\label{lemma:subgaussian}
	Let $\epsilon_1,\ldots,\epsilon_n$ be independent, mean-zero, sub-Gaussian variables with sub-Gaussian norm at most $K$:
	\[ \mathbb{P}\left(|\epsilon_i| > t \right) \le 2 \exp\left(  -\frac{t^2}{K^2}\right) \;\; \text{ for all } t \ge 0. \]
	Then for some universal constant C,
	\[  \mathbb{P}\left(\frac{1}{n}\left|\sum_{i=1}^n\epsilon_i\right| > t \right) \le 2 \exp\left(  -\frac{nt^2}{CK^2}\right)  \;\; \text{ for all } t \ge 0. \]
	Moreover, $C$ is upper bounded by $144e$.
\end{lemma}

 \vspace{1em}
\section{Proof of Lemma \ref{upper bound on regret: past-demand-only}}
\begin{proof}[\textbf{Proof of Lemma \ref{upper bound on regret: past-demand-only}.}] \label{appendix:lemma2}
Because our bounds are all asymptotic, we ignore the rounding and write $n= \kappa T^{(1-v) / 2}$ to simplify the notation.

By Observation \ref{regret doing nothing}, $\mathcal{R}^{\pi^{\mathrm{fixed}}}(T)[1,n]\leq C_1n< C_1T^{(3+v)/4}$ for some universal constant $C_1=3\max\{b_{\max},h_{\max}\}(\delta + Q_{\max})\in(0,\infty)$.

From now on we consider the time period from $t=n+1$ to $t=T$. We first upper bound the total estimation error of the mean $\sum_{t=n+1}^{T}|\hat{\mu}_t-\mu_t|$. Note that
\begin{eqnarray*}
	\sum_{t=n+1}^{T}|\hat{\mu}_t-\mu_t|&\overset{(a)}{\leq}& \sum_{t=n+1}^{T}\left|\frac{1}{n} \sum_{s=t-n}^{t-1} d_{s}-\mu_t\right|\\
	&=&\sum_{t=1}^{T}\left|\frac{1}{n} \sum_{s=t-n}^{t-1} (\mu_s+\epsilon_s)-\mu_t\right|\\
	&=&\sum_{t=n+1}^{T}\left|\frac{1}{n} \sum_{s=t-n}^{t-1} (\mu_s-\mu_t)+\frac{1}{n}\sum_{s=t-n}^{t-1}\epsilon_s\right|\\
	&\leq& \sum_{t=n+1}^{T}\left|\frac{1}{n} \sum_{s=t-n}^{t-1} (\mu_s-\mu_t)\right|+\sum_{t=n+1}^{T}\left|\frac{1}{n}\sum_{s=t-n}^{t-1}\epsilon_s\right|,
\end{eqnarray*}
where $(a)$ follows because $\hat{\mu}_t$ is the projection of $\frac{1}{n} \sum_{s=t-n}^{t-1} d_{s}$ on  $[\mu_{\min},\mu_{\max}]$.

We bound these two parts separately through the following two lemmas.

\begin{lemma}\label{error in mean}
	There exists a universal constant $C_2\in(0,\infty)$ such that $$\sum_{t=n+1}^{T}\left|\frac{1}{n} \sum_{s=t-n}^{t-1} (\mu_s-\mu_t)\right|\leq C_2T^{(3+v)/4}.$$
\end{lemma} 

\begin{proof}[Proof of Lemma \ref{error in mean}.]
For any $n+1\leq t\leq T$ we have $$\left|\frac{1}{n} \sum_{s=t-n}^{t-1}\left(\mu_{s}-\mu_{t}\right)\right| \leq \frac{1}{n} \sum_{s=t-n}^{t-1}\left|\mu_{s}-\mu_{t}\right| \leq \max _{t-n \leq s \leq t-1}\left|\mu_{s}-\mu_{t}\right|.$$
Also, by bounded demand variation,
\begin{eqnarray*}
	\sum_{t=n+1}^{T} \max _{t-n \leq s \leq t-1}\left|\mu_{s}-\mu_{t}\right|^{2} & \overset{(a)}{=}& \sum_{j=1}^{\lceil T / n\rceil} \sum_{i=1}^{n} \max _{(j-1) n+i \leq s \leq j n+i-1}\left|\mu_{s}-\mu_{t}\right|^{2} \\
	& \overset{(b)}{=}& \sum_{i=1}^{n} \sum_{j=1}^{\lceil T / n\rceil} \max _{(j-1) n+i \leq s \leq j n+i-1}\left|\mu_{s}-\mu_{t}\right|^{2} \\
	& \overset{(c)}{\leq}& n V_{\bm\mu}\\
	&\leq&\kappa T^{(1+v)/2},
\end{eqnarray*}
where $(a)$ is obtained by partitioning the sum into time windows, $(b)$ is obtained by exchanging the summations, and $(c)$ follows by the definition of demand variation $V_{\bm\mu}$ since $\{t_j=jn+i-1:j=0,1,\dots, \lceil T / n\rceil\}$ is a partition of $\{1,\dots, T\}$ for all $i=1,\dots,n$. 

By Cauchy–Schwarz inequality, \begin{eqnarray*}\sum_{t=n+1}^{T}\max _{t-n \leq s \leq t-1}\left|\mu_{s}-\mu_{t}\right|&\leq&\sqrt{ (T-n)\cdot \sum_{t=n+1}^{T} \max _{t-n \leq s \leq t-1}\left|\mu_{s}-\mu_{t}\right|^{2}}\\ &\leq& \sqrt{(T-n)\cdot \kappa T^{(1+v)/2}}\\
	&\leq& \sqrt{\kappa }T^{(3+v)/4}.\end{eqnarray*}
Take $C_2=\sqrt{\kappa }$ gives the desired result. 
\end{proof}

\begin{lemma}\label{error in noise}
	There exists a universal constant $C_3\in(0,\infty)$ such that $$\mathbb{E}^{\pi^{\mathrm{fixed}}}_{\bm{D}}\left\{\sum_{t=n+1}^{T}\left|\frac{1}{n}\sum_{s=t-n}^{t-1}\epsilon_s\right|\right\}\leq C_3T^{(3+v)/4}.$$
\end{lemma} 

\begin{proof}[Proof of Lemma \ref{error in noise}.]
Because each $D_s$ is sub-Gaussian, each $\epsilon_s$ is sub-Gaussian. Therefore there exists a constant $\delta_s<\infty$ where $\delta_s=\inf\{\delta'\geq 0:\mathbb{E}[e^{\epsilon_s^2/\delta'^2}]\leq 2\}$. Let $\delta=\max_{s=1,\dots, T}\delta_s$, then by Hoeffding's inequality, for any $n+1\leq t\leq T$ we have
$$\mathbb{P}\left\{\left|\frac{1}{n}\sum_{s=t-n}^{t-1}\epsilon_s\right|\geq x\right\}\leq 2\text{ exp}\left(-\frac{\rho (nx)^2}{\sum_{s=t-n}^{t-1}\delta_s^2}\right)\leq 2\text{ exp}\left(-\frac{\rho n}{\delta^2}x^2\right),$$ where $\rho>0$ is a universal constant.
Therefore we have
\begin{eqnarray*}
	\mathbb{E}^{\pi^{\mathrm{fixed}}}_{\bm{D}}\left\{\left|\frac{1}{n}\sum_{s=t-n}^{t-1}\epsilon_s\right|\right\}&=&\int_{0}^{\infty}\mathbb{P}\left\{\left|\frac{1}{n}\sum_{s=t-n}^{t-1}\epsilon_s\right|\geq x\right\}dx\\
	&\leq& \int_{0}^{\infty}2\text{ exp}\left(-\frac{\rho n}{\delta^2}x^2\right)dx\\
	&=&\sqrt{\frac{\pi\delta^2}{\rho n}}\\
	&=&\sqrt{\frac{\pi\delta^2}{\rho\kappa}}T^{(v-1)/4}.
\end{eqnarray*}
Hence $\mathbb{E}^{\pi^{\mathrm{fixed}}}_{\bm{D}}\left\{\sum_{t=n+1}^{T}\left|\frac{1}{n}\sum_{s=t-n}^{t-1}\epsilon_s\right|\right\}\leq (T-n)\sqrt{\frac{\pi\delta^2}{\rho\kappa}}T^{(v-1)/4}\leq  \sqrt{\frac{\pi\delta^2}{\rho\kappa}}T^{(3+v)/4}$. Take $C_3=\sqrt{\frac{\pi\delta^2}{\rho\kappa}}$ gives the desired result. Note by \cref{lemma:subgaussian} $1/\rho\leq 144e$, so $C_3\leq\frac{\delta}{12}\sqrt{\frac{\pi}{e\kappa}}$.
\end{proof}

Now we finish the proof of Lemma \ref{upper bound on regret: past-demand-only}. With Lemma \ref{error in mean} and Lemma \ref{error in noise}, we conclude that 
\begin{eqnarray*}
	\mathbb{E}^{\pi^{\mathrm{fixed}}}_{\bm{D}}\left\{\sum_{t=n+1}^{T}|\hat{\mu}_t-\mu_t|\right\}&\leq &\mathbb{E}^{\pi^{\mathrm{fixed}}}_{\bm{D}}\left\{\sum_{t=n+1}^{T}\left|\frac{1}{n} \sum_{s=t-n}^{t-1} (\mu_s-\mu_t)\right|\right\}+\mathbb{E}^{\pi^{\mathrm{fixed}}}_{\bm{D}}\left\{\sum_{t=n+1}^{T}\left|\frac{1}{n}\sum_{s=t-n}^{t-1}\epsilon_s\right|\right\}\\&\leq& (C_2+C_3)T^{(3+v)/4}. 
\end{eqnarray*}
Therefore by Observation \ref{enough to estimate mean} there exists a universal constant $C_4=2\ell\in(0,\infty)$ such that
\begin{align*}
	\mathcal{R}^{\pi^{\mathrm{fixed}}}(T)&=\mathcal{R}^{\pi^{\mathrm{fixed}}}(T)[1,n]+\mathcal{R}^{\pi^{\mathrm{fixed}}}(T)[n+1,T]\\
	&\leq C_1T^{(3+v)/4}+C_4(C_2+C_3)T^{(3+v)/4}\\
	&\leq (C_1+C_4(C_2+C_3))T^{(3+v)/4}.
\end{align*}
Take $C=(C_1+C_4(C_2+C_3))$ we get $\mathcal{R}^{\pi^{\mathrm{fixed}}}(T)\leq CT^{(3+v)/4}$.
\end{proof}

	\vspace{1em}

\section{Proof of Theorem \ref{upper bound on regret: past-demand-only with unknown variation}}\label{Appendix C}
\begin{proof}[\textbf{Proof of Theorem \ref{upper bound on regret: past-demand-only with unknown variation}.}]
Because our bounds are all asymptotic, we ignore the roundings and write $n_i= \kappa T^{(1-v_i) / 2}$ to simplify the notation.

By Observation \ref{regret doing nothing} $\mathcal{R}^{\pi^{\mathrm{shrinking}}}(T)[1,T^{3/4}]\leq C_1T^{3/4}$ for some universal constant $C_1=3\max\{b_{\max},h_{\max}\}(\delta + Q_{\max})\in(0,\infty)$. 

From now on we consider the time periods after $T^{3/4}$. Let $\ell$ be the smallest index such that $v_\ell\geq v$, then $v_{\ell}\leq (1+\frac{1}{\log T})v$, so 
\begin{equation} \label{eqn:ell} T^{v_{\ell}}\leq T^{(1+1/\log T)v}=e^vT^v\leq eT^v.\end{equation} First, we show that $\hat{\mu}_t^{j}$ is close to $\mu_t$ when $j\geq \ell$ via the following lemma:
\begin{lemma}\label{small failing probability}
	For every $j\geq \ell$ with the corresponding window size $n_j=\kappa T^{(1-v_j)/2}$, there exists a universal constant $\gamma>0$ such that
	$$\mathbb{P}\left\{\sum_{t=n_j+1}^{T}\left|\hat{\mu}_t^{j}-\mu_t\right|\geq \left(\gamma\sqrt{\log T}+\sqrt{\kappa}\right)\cdot T^{(3+v_j)/4}\right\}\leq \frac{2}{T^{3/2}}.$$
\end{lemma}

\begin{proof}[Proof of Lemma \ref{small failing probability}.]
Same as in the proof of Lemma \ref{error in noise}, by Hoeffding's inequality for any $n_j+1\leq t\leq T$ we have
$$\mathbb{P}\left\{\left|\frac{1}{n_j}\sum_{s=t-n_j}^{t-1}\epsilon_s\right|\geq x\right\}\leq 2\text{ exp}\left(-\frac{\rho (n_jx)^2}{\sum_{s=t-n_j}^{t-1}\delta_s^2}\right)\leq 2\text{ exp}\left(-\frac{\rho n_j}{\delta^2}x^2\right),$$
where $\rho$ and $\delta$ are the same as in the previous proof. Set $\gamma>0$ to be large enough so that 
\begin{equation} \label{eqn:gamma} \frac{\rho\kappa \gamma^2}{\delta^2}\geq \frac{5}{2}.\end{equation} 
Take $x=\gamma\sqrt{\log T}\cdot T^{(v_j-1)/4}$ and plug in $n_j=\kappa T^{(1-v_j)/2}$ yields $$\mathbb{P}\left\{\left|\frac{1}{n_j}\sum_{s=t-n_j}^{t-1}\epsilon_s\right|\geq \gamma\sqrt{\log T}\cdot T^{(v_j-1)/4}\right\}\leq2\text{ exp}\left(-\frac{\rho\kappa \gamma^2}{\delta^2}\log T\right)\leq \frac{2}{T^{5/2}}.$$
Then we get \begin{align*}
	\mathbb{P}\left\{\sum_{t=n_j+1}^{T}\left|\frac{1}{n_j}\sum_{s=t-n_j}^{t-1}\epsilon_s\right|\geq \gamma\sqrt{\log T}\cdot T^{(3+v_j)/4}\right\}
	&\leq\mathbb{P}\left\{\max_{n_j+1\leq t\leq T}\left|\frac{1}{n_j}\sum_{s=t-n_j}^{t-1}\epsilon_s\right|\geq \gamma\sqrt{\log T}\cdot T^{(v_j-1)/4}\right\}\\
	&\overset{(a)}{\leq} \sum_{t={n_j+1}}^T\cdot\mathbb{P}\left\{\left|\frac{1}{n_j}\sum_{s=t-n_j}^{t-1}\epsilon_s\right|\geq \gamma\sqrt{\log T}\cdot T^{(v_j-1)/4}\right\}\\
	&\leq \frac{2}{T^{3/2}},
\end{align*}
where $(a)$ follows by union bound. Note that since $V_{\mathbb{\bm{\mu}}}\leq T^v\leq T^{v_j}$, by Lemma \ref{error in mean} we know $\sum_{t=n_j+1}^{T}\left|\frac{1}{n_j} \sum_{s=t-n_j}^{t-1} (\mu_s-\mu_t)\right|\leq \sqrt{\kappa}T^{(3+v_j)/4}$, and in the proof of Lemma \ref{upper bound on regret: past-demand-only} we have 
$$\sum_{t=n_j+1}^{T}|\hat{\mu}_t^{j}-\mu_t|\leq\sum_{t=n_j+1}^{T}\left|\frac{1}{n_j} \sum_{s=t-n_j}^{t-1} (\mu_s-\mu_t)\right| +\sum_{t=n_j+1}^{T}\left|\frac{1}{n_j}\sum_{s=t-n_j}^{t-1}\epsilon_s\right|.$$ 
Therefore \begin{align*}
	&\quad\mathbb{P}\left\{\sum_{t=n_j+1}^{T}\left|\hat{\mu}_t^{j}-\mu_t\right|\geq \left(\gamma\sqrt{\log T}+\sqrt{\kappa}\right)\cdot T^{(3+v_j)/4}\right\}\\
	&\leq \mathbb{P}\left\{\sum_{t=n_j+1}^{T}\left|\frac{1}{n_j} \sum_{s=t-n_j}^{t-1} (\mu_s-\mu_t)\right| +\sum_{t=n_j+1}^{T}\left|\frac{1}{n_j}\sum_{s=t-n_j}^{t-1}\epsilon_s\right|\geq \left(\gamma\sqrt{\log T}+\sqrt{\kappa}\right)\cdot T^{(3+v_j)/4}\right\}\\
	&\leq \mathbb{P}\left\{\sum_{t=n_j+1}^{T}\left|\frac{1}{n_j}\sum_{s=t-n_j}^{t-1}\epsilon_s\right|\geq \gamma\sqrt{\log T}\cdot T^{(3+v_j)/4}\right\}\\
	&\leq  \frac{2}{T^{3/2}}.
\end{align*}
\end{proof}

For each $j\geq \ell$, let $E_{j}$ be the event $\left\{\sum_{t=n_j+1}^{T}\left|\hat{\mu}_t^{j}-\mu_t\right|\geq \left(\gamma\sqrt{\log T}+\sqrt{\kappa}\right)\cdot T^{(3+v_j)/4}\right\}$, then $\mathbb{P}(E_j)\leq  \frac{2}{T^{3/2}}$ for each $j$. First we assume that $E_{j}$ does not happen for any $j\geq \ell$. We break the proof into three lemmas.

\begin{lemma}\label{if happens}
	Each time an \textbf{if} condition happens, for the current $i$ we have $i<\ell$. Therefore $i\leq\ell$ throughout the algorithm.
\end{lemma}
\begin{proof}[Proof of Lemma \ref{if happens}]
Suppose an \textbf{if} condition happens at time $t$ triggered by some index $j>i$, then $\sum_{s=t_{\textbf{if}}}^{t}|\hat{\mu}_s^{i}-\hat{\mu}_s^{j}|\geq 2\left(\gamma\sqrt{\log T}+\sqrt{\kappa}\right) T^{(3+v_j)/4}$. Suppose  for the sake of contradiction that $i\geq \ell$, then 
\begin{align*}
	\sum_{s=t_{\textbf{if}}}^{t}|\hat{\mu}_s^{i}-\hat{\mu}_s^{j}|&\overset{(a)}{\leq}\sum_{s=t_{\textbf{if}}}^{t}|\hat{\mu}_s^{i}-\mu_s|+\sum_{s=t_{\textbf{if}}}^{t}|\hat{\mu}_s^{j}-\mu_s|\\
	&\overset{(b)}{\leq}\sum_{s=n_i+1}^{T}|\hat{\mu}_s^{i}-\mu_s|+\sum_{s=n_j+1}^{T}|\hat{\mu}_s^{j}-\mu_s|\\
	&\overset{(c)}{<}\left(\gamma\sqrt{\log T}+\sqrt{\kappa}\right)\cdot T^{(3+v_i)/4}+\left(\gamma\sqrt{\log T}+\sqrt{\kappa}\right)\cdot T^{(3+v_j)/4}\\
	&\overset{(d)}{<}2\left(\gamma\sqrt{\log T}+\sqrt{\kappa}\right)\cdot T^{(3+v_j)/4},
\end{align*}
where $(a)$ comes from the triangle inequality and $(b)$ is because $t_{\textbf{if}}>T^{3/4}$ and $n_i,n_j\leq \kappa T^{1/2}$; since $j>i\geq \ell$, by our assumption neither $E_{i}$ nor $E_{j}$ occurs, so we get $(c)$; $(d)$ follows since $v_j>v_i$. This contradicts with $\sum_{s=t_{\textbf{if}}}^{t}|\hat{\mu}_s^{i}-\hat{\mu}_s^{j}|\geq 2\left(\gamma\sqrt{\log T}+\sqrt{\kappa}\right) T^{(3+v_j)/4}$. Therefore, we must have $i<\ell$. Because the index $i$ never decreases and can only increase by 1 each time an \textbf{if} condition happens, $i\leq\ell$ throughout the algorithm.

\begin{lemma}\label{between if's}
	Suppose two consecutive \textbf{if} conditions occur at time $t'$ and $t''$, then $\mathcal{R}^{\pi^{\mathrm{shrinking}}}(T)[t',t''-1]\leq C_2T^{(3+v)/4}\sqrt{\log T}$ for some universal constant $C_2\in(0,\infty)$.
\end{lemma}
\end{proof}

\begin{proof}[Proof of Lemma \ref{between if's}]
At time $t$ where $t'\leq t\leq t''-1$, by Lemma \ref{if happens} $i<\ell$. We have
\begin{align*}
	\sum_{s=t'+1}^{t''-1}|\hat{\mu}_s^{i}-\mu_s|&\overset{(a)}{\leq}\sum_{s=t'+1}^{t''-1}|\hat{\mu}_s^{\ell}-\mu_s|+\sum_{s=t'+1}^{t''-1}|\hat{\mu}_s^{i}-\hat{\mu}_s^{\ell}|\\
	&\overset{(b)}{\leq}\sum_{s=n_i+1}^{T}|\hat{\mu}_s^{\ell}-\mu_s|+\sum_{s=t'+1}^{t''-1}|\hat{\mu}_s^{i}-\hat{\mu}_s^{\ell}|\\
	&\overset{(c)}{<}\left(\gamma\sqrt{\log T}+\sqrt{\kappa}\right)\cdot T^{(3+v_{\ell})/4}+2\left(\gamma\sqrt{\log T}+\sqrt{\kappa}\right)\cdot T^{(3+v_{\ell})/4}\\
	&\overset{(d)}{\leq}3e^{1/4}\left(\gamma\sqrt{\log T}+\sqrt{\kappa}\right)\cdot T^{(3+v)/4},
\end{align*}
where $(a)$ comes from the triangle inequality and $(b)$ is because $t'>T^{3/4}$ and $n_i\leq \kappa T^{1/2}$; the first part of $(c)$ follows by our assumption that $E_{\ell}$ does not occur, and the second part of $(c)$ follows since the \textbf{if} condition is not triggered between time $t'$ and time $t''-1$; $(d)$ follows by \cref{eqn:ell}. Then by Observation \ref{enough to estimate mean} there exists a universal constant $C'=2\ell\in (0,\infty)$ such that $$\mathcal{R}^{\pi^{\mathrm{shrinking}}}(T)[t',t''-1]\leq C'\sum_{s=t'}^{t''-1}|\hat{\mu}_s^{i}-\mu_s|\leq 3e^{1/4}C'\left(\gamma\sqrt{\log T}+\sqrt{\kappa}\right)\cdot T^{(3+v)/4}.$$ Take $C_2=3e^{1/4}C'\left(\gamma+\sqrt{\kappa}\right)$ gives the desired result. \end{proof}

The above proof also works for $\mathcal{R}^{\pi^{\mathrm{shrinking}}}(T)[T^{3/4}+1,t_{\mathrm{first}}-1]$ if the first \textbf{if} condition happens at time $t_{\mathrm{first}}$, or $\mathcal{R}^{\pi^{\mathrm{shrinking}}}(T)[T^{3/4}+1,T]$ if the \textbf{if} condition never happens. Note that the index $i$ never decreases and increases by 1 if and only if the \textbf{if} condition happens. Suppose that the last \textbf{if} condition happens at time $t_{\mathrm{last}}$, then we have 
\begin{align*}
	\mathcal{R}^{\pi^{\mathrm{shrinking}}}(T)[T^{3/4}+1,t_{\mathrm{last}}-1]&\overset{(a)}{\leq} (\text{\# of \textbf{if} conditions happened})\cdot C_2T^{(3+v)/4}\sqrt{\log T}\\
	&\overset{(b)}{\leq} kC_2T^{(3+v)/4}\sqrt{\log T}\\
	&\overset{(c)}{\approx} C_2\log_{1+\frac{1}{\log T}}(T)\cdot T^{(3+v)/4}\sqrt{\log T}\\
	&=C_2\frac{\log (T)}{\log(1+\frac{1}{\log T})}T^{(3+v)/4}\sqrt{\log T}\\
	&\overset{(d)}{\approx} C_2 T^{(3+v)/4}\log^{5/2} T
\end{align*}
for some universal constant $C_2\in(0,\infty)$. Here $(a)$ follows by Lemma \ref{between if's}, $(b)$ is because the maximum index of $v_i$ is $k$, $(c)$ is because $v_{k-1} < 1 \leq v_k$, and $(d)$ follows by $\log(1+x)\approx x$ when $x$ is small.

Finally, we analyze the time periods after $t_{\mathrm{last}}$.

\begin{lemma}\label{after last if}
	$\mathcal{R}^{\pi^{\mathrm{shrinking}}}(T)[t_{\mathrm{last}},T]\leq C_3T^{(3+v)/4}\sqrt{\log T}$ for some universal constant $C_3\in(0,\infty)$. 
\end{lemma}

\begin{proof}[Proof of Lemma \ref{after last if}]
Because the \textbf{if} condition happens at $t_{\mathrm{last}}$, by Lemma \ref{if happens} either $i<\ell$ or $i=\ell$. Suppose $i<\ell$, then similar to Lemma \ref{between if's} we have
\begin{align*}
	\sum_{s=t_{\mathrm{last}}}^{T}|\hat{\mu}_s^{i}-\mu_s|&\overset{(a)}{\leq}\sum_{s=t_{\mathrm{last}}}^{T}|\hat{\mu}_s^{\ell}-\mu_s|+\sum_{s=t_{\mathrm{last}}}^{T}|\hat{\mu}_s^{i}-\hat{\mu}_s^{\ell}|\\
	&\overset{(b)}{\leq}\sum_{s=n_i+1}^{T}|\hat{\mu}_s^{\ell}-\mu_s|+\sum_{s=t_{\mathrm{last}}}^{T}|\hat{\mu}_s^{i}-\hat{\mu}_s^{\ell}|\\
	&\overset{(c)}{<}\left(\gamma\sqrt{\log T}+\sqrt{\kappa}\right)\cdot T^{(3+v_\ell)/4}+2\left(\gamma\sqrt{\log T}+\sqrt{\kappa}\right)\cdot T^{(3+v_\ell)/4}\\
	&\overset{(d)}{\leq}3e^{1/4}\left(\gamma\sqrt{\log T}+\sqrt{\kappa}\right)\cdot T^{(3+v)/4},
\end{align*}
where $(a)$ comes from the triangle inequality and $(b)$ is because $t_{\mathrm{last}}>T^{3/4}$ and $n_i\leq \kappa T^{1/2}$; the first part of $(c)$ follows by our assumption that $E_{\ell}$ does not occur, and the second part of $(c)$ follows since the \textbf{if} condition is never triggered after $t_{\mathrm{last}}$; $(d)$ follows by \cref{eqn:ell}. By Observation \ref{enough to estimate mean}, we get $\mathcal{R}^{\pi^{\mathrm{shrinking}}}(T)[t_{\mathrm{last}},T]\leq 3e^{1/4}C'\left(\gamma\sqrt{\log T}+\sqrt{\kappa}\right)\cdot T^{(3+v)/4}$ for some universal constant $C'=2\ell\in(0,\infty)$.

On the other hand, suppose $i=\ell$. Note $V_{\bm{\mu}}\leq T^v\leq T^{v_\ell}$ and after time $t_{\mathrm{last}}$ the Shrinking-Time-Window Policy just performs the Fixed-Time-Window Policy with variation parameter $v_\ell$, so by Lemma \ref{upper bound on regret: past-demand-only} $\mathcal{R}^{\pi^{\mathrm{shrinking}}}(T)[t_{\mathrm{last}},T]\leq C''T^{(3+v_\ell)/4}\leq C''e^{1/4}T^{(3+v)/4}$ for some universal constant $C''\in(0,\infty)$, where the last inequality again follows by \cref{eqn:ell}.

Taking $C_3=3e^{1/4}\max\{C'\left(\gamma+\sqrt{\kappa}\right),C''\}$ we get $\mathcal{R}^{\pi^{\mathrm{shrinking}}}(T)[t_{\mathrm{last}},T]\leq C_3T^{(3+v)/4}\sqrt{\log T}$ in both cases.  \end{proof}

Combining everything above, in the case where $E_{j}$ doesn't happen for any $j\geq \ell$, we get 
\begin{align*}
	\mathcal{R}^{\pi^{\mathrm{shrinking}}}(T)&=\mathcal{R}^{\pi^{\mathrm{shrinking}}}(T)[1,T^{3/4}]+\mathcal{R}^{\pi^{\mathrm{shrinking}}}(T)[T^{3/4}+1,t_{\mathrm{last}}-1]\\
	&\quad\text{ }+\mathcal{R}^{\pi^{\mathrm{shrinking}}}(T)[t_{\mathrm{last}},T]\\
	&\leq C_1T^{3/4}+C_2 T^{(3+v)/4}\log^{5/2} T+C_3T^{(3+v)/4}\sqrt{\log T}.
\end{align*}

Now let us consider the case where $E_{j}$ happens for some $j\geq \ell$. Because $\mathbb{P}\{E_j\}\leq\frac{2}{T^{3/2}}$ for each $j\geq \ell$, by union bound $$\mathbb{P}\left\{E_{j}\text{ happens for some }j\geq \ell\right\}\leq \frac{2}{T^{3/2}}\cdot|\mathcal{V}| \leq \frac{2}{T},$$ where the last inequality follows by $|\mathcal{V}|=k\approx \log^2 T$. By Observation \ref{regret doing nothing} we always have $\mathcal{R}^{\pi^{\mathrm{shrinking}}}(T)\leq c_4T$ for some universal constant $C_4\in(0,\infty)$. Therefore in summary we have
\begin{align*}
	\mathcal{R}^{\pi^{\mathrm{shrinking}}}(T)&\leq \mathbb{P}\{E_{j}\text{ doesn't happen for any }j\geq\ell\}\left(C_1T^{3/4}+C_2 T^{(3+v)/4}\log^{5/2} T+C_3T^{(3+v)/4}\sqrt{\log T}\right)\\
	&\quad +\mathbb{P}\left\{E_{j}\text{ happens for some } j\geq \ell\right\} C_4T\\
	&\leq C_1T^{3/4}+C_2 T^{(3+v)/4}\log^{5/2} T+C_3T^{(3+v)/4}\sqrt{\log T}+2C_4.
\end{align*}
Therefore, there exists some universal constant $C\in(0,\infty)$ such that $\mathcal{R}^{\pi^{\mathrm{shrinking}}}(T)\leq C T^{(3+v)/4}\log^{5/2} T$.

\end{proof}

    \vspace{1em}

\section{Proof of Proposition \ref{lower bound on regret: past-demand-only}, Proposition \ref{lower bound on regret}, and \cref{upper bound on regret: predictions-only}b)} \label{appD}

Note that Proposition \ref{lower bound on regret: past-demand-only} and Observation \ref{upper bound on regret: predictions-only}b) can be easily deduced from Proposition \ref{lower bound on regret} by setting $a=1$ and $v=1$ respectively, so it suffices to prove Proposition \ref{lower bound on regret}.

\begin{proof}[\textbf{Proof of Proposition \ref{lower bound on regret}.}]
We construct the following worst-case problem instance: divide the time horizon $T$ into cycles of length $T^{(1-v)/2}$ (since the analysis below is compatible with scaling, for simplicity we assume $T^{(1-v)/2}$ is an integer), so there are $T^{(1+v)/2}$ cycles. Assume that the demand distribution within each cycle is a Bernoulli distribution that equals to $1$ with probability $p$ and equals to $0$ with probability $1-p$. At the beginning of each cycle, we set the $p$ of the upcoming cycle to be either $\frac{1}{2}+\frac{1}{\sqrt{20}}T^{(v-1)/4}$ or $\frac{1}{2}-\frac{1}{\sqrt{20}}T^{(v-1)/4}$, each with probability $\frac{1}{2}$. Set $Q=\mathbb{R}_+$ and $b_t=h_t=1$ for all $t$, so the optimal ordering amount is the median of the demand distribution at each time.

First we show that the demand variation $V_{\bm{\mu}}$ is at most $T^v$. Note that since $\mu_t=p$ is fixed within each cycle, only the times between cycles contribute to the demand variation. Suppose the cycle changes between time $t$ and $t+1$, then $\left(\mu_{t+1}-\mu_t\right)^2\leq (\frac{2}{\sqrt{20}}T^{(v-1)/4})^2=\frac{1}{5}T^{(v-1)/2}$. Because there are $T^{(1+v)/2}$ cycles, $V_{\bm{\mu}}\leq T^{(1+v)/2}\cdot \frac{1}{5}T^{(v-1)/2}=\frac{1}{5}T^v$.

Then we add predictions into the instance. We divide the analysis into two cases.

Case 1:  $a\geq \frac{3+v}{4}$.

For each $t$ we set $a_t$ to be either $\frac{1}{2}+\frac{1}{\sqrt{20}}T^{(v-1)/4}$ or $\frac{1}{2}-\frac{1}{\sqrt{20}}T^{(v-1)/4}$, each with probability $\frac{1}{2}$. Then because each $a_t$ and $\mu_t$ are i.i.d., $a_t$ provides no information about $\mu_t$. Hence the predictions are useless in this instance. Note that $$\sum_{t=1}^{T}|a_t-\mu_t|\leq T\cdot \frac{2}{\sqrt{20}}T^{(v-1)/4}=\frac{1}{\sqrt{5}}T^{(3+v)/4}\leq \frac{1}{\sqrt{5}}T^{a},$$ so the prediction accuracy is within $T^a$.

Then we analyze the amount of regret incurred. For $i=1,\dots,T^{(1-v)/2}$, let $P_i$, $Q_i$ be i.i.d. distributions respectively, where $P_i\sim\text{Bernoulli }(\frac{1}{2}+\frac{1}{\sqrt{20}} T^{(v-1)/4})$ and $Q_i\sim\text{Bernoulli }(\frac{1}{2}-\frac{1}{\sqrt{20}} T^{(v-1)/4})$. Then the Kullback-Leibler divergence of $P_i$ from $Q_i$ is 
\begin{align*}
	D_{\text{KL}}\left(P_i\parallel Q_i\right)=&\left(\frac{1}{2}+\frac{1}{\sqrt{20}} T^{(v-1)/4}\right)\log\left(\frac{\frac{1}{2}+\frac{1}{\sqrt{20}} T^{(v-1)/4}}{\frac{1}{2}-\frac{1}{\sqrt{20}} T^{(v-1)/4}}\right)\\&+\left(\frac{1}{2}-\frac{1}{\sqrt{20}} T^{(v-1)/4}\right)\log\left(\frac{\frac{1}{2}-\frac{1}{\sqrt{20}} T^{(v-1)/4}}{\frac{1}{2}+\frac{1}{\sqrt{20}} T^{(v-1)/4}}\right).
\end{align*} 
We show that $D_{\text{KL}}\left(P_i\parallel Q_i\right)\leq \frac{13}{20} T^{(v-1)/2}$. Let $x=\frac{1}{\sqrt{20}} T^{(v-1)/4}$, then because $v\in[0,1]$, $x\in[0,\frac{1}{\sqrt{20}}]$. Note that the $D_{\text{KL}}\left(P_i\parallel Q_i\right)=13x^2$ for $x=0$, and we have
$$\frac{d}{dx}D_{\text{KL}}\left(P_i\parallel Q_i\right)\biggm\vert_{x=0}=0=\frac{d}{dx}13x^2\biggm\vert_{x=0}$$ and
$$\frac{d^2}{dx^2}D_{\text{KL}}\left(P_i\parallel Q_i\right)=\frac{1}{(x^2-0.25)^2}< 26=\frac{d^2}{dx^2}13x^2$$ for $x\in[0,\frac{1}{\sqrt{20}}]$. This shows $13x^2-D_{\text{KL}}\left(P_i\parallel Q_i\right)=0$ at $x=0$, the first derivative is $0$ at $x=0$, and the second derivative is non-negative for $x\in[0,\frac{1}{\sqrt{20}}]$. This implies $D_{\text{KL}}\left(P_i\parallel Q_i\right)\leq 13x^2=\frac{13}{20} T^{(v-1)/2}$.

Let $P=\sum_{i=1}^{T^{(1-v)/2}}P_i$ and $Q=\sum_{i=1}^{T^{(1-v)/2}}Q_i$ be the two possible demand distributions within a cycle, then because $P_i$'s are i.i.d. and $Q_i$'s are i.i.d., $$D_{\text{KL}}\left(P\parallel Q\right)=\sum_{i=1}^{T^{(1-v)/2}}D_{\text{KL}}\left(P_i\parallel Q_i\right)\leq T^{(1-v)/2}\cdot\frac{13}{20} T^{(v-1)/2}=\frac{13}{20}.$$ We claim that within each cycle, any attempt to distinguish between $P$ and $Q$ has at least a constant probability of making a mistake, i.e., one cannot effectively estimate the $p$ value within each cycle. Let $\mathcal{C}:\{0,1\}^{T^{(1-v)/2}}\to \{P,Q\}$ be any classifier that takes the demand observations within a cycle as inputs and determine the true demand distribution of this cycle. Let $E=\mathcal{C}^{-1}(P)$ be the event where $\mathcal{C}$ classifies the demand observations as from demand distribution $P$, then by Pinsker's inequality,
$$|P(E)-Q(E)|\leq\sqrt{\frac{1}{2}D_{\text{KL}}\left(P\parallel Q\right)}\leq\sqrt{\frac{13}{40}},$$ where $P(E)$ is the probability of $E$ happening under the condition that the true demand distribution is $P$, and the same for $Q(E)$. Therefore we have
\begin{eqnarray*}
	\mathbb{P}\left\{\mathcal{C}\text{ makes a mistake}\right\}&=&P(E^c)+Q(E)\\
	&\geq& P(E^c)+P(E)-\sqrt{\frac{13}{40}}\\
	&=&1-\sqrt{\frac{13}{40}}.
\end{eqnarray*}
Hence for any classifier $\mathcal{C}$ the probability of making the wrong guess of $p$ in the current cycle is at least $1-\sqrt{\frac{13}{40}}$. 

Note for demand distribution $P_i$, $q^*=1$ and $C(\mu_{P_i},0)-C(\mu_{P_i},1)=\frac{1}{2}+\frac{1}{\sqrt{20}} T^{(v-1)/4}-(1-(\frac{1}{2}+\frac{1}{\sqrt{20}}T^{(v-1)/4}))=\frac{1}{\sqrt{5}}T^{(v-1)/4}$, and similarly for demand distribution $Q_i$, $q^*=0$ and $C(\mu_{Q_i},1)-C(\mu_{Q_i},0)=(1-(\frac{1}{2}-\frac{1}{\sqrt{20}} T^{(v-1)/4}))-(\frac{1}{2}-\frac{1}{\sqrt{20}} T^{(v-1)/4})=\frac{1}{\sqrt{5}} T^{(v-1)/4}$. Therefore each wrong guess of $p$ incurs a difference between $C(\mu_t, q_t)$ and $C(\mu_t, q^*_t)$ by $\frac{1}{\sqrt{5}} T^{(v-1)/4}$, which  incurs a regret of $C(\mu_t,q_t)-C(\mu_t,q^*_t)=\frac{1}{\sqrt{5}}  T^{(v-1)/4}.$ Therefore over $T$ time periods the expected total regret of any General Policy $\pi$ satisfies \begin{eqnarray*}
	\mathcal{R}^{\pi}(T)&\geq& \left(1-\sqrt{\frac{13}{40}}\right)\frac{1}{\sqrt{5}}T^{(v-1)/4}\cdot T\\&=&\left(1-\sqrt{\frac{13}{40}}\right)\frac{1}{\sqrt{5}} T^{(3+v)/4}\\&=&\left(1-\sqrt{\frac{13}{40}}\right)\frac{1}{\sqrt{5}} T^{\min\{(3+v)/4,a\}},
\end{eqnarray*} where the last equality follows from the assumption that $a\geq \frac{3+v}{4}$. Take $c=\left(1-\sqrt{\frac{13}{40}}\right)\frac{1}{\sqrt{5}}$ gives the desired result.

Case 2: $a< \frac{3+v}{4}$.

		 For the first $T^{a-(1+3v)/4}$ time periods of each cycle, we set $a_t$ to be either $\frac{1}{2}+\frac{1}{\sqrt{20}}T^{(v-1)/4}$ or $\frac{1}{2}-\frac{1}{\sqrt{20}}T^{(v-1)/4}$, each with probability $\frac{1}{2}$. {Note that $a<\frac{3+v}{4}$ implies $a-\frac{1+3v}{4}<\frac{1-v}{2}$, so time periods of length ${T^{a-(1+3v)/4}}$ are indeed contained in each cycle, which has length $T^{(1-v)/2}$.} For the other time periods we set $a_t=\mu_t$. Then, similar as in the case above, the predictions are useless for the first ${T^{a-(1+3v)/4}}$ time periods of each cycle in this instance. Since there are $T^{(1+v)/2}$ number of cycles, $$\sum_{t=1}^{T}|a_t-\mu_t|\leq T^{(1+v)/2}\cdot {T^{a-(1+3v)/4}}\cdot \frac{2}{\sqrt{20}}T^{(v-1)/4}=\frac{1}{\sqrt{5}}T^{a},$$ so the prediction accuracy is within $T^a$.

Again, the same analysis as the case above shows that for the first ${T^{a-(1+3v)/4}}$ time periods of each cycle, the probability of making the wrong guess of $p$ in the current cycle is at least $1-\sqrt{\frac{13}{40}}$. Also, each wrong guess of $p$ incurs a regret of $ \frac{1}{\sqrt{5}}T^{(v-1)/4}$. Because there are $T^{(1+v)/2}$ number of cycles, over $T$ time periods the expected total regret of any General Policy $\pi$ satisfies \begin{align*}
	\mathcal{R}^{\pi}(T)&\geq \left(1-\sqrt{\frac{13}{40}}\right)\frac{1}{\sqrt{5}}T^{(v-1)/4}\cdot T^{(1+v)/2}\cdot T^{a-\frac{1+3v}{4}}\\&=\left(1-\sqrt{\frac{13}{40}}\right)\frac{1}{\sqrt{5}} T^{a}\\&=\left(1-\sqrt{\frac{13}{40}}\right)\frac{1}{\sqrt{5}} T^{\min\{(3+v)/4,a\}},
\end{align*} where the last equality follows from the assumption that $a<\frac{3+v}{4}$. Take $c=\left(1-\sqrt{\frac{13}{40}}\right)\frac{1}{\sqrt{5}}$ gives the desired result.
\end{proof}

\vspace{1em}

\section{General Variation}\label{sec:general-variation}
Recall that for any sequence of means $\bm{\mu}=\{\mu_1,\dots,\mu_T\}$, we define the demand variation to be its \textit{quadratic variation} $$V_{\bm\mu}=\max_{\{t_0,\dots, t_K\}\in\mathcal{P}}\left\{\sum_{k=1}^{K}\left|\mu_{t_k}-\mu_{t_{k-1}}\right|^2\right\},$$
where $\mathcal{P}$ is the set of all partitions. The choice of quadratic variation follows from previous literature \cite{keskin2017chasing,keskin2021nonstationary}. We can also define the demand variation using what we will call {\em $\theta$-variation} for some $0\leq {\theta}<\infty$: $$V_{\bm\mu}=\max_{\{t_0,\dots, t_K\}\in\mathcal{P}}\left\{\sum_{k=1}^{K}\left|\mu_{t_k}-\mu_{t_{k-1}}\right|^{\theta}\right\}.$$ In the special case of ${\theta}=1$, this is the \textit{total variation}, which has also been used extensively in previous literature \cite{besbes2015non,karnin2016multi,luo2018efficient,cheung2022hedging}. We prove the following lower bound, which generalizes \cref{lower bound on regret: past-demand-only}.

\begin{proposition}
	[Lower Bound: Nonstationary Newsvendor with $\theta$-variation]\label{lower bound on regret: lq}
	Suppose the demand variation is defined using ${\theta}$-variation and $V_{\bm\mu}\leq T^v$. For any variation parameter $v\in[0,1]$, and any policy $\pi$ (which may depend on {the knowledge of} $v$), we have $$\mathcal{R}^{\pi}(T)\geq cT^{(1+\theta+v)/(2+\theta)},$$ where $c>0$ is a universal constant.
\end{proposition}

\begin{proof}[\textbf{Proof of Proposition \ref{lower bound on regret: lq}.}]
Similar to the proof of \cref{lower bound on regret: past-demand-only}, we construct the following worst-case problem instance: divide the time horizon $T$ into cycles of length $T^{(2-2v)/(2+{\theta})}$ (since the analysis below is compatible with scaling, for simplicity we assume $T^{(2-2v)/(2+{\theta})}$ is an integer), so there are $T^{( \theta+2v)/(2+ \theta)}$ cycles. Assume that the demand distribution within each cycle is a Bernoulli distribution that equals to $1$ with probability $p$ and equals to $0$ with probability $1-p$. At the beginning of each cycle, we set the $p$ of the upcoming cycle to be either $\frac{1}{2}+\frac{1}{\sqrt{20}}T^{(v-1)/(2+ \theta)}$ or $\frac{1}{2}-\frac{1}{\sqrt{20}}T^{(v-1)/(2+ \theta)}$, each with probability $\frac{1}{2}$. Set $Q=\mathbb{R}_+$ and $b_t=h_t=1$ for all $t$, so the optimal ordering amount is the median of the demand distribution at each time.

First we show that the demand variation $V_{\bm{\mu}}$ is at most $T^v$. Note that since $\mu_t=p$ is fixed within each cycle, only the times between cycles contribute to the demand variation. Suppose the cycle changes between time $t$ and $t+1$, then $\left(\mu_{t+1}-\mu_t\right)^ \theta\leq (\frac{2}{\sqrt{20}}T^{(v-1)/(2+ \theta)})^ \theta\leq T^{ \theta(v-1)/(2+ \theta)}$. Because there are $T^{( \theta+2v)/(2+ \theta)}$ cycles, $V_{\bm{\mu}}\leq T^{( \theta+2v)/(2+ \theta)}\cdot T^{ \theta(v-1)/(2+ \theta)}=T^v$. 

Then, following the calculations in the proof of \cref{lower bound on regret}, the probability of making the wrong guess of $p$ in the current cycle is at least $1-\sqrt{\frac{13}{40}}$. Also, each wrong guess of $p$ incurs a regret of $ \frac{1}{\sqrt{5}}T^{(v-1)/(2+ \theta)}$ at each time period. Therefore over $T$ time periods the expected total regret of any General Policy $\pi$ satisfies \begin{align*}
	\mathcal{R}^{\pi}(T)&\geq \left(1-\sqrt{\frac{13}{40}}\right)\frac{1}{\sqrt{5}}T^{(v-1)/(2+ \theta)}\cdot T=\left(1-\sqrt{\frac{13}{40}}\right)\frac{1}{\sqrt{5}}T^{(1+ \theta+v)/(2+ \theta)}.
\end{align*} Take $c=\left(1-\sqrt{\frac{13}{40}}\right)\frac{1}{\sqrt{5}}$ gives the desired result.
\end{proof}

    \vspace{1em}

\section{Proof of Proposition \ref{lower bound on regret unknown parameters}} \label{appE}
\begin{proof}[\textbf{Proof of Proposition \ref{lower bound on regret unknown parameters}.}]
Set $Q=\mathbb{R}_+$ and $b_t=h_t=1$ for all $t$, so the optimal ordering amount is the median of the demand distribution at each time. We construct the following two problem instances: 

Instance 1: for all $t$, set $D^{(1)}_t\sim Unif(0,2)$. Then $\mu^{(1)}_t=1$. Set $a^{(1)}_t=d^{(1)}_t$, where $d^{(1)}_t$ is the realization of $D^{(1)}_t$. Since $\mu^{(1)}_t=1$ for all $t$, the variation parameter $v^{(1)}=0$. Because  $a^{(1)}_t$ is a constant away from $\mu_t^{(1)}$ with probability 1, $a^{(1)}=1$. Because the optimal order amount is $q_t^{(1)*}=\mu_t^{(1)}$, $C(\mu_t^{(1)}, q_t^{(1)*})=\mathbb{E}\left[b_t(D^{(1)}_t-\mu_t^{(1)})^{+}+h_t(\mu_t^{(1)}-D^{(1)}_t)^{+}\right]=\mathbb{E}\left[|D^{(1)}_t-1|\right]=\frac{1}{2}$.

Instance 2: for all $t$, set $D^{(2)}_t=\mu_t^{(2)}\sim Unif(0,2)$. i.e., $D_t^{(2)}$ is a one-point distribution. Set $a_t^{(2)}=d_t^{(2)}$, where $d_t^{(2)}$ is the realization of $D_t^{(2)}$. Since $\mu_t^{(2)}$ changes by a constant amount from $\mu_{t-1}^{(2)}$ with probability 1 throughout $t=2,\dots, T$, the variation parameter $v^{(2)}=1$. Because $a_t^{(2)}=\mu_t^{(2)}$ for every $t$, $a^{(2)}=0$. Because the optimal order amount is $q_t^{(2)*}=\mu_t^{(2)}$, $C(\mu_t^{(2)}, q_t^{(2)*})=\mathbb{E}\left[b_t(D^{(2)}_t-\mu_t^{(2)})^{+}+h_t(\mu_t^{(2)}-D^{(2)}_t)^{+}\right]=0$.

Note that a General Policy can only observe information on $a_t$'s and $d_t$'s. Because $a_t^{(1)}=d_t^{(1)}$ and $a_t^{(2)}=d_t^{(2)}$ have the same distribution for every $t$, no General Policies can distinguish between the two instances. For any General Policy $\pi$, let $q_t$ be its output at time $t$. Without loss of generality, we may assume $q_t\in[0,2]$ since any other ordering amount is clearly sub-optimal. Then $$C(\mu_t^{(1)}, q_t)=\mathbb{E}\left[b_t(D^{(1)}_t-q_t)^{+}+h_t(q_t-D^{(1)}_t)^{+}\right]=\frac{q_t^2+(2-q_t)^2}{4}\geq\frac{1}{2}$$ and $$C(\mu_t^{(2)}, q_t)=\mathbb{E}\left[b_t(D^{(2)}_t-q_t)^{+}+h_t(q_t-D^{(2)}_t)^{+}\right]=\frac{2-q_t}{2}+\frac{q_t}{2}=1.$$ Let $\mathcal{R}^{\pi}[I_1](T)$ denote the regret of $\pi$ on instance 1 and $\mathcal{R}^{\pi}[I_2](T)$ denote the regret of $\pi$ on instance 2, then \begin{eqnarray*}
	\mathcal{R}^{\pi}[I_1](T)+
	\mathcal{R}^{\pi}[I_2](T)&=&
	\sum_{t=1}^T\left(C(\mu_t^{(1)},q_t)-C(\mu_t^{(1)}, q_t^{(1)*})\right)+\sum_{t=1}^T\left(C(\mu_t^{(2)},q_t)-C(\mu_t^{(2)}, q_t^{(2)*})\right)\\ &\geq&\sum_{t=1}^T\left(\frac{1}{2}-\frac{1}{2}\right)+\sum_{t=1}^T\left(1-0\right)\\
	&=& T.
\end{eqnarray*}
Because $\mathcal{R}^{\pi}[I_1](T)+
\mathcal{R}^{\pi}[I_2](T)\geq T$, any General Policy $\pi$ incurs regret at least $0.5T$ on at least one of the instances. Since no General Policy can distinguish between the two instances, we can always choose the worse one of the two instances to feed to the policy. Also, since $v^{(1)}=0$, $a^{(1)}=1$, $v^{(2)}=1$, $a^{(2)}=0$, in both instances we have $a\neq\frac{3+v}{4}$. Therefore for any General Policy $\pi$ there always exists a problem instance with $a\neq\frac{3+v}{4}$ such that $\mathcal{R}^{\pi}(T)\geq 0.5T=\Omega(T^{\max\{(3+v)/4,a\}})$ on the instance.
\end{proof}

	\vspace{1em}
\section{Proof of Theorem \ref{upper bound on regret}}\label{Appendix F}
\begin{proof}[\textbf{Proof of Theorem \ref{upper bound on regret}.}] 	Because our bounds are all asymptotic, we ignore the rounding and write $n= \kappa T^{(1-v) / 2}$ to simplify the notation.

Because $\pi^{\mathrm{PERP}}_t=\pi^{a}_t$ for $t$ from $1$ to $n$, by Observation \ref{upper bound on regret: predictions-only} $\mathcal{R}^{\pi^{\mathrm{PERP}}}(T)[1,n]=\mathcal{R}^{\pi^{a}}(T)[1,n]\leq C_1'T^a$ for some universal constant $C_1'\in(0,\infty)$. Also, by Observation \ref{regret doing nothing}, since $n=\kappa T^{(1-v) / 2}<\kappa T^{(3+v) / 4}$, there exists some universal constant $C_1''\in(0,\infty)$ such that $\mathcal{R}^{\pi^{\mathrm{PERP}}}(T)[1,n]\leq C_1''T^{(3+v)/4}$. Take $C_1=\max\{C_1',C_1''\}$ we get $\mathcal{R}^{\pi^{\mathrm{PERP}}}(T)[1,n]\leq C_1 T^{\min\{(3+v)/4,a\}}\sqrt{\log T}.$

Now we consider the time periods after time $n+1$. First, same as in the proof of Lemma \ref{error in noise}, by Hoeffding's inequality for any $n+1 \leq t \leq T$ we have
$$
\mathbb{P}\left\{\left|\frac{1}{n} \sum_{s=t-n}^{t-1} \epsilon_s\right| \geq x\right\} \leq 2 \exp \left(-\frac{\rho(n x)^2}{\sum_{s=t-n}^{t-1} \delta_s^2}\right) \leq 2 \exp \left(-\frac{\rho n}{\delta^2} x^2\right),
$$
where $\rho$ and $\delta$ are the same as in the previous proof. Set $\gamma>0$ to be large enough so that 
\begin{equation} \label{eqn:gamma2} \frac{\rho\kappa \gamma^2}{\delta^2}\geq 2.\end{equation}  Take $x=\gamma \sqrt{\log T} \cdot T^{(v-1) / 4}$ and plug in $n=\kappa T^{(1-v) / 2}$ yields
$$
\mathbb{P}\left\{\left|\frac{1}{n} \sum_{s=t-n}^{t-1} \epsilon_s\right| \geq \gamma \sqrt{\log T} \cdot T^{(v-1) / 4}\right\} \leq 2 \exp \left(-\frac{\rho \kappa \gamma^2}{\delta^2} \log T\right) \leq \frac{2}{T^2} .
$$
Then we get
\begin{eqnarray*}
	\mathbb{P}\left\{\sum_{t=n+1}^T\left|\frac{1}{n} \sum_{s=t-n}^{t-1} \epsilon_s\right| \geq \gamma \sqrt{\log T} \cdot T^{(3+v) / 4}\right\} &\leq& \mathbb{P}\left\{\max _{n+1 \leq t \leq T}\left|\frac{1}{n} \sum_{s=t-n}^{t-1} \epsilon_s\right| \geq \gamma \sqrt{\log T} \cdot T^{(v-1) / 4}\right\} \\
	&\stackrel{(a)}{\leq}& T \cdot \mathbb{P}\left\{\left|\frac{1}{n} \sum_{s=t-n}^{t-1} \epsilon_s\right| \geq \gamma \sqrt{\log T} \cdot T^{(v-1) / 4}\right\} \\
	&\leq& \frac{2}{T},
\end{eqnarray*}
where (a) follows by union bound. Note Lemma \ref{error in mean} says $\sum_{t=n+1}^T\left|\frac{1}{n} \sum_{s=t-n}^{t-1}\left(\mu_s-\mu_t\right)\right| \leq \sqrt{\kappa} T^{(3+v) / 4}$, and in the proof of Lemma \ref{upper bound on regret: past-demand-only} we have
$$
\sum_{t=n+1}^T\left|\hat{\mu}_t^{\mathrm{fixed}}-\mu_t\right| \leq \sum_{t=n+1}^T\left|\frac{1}{n} \sum_{s=t-n}^{t-1}\left(\mu_s-\mu_t\right)\right|+\sum_{t=n+1}^T\left|\frac{1}{n} \sum_{s=t-n}^{t-1} \epsilon_s\right|.
$$
Therefore \begin{align*}
	&\quad\mathbb{P}\left\{\sum_{t=n+1}^{T}\left|\hat{\mu}_t^{\mathrm{fixed}}-\mu_t\right|\geq \left(\gamma\sqrt{\log T}+\sqrt{\kappa}\right)\cdot T^{(3+v)/4}\right\}\\
	&\leq \mathbb{P}\left\{\sum_{t=n+1}^{T}\left|\frac{1}{n} \sum_{s=t-n}^{t-1} (\mu_s-\mu_t)\right| +\sum_{t=n+1}^{T}\left|\frac{1}{n}\sum_{s=t-n}^{t-1}\epsilon_s\right|\geq \left(\gamma\sqrt{\log T}+\sqrt{\kappa}\right)\cdot T^{(3+v)/4}\right\}\\
	&\leq \mathbb{P}\left\{\sum_{t=n+1}^{T}\left|\frac{1}{n}\sum_{s=t-n}^{t-1}\epsilon_s\right|\geq \gamma\sqrt{\log T}\cdot T^{(3+v)/4}\right\}\\
	&\leq  \frac{2}{T}.
\end{align*}

Because once the \textbf{if} condition in PERP is triggered we break the \textbf{for} loop, the \textbf{if} condition can happen at most once thoughout the algorithm. We consider two cases separately depending on whether the \textbf{if} condition happens or not:

Case 1: the \textbf{if} condition happens at some time $s$.
	
	First, suppose $\sum_{t=n+1}^T\left|\hat{\mu}_t^{\mathrm{fixed}}-\mu_t\right|<(\gamma \sqrt{\log T}+\sqrt{\kappa}) \cdot T^{(3+v) / 4}$, then because the \textbf{if} condition happens at time $s$,
	\begin{eqnarray*}
		T^a & =&\sum_{t=1}^T\left|a_t-\mu_t\right| \\
		& \stackrel{(a)}{\geq}& \sum_{t=1}^s\left|\hat{\mu}_t^{a}-\mu_t\right| \\
		& \stackrel{(b)}{\geq}& \sum_{t=1}^s\left|\hat{\mu}_t^{a}-\hat{\mu}_t^{\mathrm{fixed}}\right|-\sum_{t=1}^s\left|\hat{\mu}_t^{\mathrm{fixed}}-\mu_t\right| \\
		& \stackrel{(c)}{\geq}& T^{(3+v) / 4},
	\end{eqnarray*}
	where $a_t=\hat{\mu}_t^{a}$ gives $(a),(b)$ follows by triangle inequality, and $(c)$ follows by the the \textbf{if} condition. This shows $a \geq \frac{3+v}{4}$, so $\min \left\{\frac{3+v}{4}, a\right\}=\frac{3+v}{4}$.
	Note that between time $n+1$ and time $s-1$ we have:
	\begin{eqnarray*}
		\sum_{t=n+1}^{s-1}\left|\hat{\mu}_t^{a}-\mu_t\right| & \stackrel{(a)}{\leq}& \sum_{t=n+1}^{s-1}\left|\hat{\mu}_t^{a}-\hat{\mu}_t^{\mathrm{fixed}}\right|+\sum_{t=n+1}^{s-1}\left|\hat{\mu}_t^{\mathrm{fixed}}-\mu_t\right| \\
		& \stackrel{(b)}{\leq}&(2 \gamma \sqrt{\log T}+2 \sqrt{\kappa}+1) \cdot T^{(3+v) / 4},
	\end{eqnarray*}
	where $(a)$ follows by triangle inequality and $(b)$ follows by the \textbf{if} condition (note by the algorithm's construction $s \geq n+1$; for the simplicity of writing we assume $s-1 \geq n+1$, and the case $s=n+1$ follows similarly). Then by Observation \ref{enough to estimate mean}, let $\ell$ be the Lipschitz constant, we have
	\begin{eqnarray*}
		\mathcal{R}^{\pi^{\mathrm{PERP}}}(T)[n+1, s-1] & =&\sup_{\bm{D}\in \mathcal{D}(v)}\mathbb{E}_{D, a}^{\pi^{\mathrm{PERP}}}\left\{\sum_{t=n+1}^{s-1} \left(C\left(\mu_t, q_t\right)-C\left(\mu_t, q_t^*\right)\right)\right\} \\
		& \stackrel{(a)}{\leq}& 2 \ell \cdot \sup_{\bm{D}\in \mathcal{D}(v)}\mathbb{E}_{D, a}^{\pi^{\mathrm{PERP}}}\left\{\sum_{t=n+1}^{s-1}\left|\hat{\mu}_t^{a}-\mu_t\right|\right\} \\
		& \leq& 2 \ell (2 \gamma \sqrt{\log T}+2 \sqrt{\kappa}+1) \cdot T^{(3+v) / 4},
	\end{eqnarray*}
	where $(a)$ is because $\pi^{\mathrm{PERP}}=\pi^{a}$ before time $s$. Hence $\mathcal{R}^{\pi^{\mathrm{PERP}}}(T)[n+1, s-1]$ is on the order of $ T^{(3+v) / 4}\sqrt{\log T}$, so there exists some universal constant $C_2 \in(0, \infty)$ such that $\mathcal{R}^{\pi^{\mathrm{PERP}}}(T)[n+1, s-1] \leq$ $C_2 T^{(3+v) / 4}\sqrt{\log T}$. Also, since after time $s$ we have $\pi^{\mathrm{PERP}}=\pi^{\mathrm{fixed}}$, by Lemma \ref{upper bound on regret: past-demand-only} there exists some universal constant $C_3 \in(0, \infty)$ such that
	$$
	\mathcal{R}^{\pi^{\mathrm{PERP}}}(T)[s, T]=\mathcal{R}^{\pi^{\mathrm{fixed}}}(T)[s, T] \leq \mathcal{R}^{\pi^{\mathrm{fixed}}}(T) \leq C_3  T^{(3+v) / 4}\sqrt{\log T}.
	$$
	
	In summary, if $\sum_{t=n+1}^T\left|\hat{\mu}_t^{\mathrm{fixed}}-\mu_t\right|<(\gamma \sqrt{\log T}+\sqrt{\kappa}) \cdot T^{(3+v) / 4}$, we have
	$$
	\begin{aligned}
		\mathcal{R}^{\pi^{\mathrm{PERP}}}(T) & =\mathcal{R}^{\pi^{\mathrm{PERP}}}(T)[1, n]+\mathcal{R}^{\pi^{\mathrm{PERP}}}(T)[n+1, s-1]+\mathcal{R}^{\pi^{\mathrm{PERP}}}(T)[s, T] \\
		& \leq\left(C_1+C_2+C_3\right) T^{(3+v) / 4}\sqrt{\log T} \\
		& =\left(C_1+C_2+C_3\right) T^{\min \{(3+v) / 4, a\}}\sqrt{\log T},
	\end{aligned}
	$$
	where the last equality is because $\min \left\{\frac{3+v}{4}, a\right\}=\frac{3+v}{4}$.
	
	Second, suppose $\sum_{t=n+1}^T\left|\hat{\mu}_t^{\text {fixed}}-\mu_t\right| \geq(\gamma \sqrt{\log T}+\sqrt{\kappa}) \cdot T^{(3+v) / 4}$. By Observation \ref{regret doing nothing} there exists some universal constant $C_4 \in(0, \infty)$ such that 
	$
	\mathcal{R}^{\pi^{\mathrm{PERP}}}(T)\leq C_4T .
	$
	Therefore combining the above two scenarios we get
	$$
	\begin{aligned}
		\mathcal{R}^{\pi^{\mathrm{PERP}}}(T) \leq & \mathbb{P}\left\{\sum_{t=n+1}^T\left|\hat{\mu}_t^{\mathrm{fixed}}-\mu_t\right|<(\gamma \sqrt{\log T}+\sqrt{\kappa}) \cdot T^{(3+v) / 4}\right\}\left(C_1+C_2+C_3\right) T^{\min \{(3+v) / 4, a\}}\sqrt{\log T} \\
		& +\mathbb{P}\left\{\sum_{t=n+1}^T\left|\hat{\mu}_t^{\mathrm{fixed}}-\mu_t\right| \geq(\gamma \sqrt{\log T}+\sqrt{\kappa}) \cdot T^{(3+v) / 4}\right\} C_4 T \\
		\leq & \left(C_1+C_2+C_3\right) T^{\min \{(3+v) / 4, a\}}\sqrt{\log T} +\frac{2}{T} \cdot C_4 T \\
		= & \left(C_1+C_2+C_3\right) T^{\min \{(3+v) / 4, a\}}\sqrt{\log T}+2C_4.
	\end{aligned}
	$$
	Hence $\mathcal{R}^{\pi^{\mathrm{PERP}}}(T) \leq C T^{\min \{(3+v) / 4, a\}}\sqrt{\log T}$ for some universal constant $C \in(0, \infty)$.
	
Case 2: the \textbf{if} condition does not happen.
	
	In this case $\pi^{\mathrm{PERP}}=\pi^{a}$, so by Observation \ref{upper bound on regret: predictions-only}a) we immediately have $\mathcal{R}^{\pi^{\mathrm{PERP}}}(T) \leq C_5  T^a\sqrt{\log T}$ for some universal constant $C_5 \in(0, \infty)$. Also, following the same analysis as the part of Case 1 where an \textbf{if} condition has not happened, i.e., between time $n+1$ and time $s-1$, we get $\mathcal{R}^{\pi^{\mathrm{PERP}}}(T)[n+1, T] \leq$ $2 l (2 \gamma \sqrt{\log T}+2 \sqrt{\kappa}+1) \cdot T^{(3+v) / 4} \leq C_6 T^{(3+v) / 4}\sqrt{\log T}$ for some universal constant $C_6 \in(0, \infty)$. Then we have $\mathcal{R}^{\pi^{\mathrm{PERP}}}(T)=\mathcal{R}^{\pi^{\mathrm{PERP}}}(T)[1, n]+\mathcal{R}^{\pi^{\mathrm{PERP}}}(T)[n+1, T] \leq\left(C_1+C_6\right) T^{(3+v) / 4}\sqrt{\log T}$. Hence take $C=\max \left\{C_5, C_1+C_6\right\}$ we have $\mathcal{R}^{\pi^{\mathrm{PERP}}}(T) \leq C T^{\min \{(3+v) / 4, a\}}\sqrt{\log T}$.

\end{proof}

\section{Unknown $v$ and Known $a$} \label{appFinal}

We design a policy for the case of unknown $v$ and known $a$. Our policy utilizes the famous Exp3 algorithm (Exponential-weight Algorithm for Exploration and Exploitation) as a subroutine. For the sake of completeness we state Exp3 in our setting and its regret bound below. One can refer to \cite{auer2002nonstochastic} for a more detailed discussion on Exp3.

\begin{algorithm}
	\caption{Exp3 (with $\pi^{\mathrm{shrinking}}$ and $\pi^{\mathrm{prediction}}$)}
	\setstretch{1.5}
	{\bf Inputs:} $\pi^{\mathrm{shrinking}}$ from Algorithm \ref{alg:shrinking} and $\pi^{\mathrm{prediction}}$ from Algorithm \ref{alg:prediction}
	
	{\bf Initialization:} $C_{\max}=\sup_{\mu_t\in [\mu_{\min},\mu_{\max}],q_t\in Q}C_t(\mu_t,q_t)$, $w^{\mathrm{shrinking}}(1)=w^{\mathrm{prediction}}(1)=1$, and parameter $\gamma=\min\left\{1,\sqrt{\frac{2\ln 2}{(e-1)C_{\max}T}}\right\}$\;
	
	\For{$t=1,\ldots,T$}{
	$p^{\mathrm{shrinking}}(t)\gets(1-\gamma)\frac{w^{\mathrm{shrinking}}(t)}{w^{\mathrm{shrinking}}(t)+w^{\mathrm{prediction}}(t)}+\frac{\gamma}{2}$ and $p^{\mathrm{prediction}}(t)\gets(1-\gamma)\frac{w^{\mathrm{prediction}}(t)}{w^{\mathrm{shrinking}}(t)+w^{\mathrm{prediction}}(t)}+\frac{\gamma}{2}$\; 
	$\pi_t\gets \pi^{\mathrm{shrinking}}_t$ with probability $p^{\mathrm{shrinking}}(t)$ and $\pi_t\gets \pi^{\mathrm{prediction}}_t$ with probability $p^{\mathrm{prediction}}(t)$\;
	Obtain $d_t$\;
	\If{$\pi_t\gets \pi^{\mathrm{shrinking}}_t$}{
		{$\delta_t=(b_t(d_t-q^{\mathrm{shrinking}}_t)^{+}+h_t(q^{\mathrm{shrinking}}_t-d_t)^{+})/p^{\mathrm{shrinking}}(t)$\;
			$w^{\mathrm{shrinking}}(t+1)=w^{\mathrm{shrinking}}(t)\exp{(-\gamma \delta_t/2)}$\;}}
	\Else{{$\delta_t=(b_t(d_t-q^{\mathrm{prediction}}_t)^{+}+h_t(q^{\mathrm{prediction}}_t-d_t)^{+})/p^{\mathrm{prediction}}(t)$\;
			$w^{\mathrm{prediction}}(t+1)=w^{\mathrm{prediction}}(t)\exp{(-\gamma \delta_t/2)}$.}}}	
	\label{alg:Exp3}
\end{algorithm}

\begin{proposition}[Corollary 3.2 in \cite{auer2002nonstochastic}] \label{prop:Exp3}
	\text{Exp3} achieves worst-case regret $$\mathcal{R}^{\mathrm{Exp3}}(T)\leq\min\{\mathcal{R}^{\pi^{\mathrm{shrinking}}}(T),\mathcal{R}^{\pi^{\mathrm{prediction}}}(T)\}+2\sqrt{2(\ln 2) (e-1)C_{\max}}\cdot\sqrt{T}.$$
\end{proposition}

We refer the readers to \cite{auer2002nonstochastic} for the proof. Note that Exp3 incurs an additive $\sqrt{T}$ regret on top of $T^{\min\{(3+v)/4,a\}}$, which is a lower order term if $a>\frac{1}{2}$. On the other hand, if $a\leq \frac{1}{2}$, we have $T^{\min\{(3+v)/4,a\}}=T^a$, so we can simply apply the Prediction Policy. This idea gives the policy of unknown $v$ and known $a$.

\begin{algorithm} 
	\caption{Divide-Into-Cases Policy}
	\setstretch{1.5}
	{\bf Inputs:} accuracy parameter $a \in [0,1]$, $\pi^{\mathrm{shrinking}}$ from Algorithm \ref{alg:shrinking}, and $\pi^{\mathrm{prediction}}$ from Algorithm \ref{alg:prediction}
	
	\If{$a\leq \frac{1}{2}$}{$\pi\gets \pi^{\mathrm{prediction}}$ \;}
	\Else{$\pi\gets \pi^{\mathrm{Exp3}}.$}
	\label{alg:divide}
\end{algorithm}

\begin{observation}[Upper Bound: Unknown $v$ and Known $a$]\label{upper bound on regret appendix}
	For any variation parameter $v\in[0,1]$ and any accuracy parameter $a\in[0,1]$, the Divide-Into-Cases Policy ${\pi^{\mathrm{Divide}}}$ achieves worst-case regret $$\mathcal{R}^{\pi^{\mathrm{Divide}}}(T)\leq C T^{\min\{(3+v)/4,a\}} \log^{5/2} T,$$ where $C$ is a universal constant.
\end{observation}

\begin{proof}[\textbf{Proof of Theorem \ref{upper bound on regret appendix}.}]
If $a\leq \frac{1}{2}$, then $T^{\min\{(3+v)/4,a\}}=T^a$. The result follows from \cref{upper bound on regret: predictions-only}. If $a> \frac{1}{2}$, then $\sqrt{T}=O(T^{\min\{(3+v)/4,a\}})$. The result follows from Proposition $\ref{prop:Exp3}$ and \cref{upper bound on regret: past-demand-only with unknown variation}.
\end{proof}

\section{Experiment Details} \label{app:experiment}

In the synthetic experiment we used triple exponential smoothing (Holt Winters) to generate the demand sequences. Triple exponential smoothing takes in the following parameters: data smoothing factor $\alpha\in[0,1]$, trend smoothing factor $\beta\in[0,1]$, seasonal change smoothing factor $\gamma\in[0,1]$, and season length $L\in\mathbb{N}$. Given historical observations $x_1,\dots, x_t (t\geq L)$, triple exponential smoothing outputs $\hat{x}_{t+m}$ for $m\geq 1$, which is an estimate of $x_{t+m}$, according to the following formula:

\begin{eqnarray*}
	s_0&=&x_0\\
	s_t&=&\alpha\frac{x_t}{x_{t-L}}+(1-\alpha)(s_{t-1}+b_{t-1})\\
	b_t&=&\beta(s_t-s_{t-1})+(1-\beta)b_{t-1}\\
	c_t&=&\gamma\frac{x_t}{s_t}+(1-\gamma)c_{t-L}\\
	x_{t+m}&=&(s_t+mb_t)c_{t-L+1+(m-1)\text{mod} L},
\end{eqnarray*}
where $s_t$ represent the smoothed value of the constant part for time $t$, $b_t$ is the sequence of best estimates of the linear trend that are superimposed on the seasonal changes, and $c_t$ is the sequence of seasonal correction factors.

In our experiment, we first generated a demand sequence for 30 time periods  where the demand at each time period is drawn uniformly between 80 and 120. This was treated as the historical observations and was fixed throughout the experiment. Then for each set of parameters $(\alpha,\beta,\gamma, L)$, which we will specify later in each case, we used triple exponential smoothing to generate the mean of demands for 365 time periods, where at each time period we also added a random Gaussian noise with mean equals to 0 and variance equals to 5. Finally the true demand at each time period was generated as a Poisson variable with the corresponding mean.

We ran two sets of experiments:%did the following to fix one of the two parameters ($v$ or $a$).
\begin{itemize}
	\item \textbf{Fixed $v$:} We fixed a single set of parameters $(\alpha,\beta,\gamma, L) = (0.5,0.5,0.5,30)$ for the demand sequence, and %generated a demand sequence as described above. We then 
	generated 1,000 different predictions of this demand sequence, each from a set of ``predicted'' parameters $(\hat{\alpha},\hat{\beta},\hat{\gamma}, \hat{L})$ where each $\hat{\alpha},\hat{\beta},\hat{\gamma}$ was sampled uniformly at random from $[0.2,0.8]$ and $\hat{L}$ from $\{10, 20,  30\}$. %, and then generated a predicted demand sequence using the same procedure with the set of predicted parameters. 
	Thus the variation parameter $v$ was fixed, and the accuracy parameter $a$ varied across instances. %Because the demand sequence was the same for all predictions, the variation parameter $v$ was fixed.
	
	\item \textbf{Fixed $a$:} We generated 1,000 demand sequences by selecting the parameters $(\alpha,\beta,\gamma, L)$ uniformly at random where each $\alpha,\beta,\gamma$ was sampled uniformly at random from $[0.2,0.8]$ and $L$ from $\{10, 20,  30\}$.
	% , along with a prediction of the demand sequence. For each demand sequence, we generated a set of parameters $\alpha,\beta,\gamma, L$ uniformly at random and generated a demand sequence as described above. 
	We then generated predictions by changing each parameter 10\% (e.g., $\alpha$ becomes either $1.1\alpha$ or $0.9\alpha$) and using the corresponding sequence. %used the new set of parameters to generate a predicted demand sequence using the same procedure. 
	%Since the set of parameters was generated uniformly at random, 
	Thus the variation parameter $v$ varied across instances, but % Because the estimation error of the parameters was (roughly) the same for all predictions, 
	the accuracy parameter $a$ was (roughly) fixed.
\end{itemize}

\section{Additional Experimental Results}
In our experiments on real data, we used four popular forecasting method to generate predictions for each instance: Exponential Smoothing (Holt Winters), ARIMA, Prophet, and LightGBM. Experiments on the datasets yielded the histograms in \cref{fig:gap}. To show the performance of \texttt{PERP} is robust to the forecasting method, in \cref{fig:J} we further divide the three histograms in \cref{fig:gap} into twelve histograms separated by the four forecasting methods.

\begin{figure}[h]
	\centering
	\begin{subfigure}[t]{\textwidth}
		\centering
		\includegraphics[width=.3\linewidth]{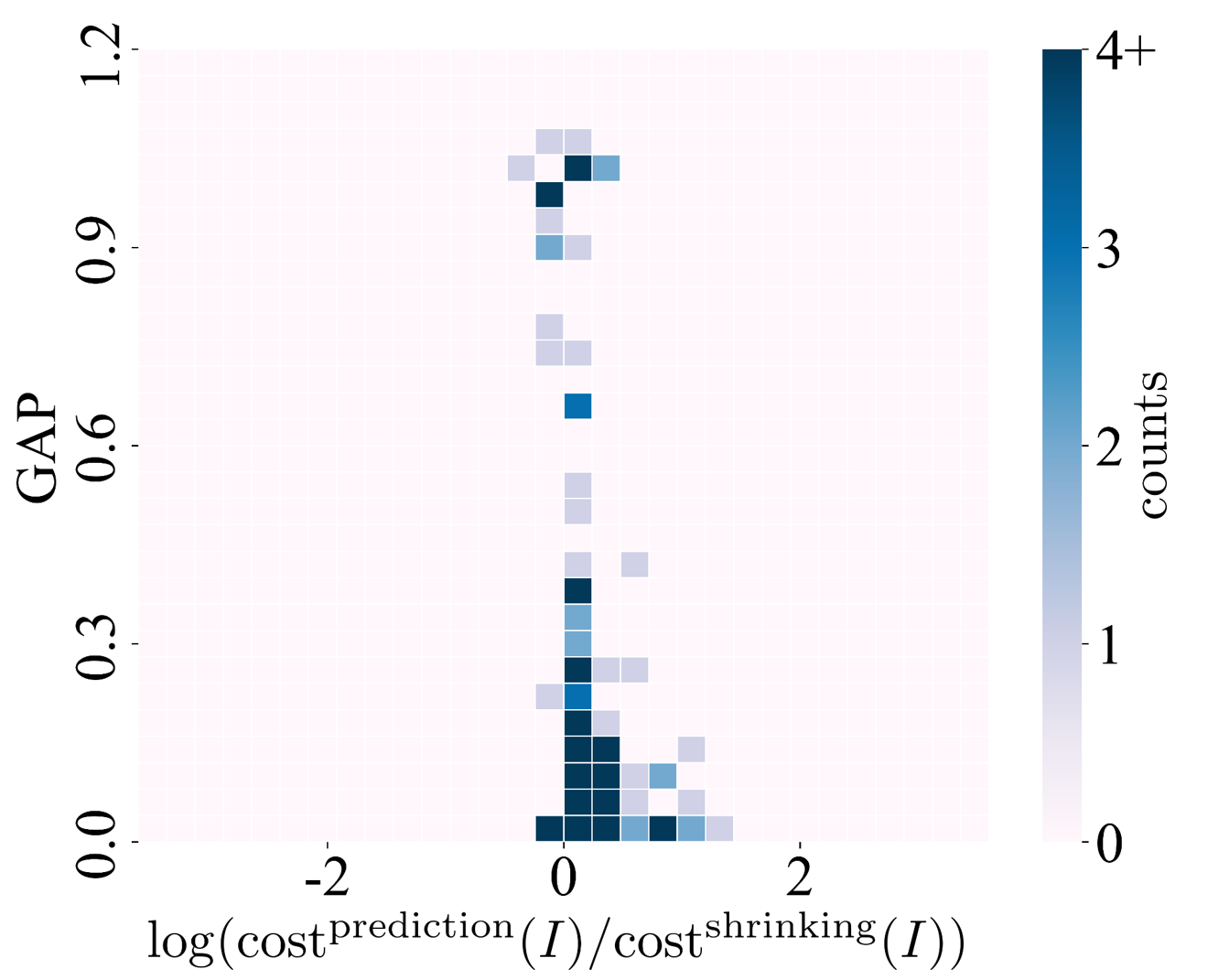}
		\hfill
		\includegraphics[width=.3\linewidth]{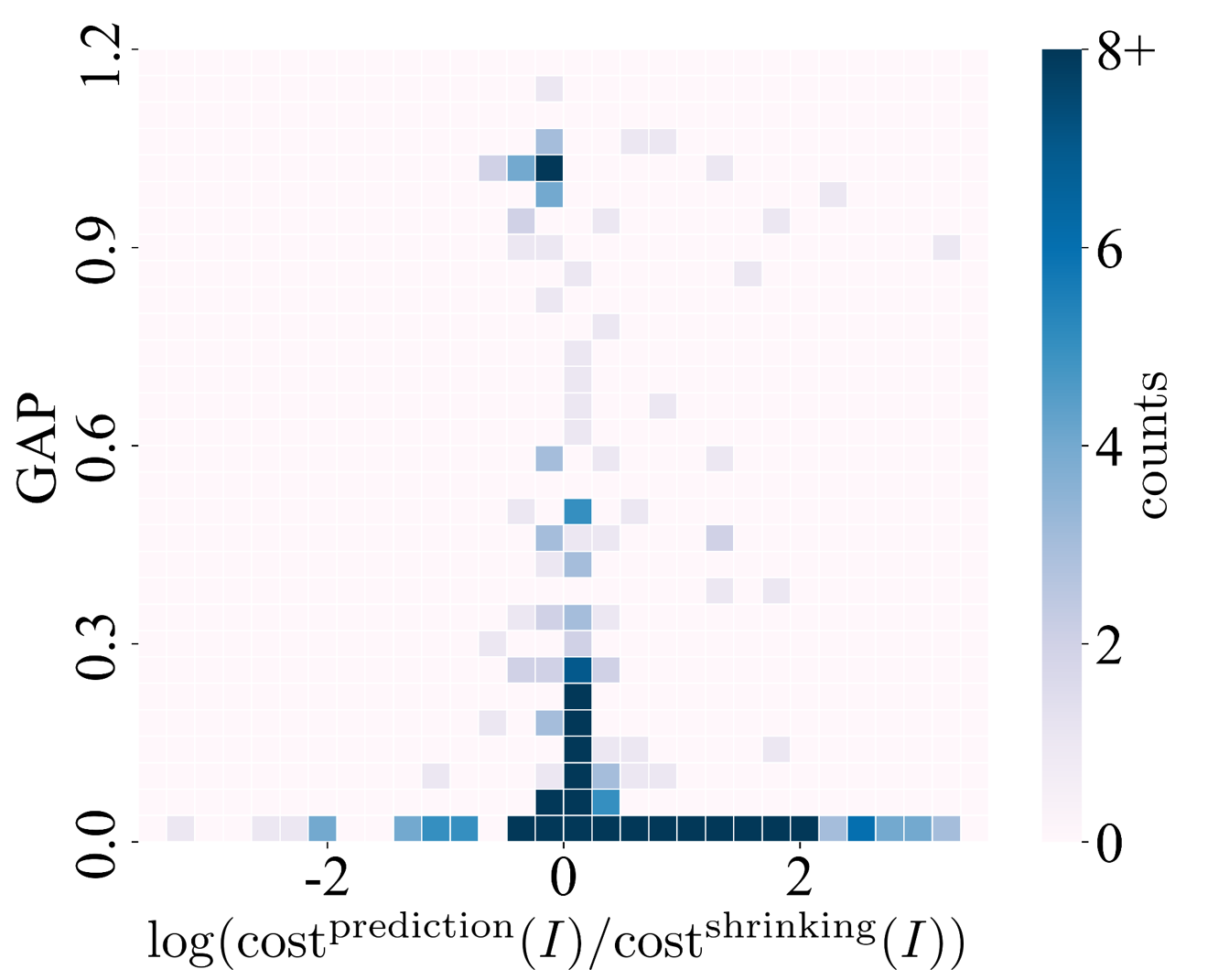}
		\hfill
		\includegraphics[width=.3\linewidth]{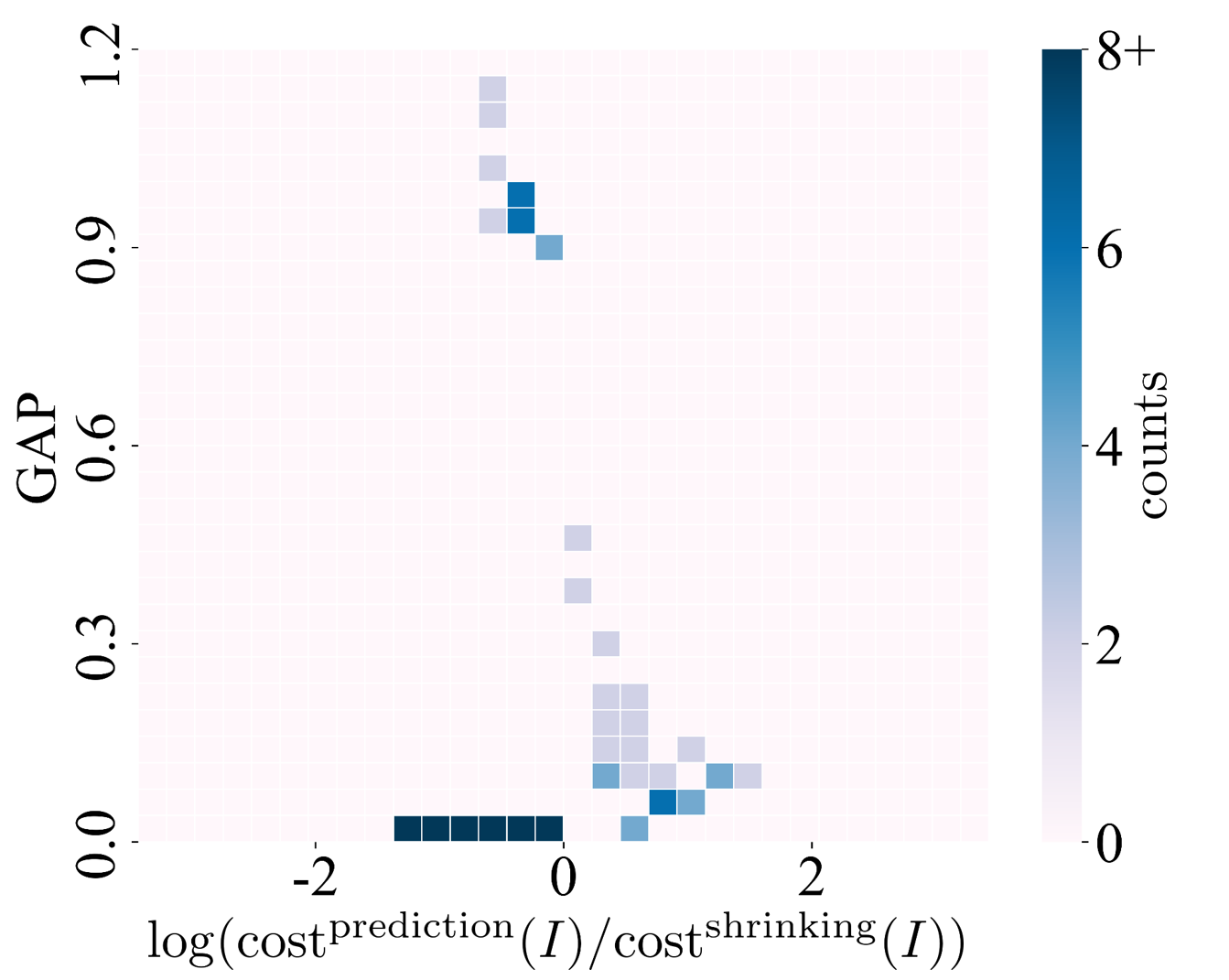}
		\caption{GAPs with Exponential Smoothing forecasts. Left to right: Rossmann, Wikipedia, Restaurant.\vspace{1em}\quad}
	\end{subfigure}
	\begin{subfigure}[t]{\textwidth}
		\centering
		\includegraphics[width=.3\linewidth]{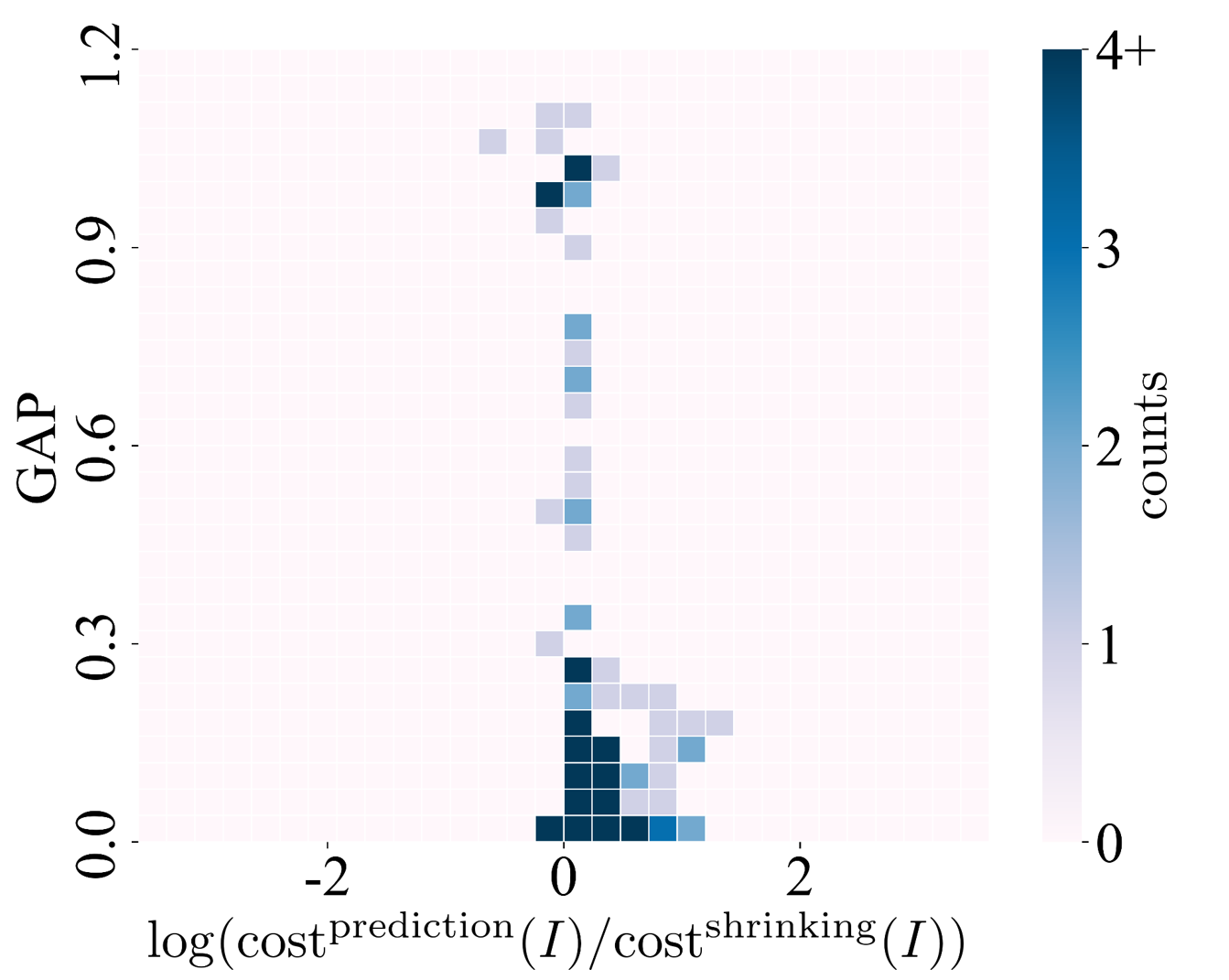}
		\hfill
		\includegraphics[width=.3\linewidth]{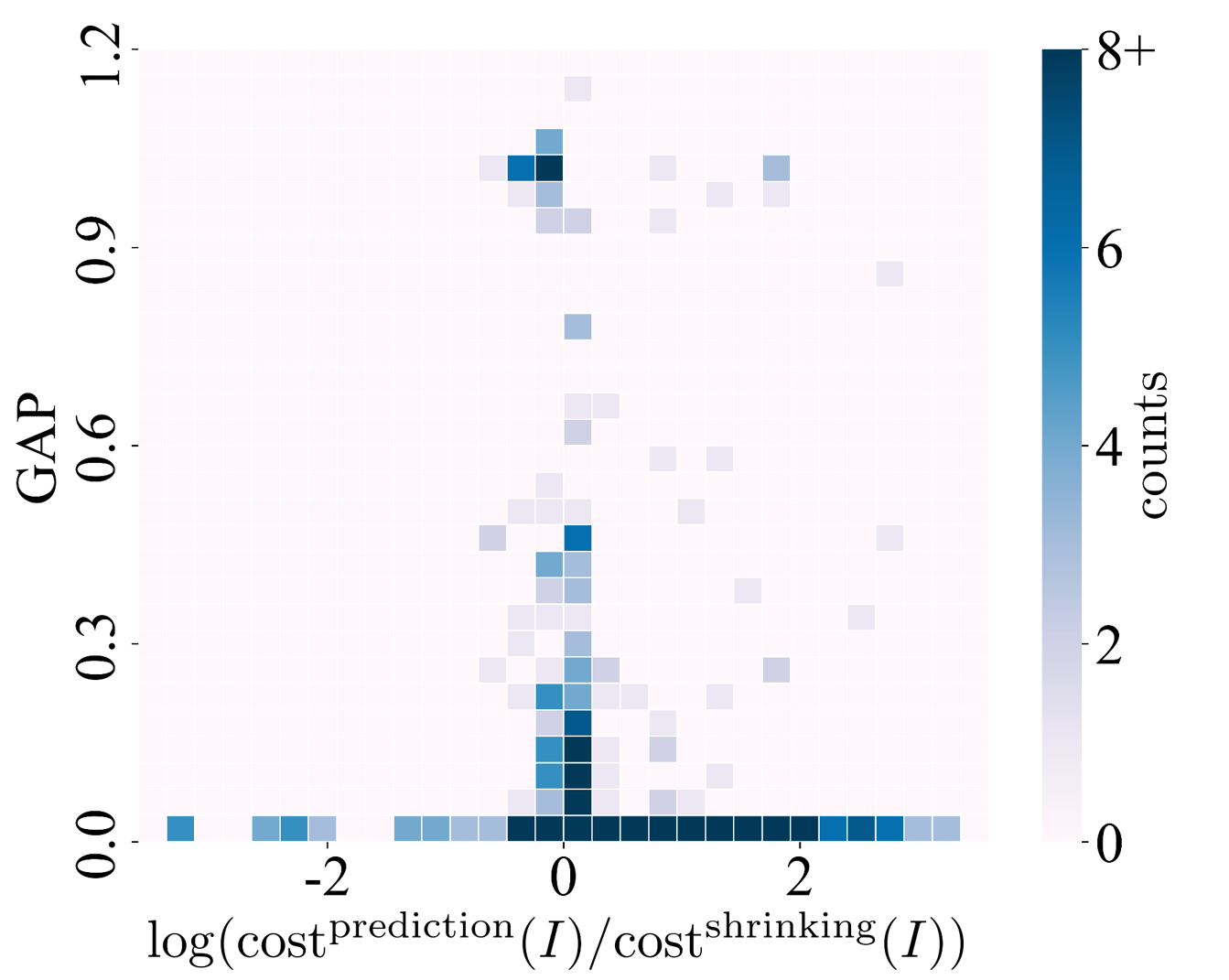}
		\hfill
		\includegraphics[width=.3\linewidth]{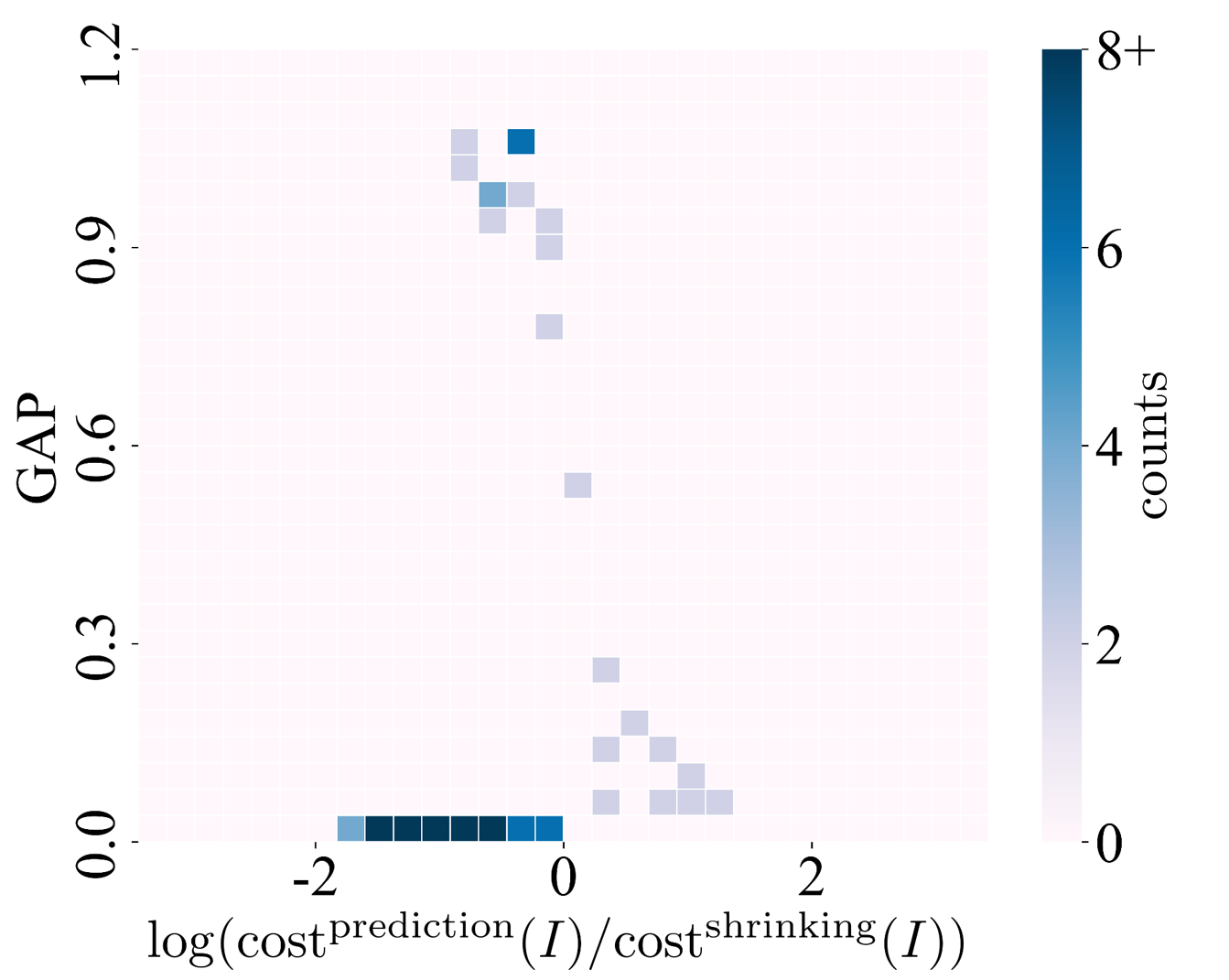}
		\caption{GAPs with ARIMA forecasts. Left to right: Rossmann, Wikipedia, Restaurant.\vspace{1em}\quad}
	\end{subfigure}
	\begin{subfigure}[t]{\textwidth}
		\centering
		\includegraphics[width=.3\linewidth]{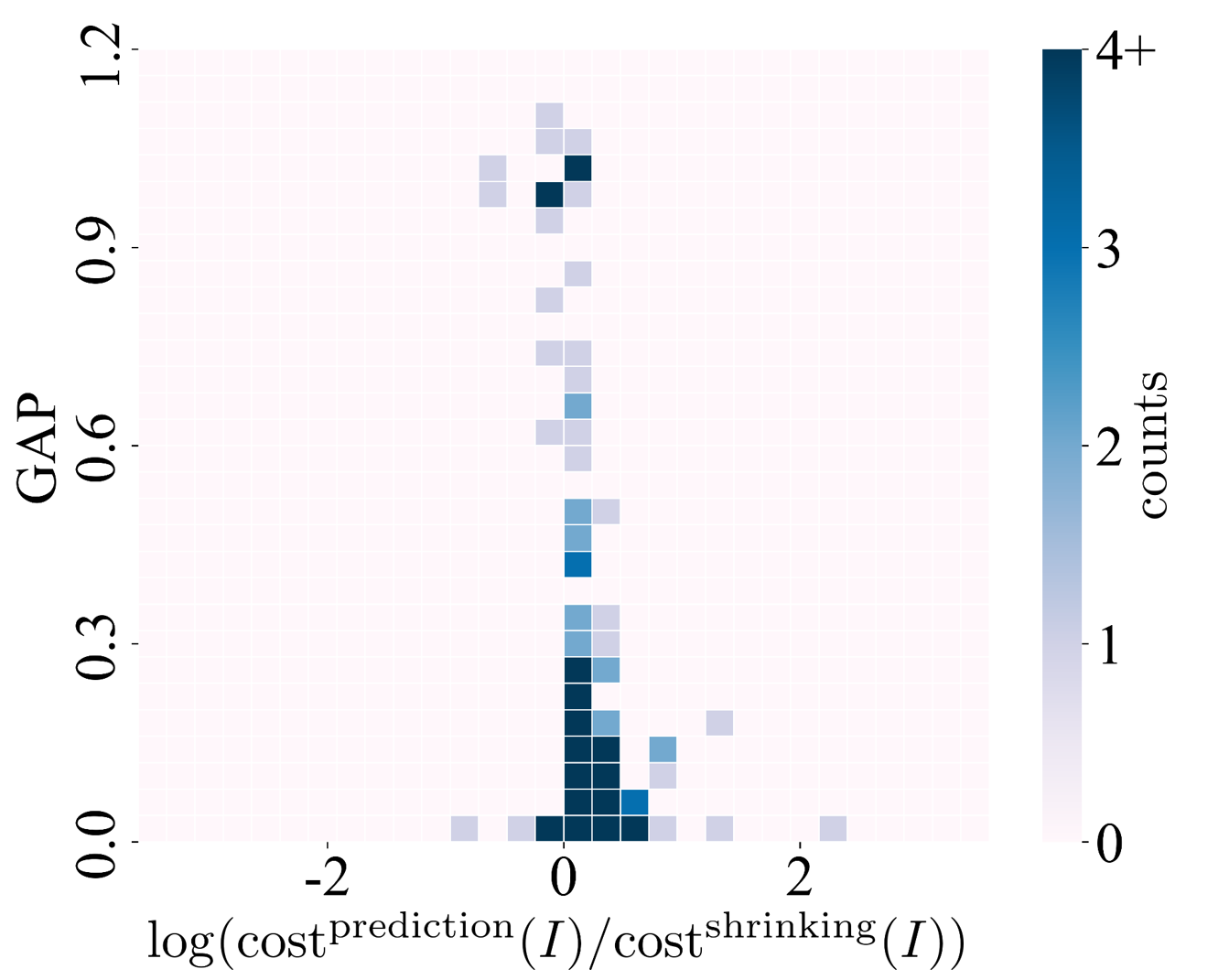}
		\hfill
		\includegraphics[width=.3\linewidth]{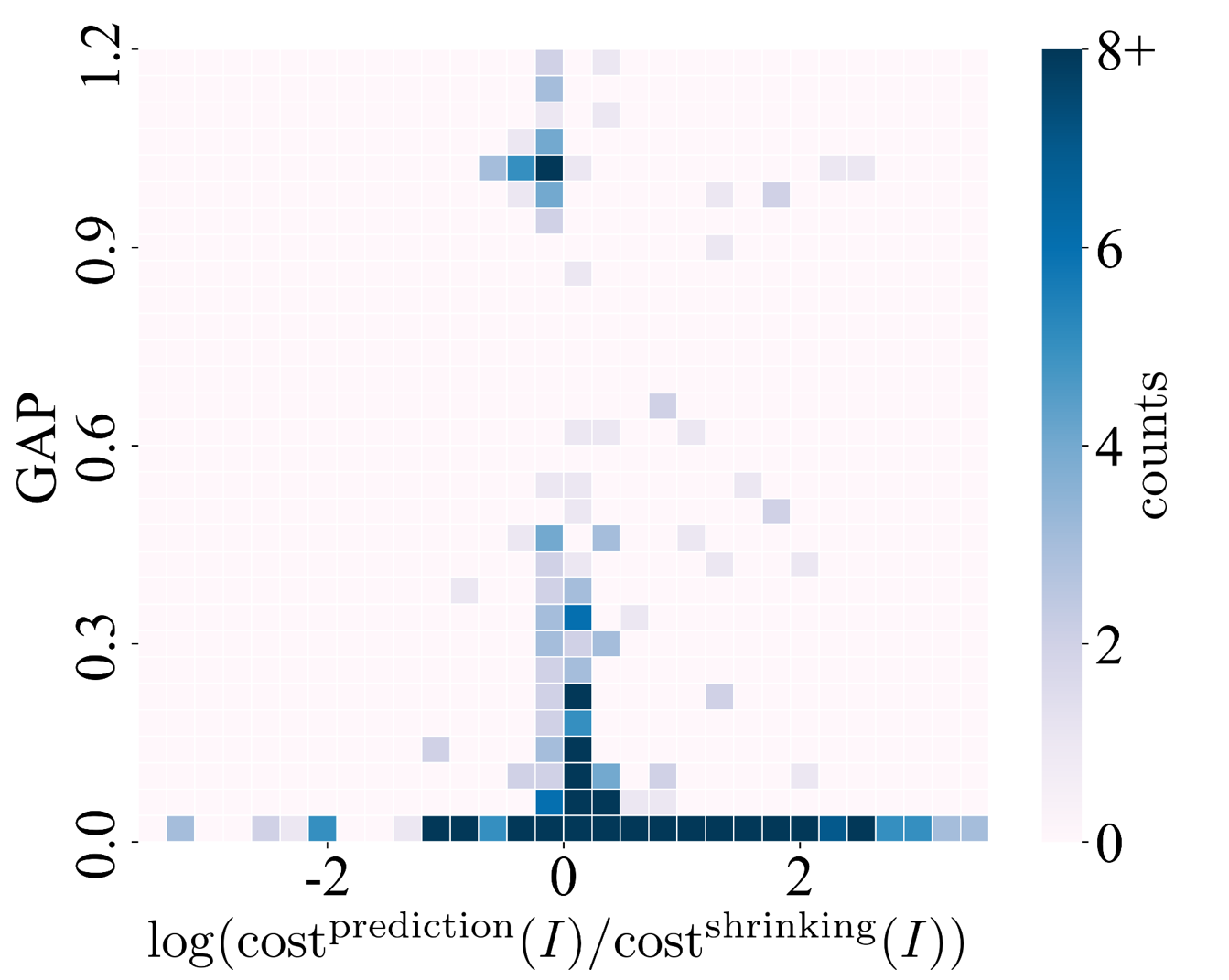}
		\hfill
		\includegraphics[width=.3\linewidth]{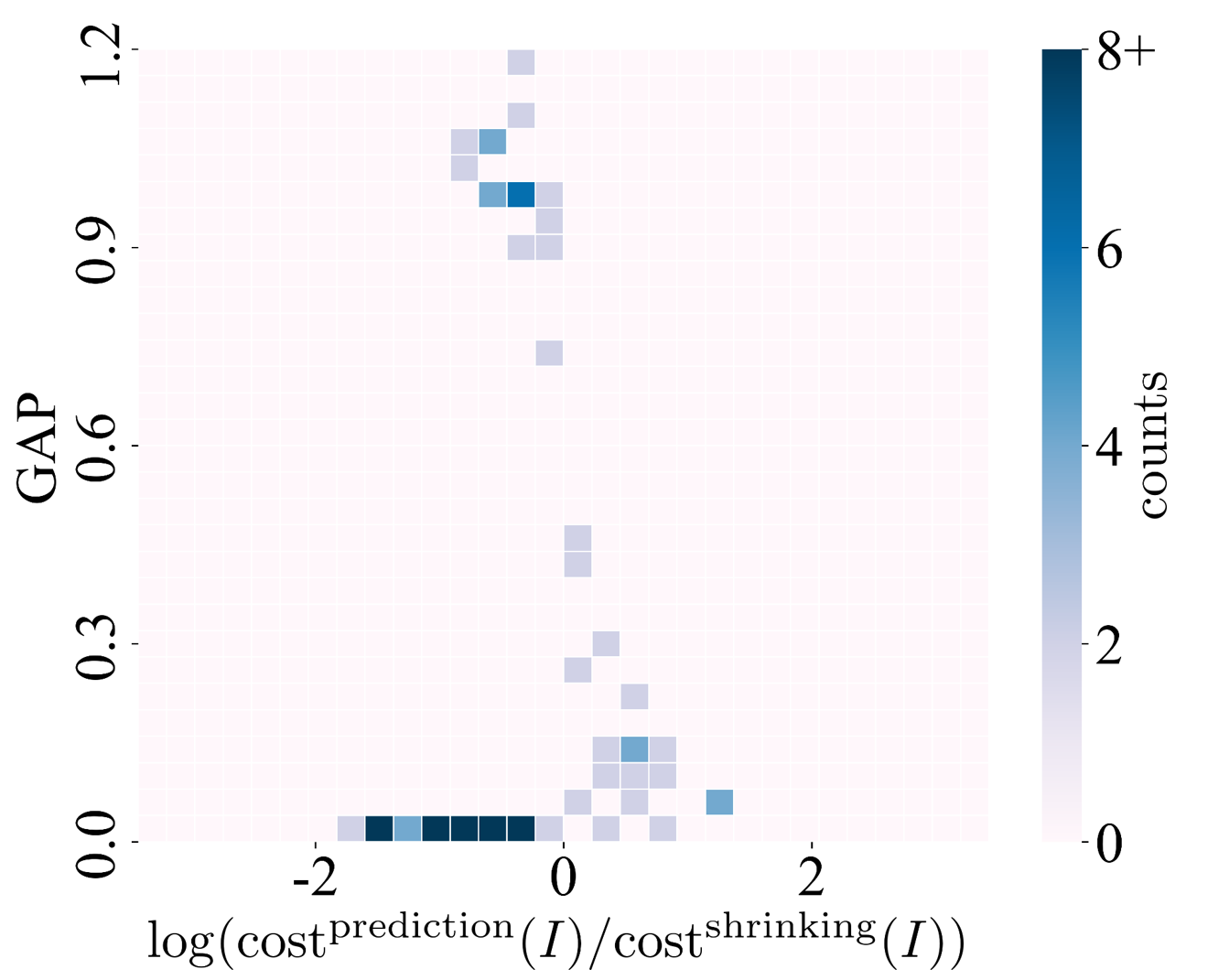}
		\caption{GAPs with Prophet forecasts. Left to right: Rossmann, Wikipedia, Restaurant.\vspace{1em}\quad}
	\end{subfigure}
	\begin{subfigure}[t]{\textwidth}
		\centering
		\includegraphics[width=.3\linewidth]{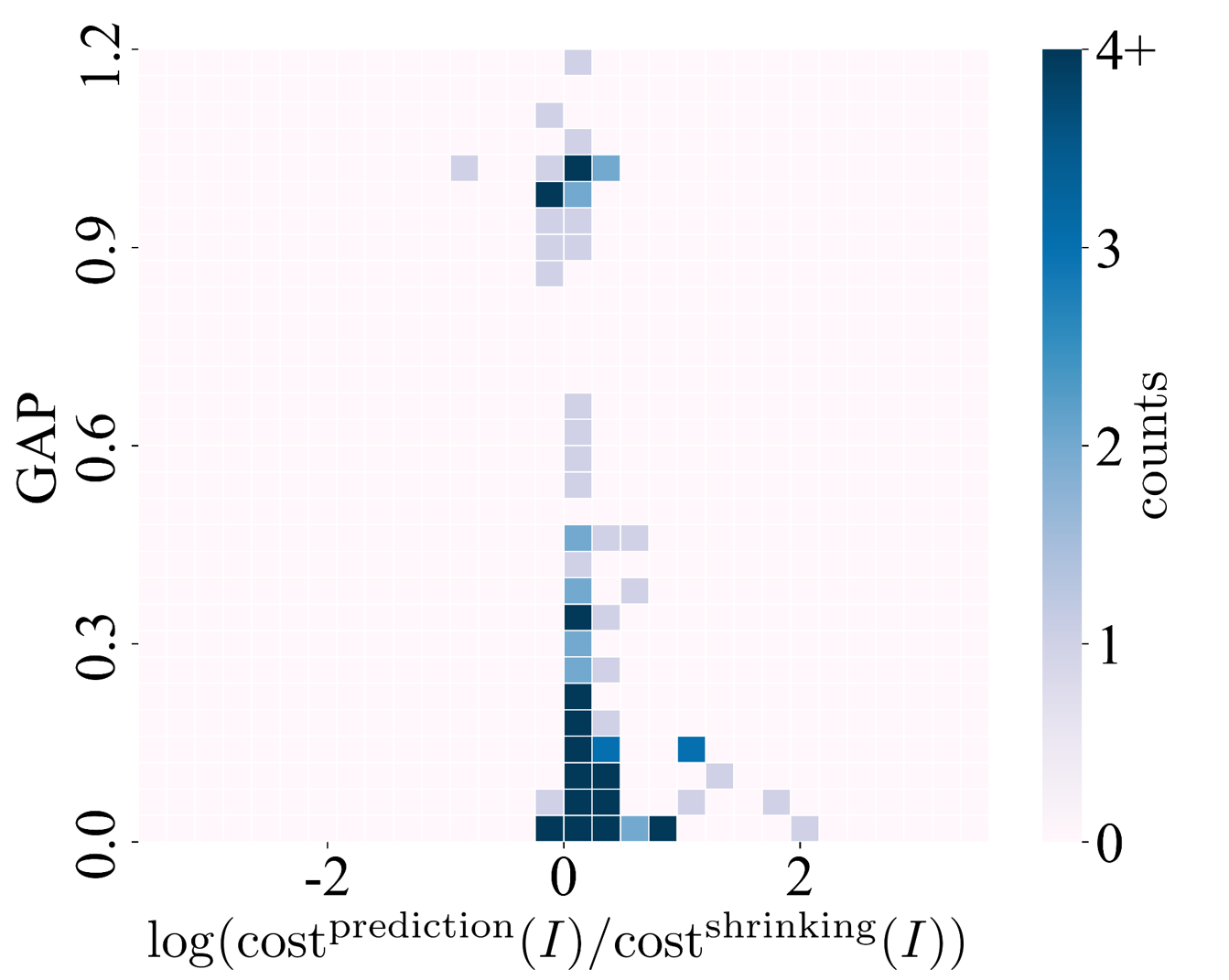}
		\hfill
		\includegraphics[width=.3\linewidth]{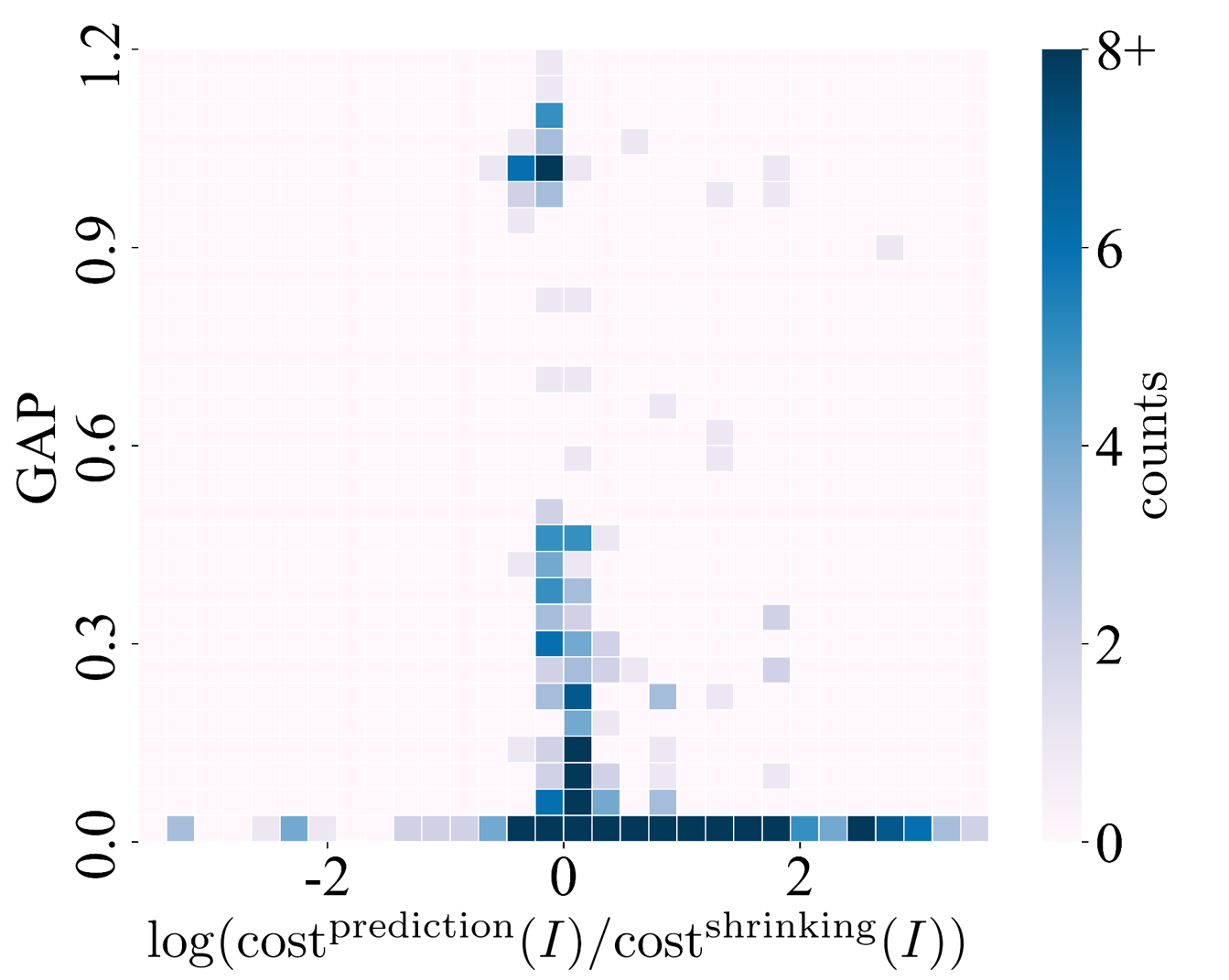}
		\hfill
		\includegraphics[width=.3\linewidth]{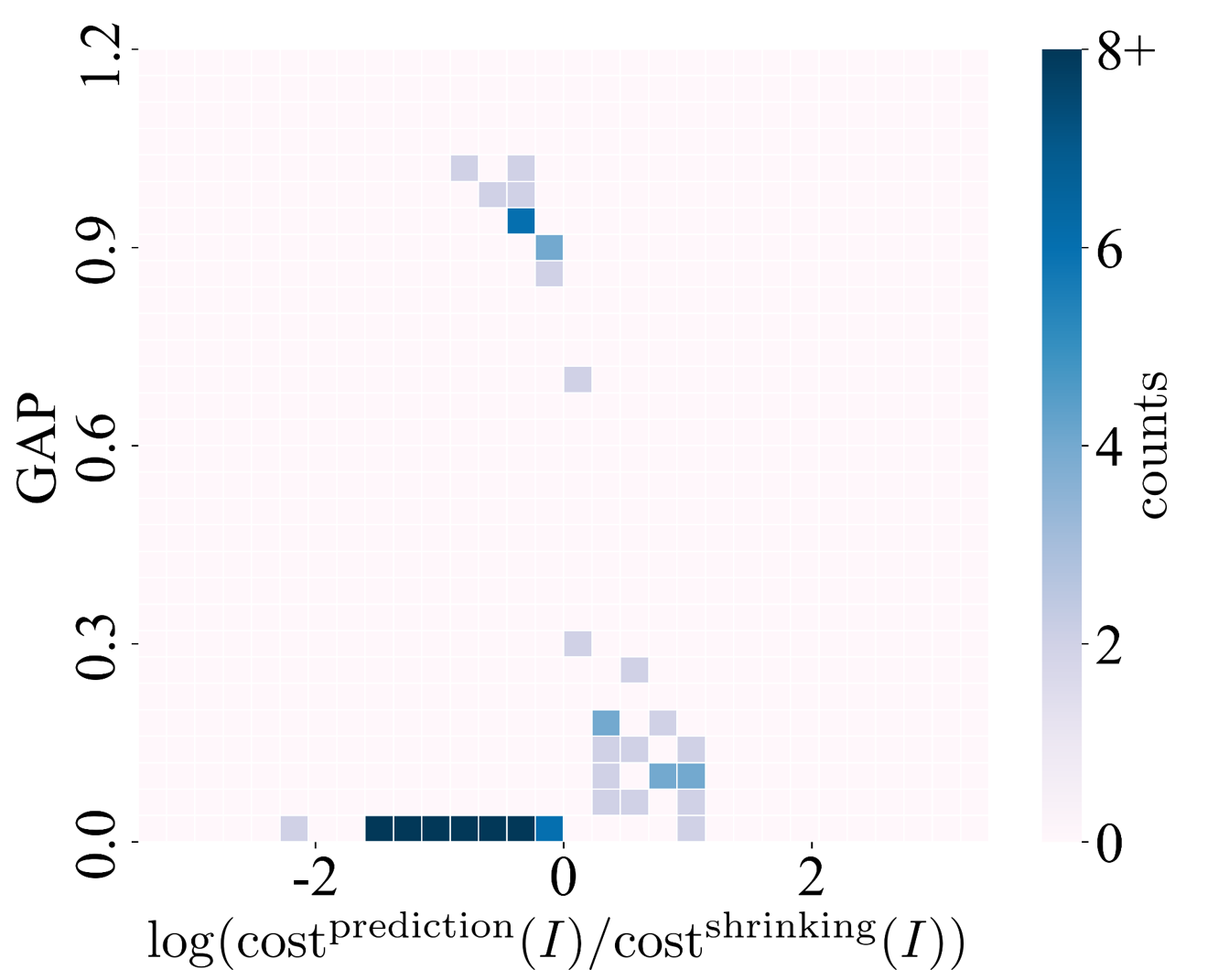}
		\caption{GAPs with ELightGBM forecasts. Left to right: Rossmann, Wikipedia, Restaurant.\vspace{1em}\quad}
	\end{subfigure}
	
	\caption{Histograms of GAPs divided by forecasting methods across the Rossmann dataset, the Wikipedia dataset, and the Restaurant dataset.}
	\label{fig:J}
\end{figure}

\end{document}